\documentclass[english]{article}
\usepackage[T1]{fontenc}
\usepackage[latin9]{inputenc}
\usepackage{amsmath}
\usepackage{amsthm}
\usepackage{amssymb}
\usepackage{xargs}[2008/03/08]
\usepackage[all]{xy}

\makeatletter
\theoremstyle{plain}
\newtheorem{thm}{\protect\theoremname}[section]
\theoremstyle{plain}
\newtheorem{lem}[thm]{\protect\lemmaname}
\theoremstyle{plain}
\newtheorem{prop}[thm]{\protect\propositionname}
\theoremstyle{definition}
\newtheorem{defn}[thm]{\protect\definitionname}
\theoremstyle{remark}
\newtheorem{rem}[thm]{\protect\remarkname}
\theoremstyle{definition}
\newtheorem{example}[thm]{\protect\examplename}
\theoremstyle{plain}
\newtheorem{cor}[thm]{\protect\corollaryname}
\theoremstyle{definition}

\@ifundefined{date}{}{\date{}}

\usepackage{pdflscape}

\usepackage{tikz}

\usepackage{pgfplots}

\usetikzlibrary{arrows}

\usetikzlibrary{automata}

\usetikzlibrary{chains}

\usetikzlibrary{positioning}

\usetikzlibrary{backgrounds}

\usetikzlibrary{calc,decorations.pathreplacing}

\pgfplotsset{compat=1.9}

\g@addto@macro\@floatboxreset{\centering}

\usepackage{enumitem}
\setenumerate{ label= (\alph*)}
\setenumerate[2]{ label= (\alph{enumi}\arabic*)} 

\newcommand{\X}{\mathcal{X}}
\newcommand{\Y}{\mathcal{Y}}

\newcommand{\U}{\mathcal{U}}

\newcommand{\W}{\mathcal{W}}
\newcommand{\D}{\mathcal{D}}

\newcommand{\Ebb}{\mathbb{E}}

\newcommand{\Gen}{\mathrm{Gen}}

\newcommand{\Mat}{\mathrm{Mat}} 
\newcommand{\Free}{\mathrm{Free}} 
\newcommand{\coFree}{\mathrm{coFree}} 

\usepackage{babel}
\providecommand{\definitionname}{Definition}
\providecommand{\examplename}{Example}
\providecommand{\lemmaname}{Lemma}
\providecommand{\problemname}{Problem}
\providecommand{\propositionname}{Proposition}
\providecommand{\remarkname}{Remark}
\providecommand{\corollaryname}{Corollary}
\providecommand{\theoremname}{Theorem}

\newtheorem*{thmA}{Theorem A}

\makeatother

\usepackage{babel}
\providecommand{\corollaryname}{Corollary}
\providecommand{\definitionname}{Definition}
\providecommand{\examplename}{Example}
\providecommand{\lemmaname}{Lemma}
\providecommand{\problemname}{Problem}
\providecommand{\propositionname}{Proposition}
\providecommand{\remarkname}{Remark}
\providecommand{\theoremname}{Theorem}

\begin{document}
\title{Recollements, coproducts and products in extriangulated categories}
\author{Alejandro Argud{\'i}n-Monroy, Octavio Mendoza, Carlos E. Parra \thanks{The first named author was supported by a postdoctoral fellowship
EPM(1) 2024 from SECIHTI. The first and second named authors were
supported by the Project PAPIIT-IN100124 Universidad Nacional Aut{\'o}noma
de M{\'e}xico. The third named author was supported by ANID+FONDECYT/REGULAR+1240253.}}

\maketitle
\global\long\def\Mod{\operatorname{Mod}}%
\global\long\def\Ker{\operatorname{Ker}}%
\global\long\def\Proj{\operatorname{Proj}}%
\global\long\def\gldim{\operatorname{gl.dim.}}%
\global\long\def\Inj{\operatorname{Inj}}%
\global\long\def\Gen{\operatorname{Gen}}%
\global\long\def\Fun{\operatorname{Fun}}%
\global\long\def\Ab{\operatorname{Ab}}%
\global\long\def\Hom{\operatorname{Hom}}%
\global\long\def\Ext{\operatorname{Ext}}%
\global\long\def\im{\operatorname{Im}}%
\global\long\def\coim{\operatorname{CoIm}}%
\global\long\def\Cok{\operatorname{Coker}}%
\global\long\def\smd{\operatorname{smd}}%
\global\long\def\C{\mathcal{C}}%
\global\long\def\s{\mathfrak{s}}%
\global\long\def\E{\mathbb{E}}%
\global\long\def\t{\mathfrak{t}}%
\global\long\def\u{\mathbf{u}}%
\global\long\def\te{\mathbf{t}}%
\global\long\def\v{\mathbf{v}}%
\global\long\def\x{\mathbf{x}}%
\global\long\def\pd{\mathsf{pd}}%
\newcommandx\suc[5][usedefault, addprefix=\global, 1=N, 2=M, 3=K, 4=, 5=]{#1\overset{#4}{\rightarrow}#2\overset{#5}{\rightarrow}#3}%
\global\long\def\stors{\mathsf{stors}}%
\newcommandx\suca[5][usedefault, addprefix=\global, 1=A, 2=B, 3=C, 4=f, 5=g]{#1\stackrel{#4}{\hookrightarrow}#2\stackrel{#5}{\twoheadrightarrow}#3}%

\begin{abstract}
We introduce a notion similar to the AB4 (resp. AB4{*}) condition
for abelian categories but in the context of extriangulated categories.
We will refer to this notion as AET4 (resp. AET4{*}). One of our main results
shows equivalent statements for AET4 (resp. AET4{*}), which generalize
statements commonly used in homological constructions in abelian categories.
As an application, we will give conditions for a recollement $(\mathcal{A},\mathcal{B},\mathcal{C})$ of extriangulated
categories  with $\mathcal{B}$
AET4 (resp. AET4{*}) to imply that the categories $\mathcal{A}$ and
$\mathcal{C}$ are AET4 (resp. AET4{*}); and we will show a relation
between the $n$-smashing (resp. $n$-co-smashing) condition for
a $t$-structure and the AET4 (resp. AET4{*}) condition of the extended
hearts of the $t$-structure. It is also included an appendix where we study in detail the properties of adjoint pairs between extriangulated categories which are necessary for the development of the paper, including some special properties for higher extension groups.
\end{abstract}
\tableofcontents{}

\section{Introduction}

Let $\mathcal{A}$ be an abelian category. If for every family of objects
$\{A_{i}\}_{i\in I}$ in $\mathcal{A}$ there exists a coproduct $\coprod_{i\in I}A_{i}$
in $\mathcal{A}$, then the category $\mathcal{A}$ is said to be
AB3. It is important to note that  the coproduct
is not (in general) an exact functor. There are examples of abelian categories having a family 
$\{\suca[A_{i}][B_{i}][C_{i}][f_{i}][g_{i}]\}_{i\in I}$ of exact sequences such that  the sequence 
$
\suc[\coprod_{i\in I}A_{i}][\coprod_{i\in I}B_{i}][\coprod_{i\in I}C_{i}][\coprod_{i\in I}f_{i}][\coprod_{i\in I}g_{i}]
$
 is not a short exact sequence (specifically, the morphism $\coprod_{i\in I}f_{i}$
is not  a monomorphism). For this reason, the AB4 condition was
introduced. Namely, $\mathcal{A}$ is AB4 if it is AB3 and the coproducts
are exact functors. It should be noted that this condition is fundamental
for performing homological constructions in co-complete abelian categories
(see \cite{AM,HJ,AP}).

In contrast, if $\mathcal{D}$ is a triangulated category with coproducts,
it is well known that the coproduct of distinguished triangles is
a distinguished triangle (see \cite[Prop.1.1.6]{N}). This comparison
is valid thanks to the notion of extriangulated categories from \cite{NP}
which encompass both abelian and triangulated categories.
Among the most common examples of extriangulated categories are exact
categories, localizations of exact categories, and any full closed-under-extensions
subcategory of an extriangulated category. Specifically, an extriangulated
category is a triple $(\mathcal{C},\E,\s)$ where: $\mathcal{C}$
is an additive category, $\E:\mathcal{C}^{op}\times\mathcal{C}\rightarrow\Ab$
is an additive functor (which plays a role similar to that of $\Ext^{1}(-,?)$
in abelian categories or $\Hom(-,\Sigma?)$ in triangulated categories)
and $\s$ is a correspondence that assigns to each element $\eta\in\mathbb{E}(A,B)$
an equivalence class of a sequence of morphisms of the form $\suc[B][E][A][f][g]$ which is known as the realization of $\eta.$

The objective of this paper is to introduce a condition similar to
AB4 for abelian categories but in the context of extriangulated categories.
To motivate our definition, we recall that (see the proof
of \cite[Lem.2.12]{AP} and \cite[Thm.4.8]{A})  the AB4 condition holds
in an AB3 abelian category $\mathcal{A}$ if, and only if, for any set $I\neq\emptyset$   there is
a (functorial) map 
$
\Gamma:\prod_{i\in I}\Ext_{\mathcal{A}}^{1}(B_{i},A)\rightarrow\Ext_{\mathcal{A}}^{1}(\coprod_{i\in I}B_{i},A^{(I)})
$
such that $\Gamma(\eta_{i})_{i\in I}$ is the extension realized by
the short exact sequence 
\[
\suca[A^{(I)}][\coprod_{i\in I}E_{i}][\coprod_{i\in I}B_{i}][\coprod_{i\in I}f_{i}][\coprod_{i\in I}g_{i}],
\]
where $\eta_{i}:\:\suca[A][E_{i}][B_{i}][f_{i}][g_{i}]$ for each $i.$  Similarly, in the case of extriangulated categories we propose the 
following notion.

\begin{defn}\label{ET4Def}
An extriangulated category $(\mathcal{C},\E,\s)$ is AET4
if for any set $I\neq\emptyset,$ $\forall\,\{B_i\}_{i\in I}$ in $\C$ and $A\in\C$  there is a natural transformation 
\[
\Gamma:\prod_{i\in I}\mathbb{E}(B_{i},A)\rightarrow\mathbb{E}(\coprod_{i\in I}B_{i},A^{(I)})
\]
satisfying the following two conditions.
\begin{enumerate}
\item (\textbf{Coproducts are extriangulated}). If $\eta_{i}$ is realized
by $\suc[A][E_{i}][B_{i}][f_{i}][g_{i}]$ for all $i\in I$, then
a realization of $\Gamma(\eta_{i})_{i\in I}$ is given by the sequence
\[
\suc[A^{(I)}][\coprod_{i\in I}E_{i}][\coprod_{i\in I}B_{i}][\coprod_{i\in I}f_{i}][\coprod_{i\in I}g_{i}].
\]
\item (\textbf{Coproduct inclusions are extriangulated}). For each $i\in I$ and the canonical
inclusions $\mu_{i}^{B}:B_{i}\rightarrow\coprod_{i\in I}B_{i}$ and
$\mu_{i}^{A}:A\rightarrow A^{(I)}$, we have that 
\[
\mathbb{E}\left(\mu_{i}^{B},A\right)\left(\Gamma(\eta_{i})_{i\in I}\right)=\mathbb{E}\Big(\coprod_{i\in I}B_{i},\mu_{i}^{A}\Big)(\eta_{i})\text{ }\forall i\in I.
\]
\end{enumerate}
\end{defn}

Once we have introduced the AET4 property, we will seek to give different
characterizations of it. To do this, we will rely on some known characterizations
of the AB4 condition in abelian categories such as in \cite{AP,A}.
In particular, in Definition \ref{def:univ}, we will introduce the
notion of universal $\mathbb{E}$-extensions based on \cite{AMP,AP}.
It is worth mentioning that to achieve our characterization, we will
need to use extriangulated categories with negative extensions 
(see \cite{G,AET1}). Furthermore, we will require the condition that the
category is \emph{coproduct-compatible} (see Definition \ref{def:cc}).
We will show in our examples that there is a wide range of categories
satisfying this condition, including exact categories, triangulated
categories and extended hearts of $t$-structures  (see \cite{AMP}).
One of the main results in this paper is the following one (see the dual result in Theorem \ref{thm:mainAop}).

\begin{thmA}\label{thm:main}Let $(\mathcal{C},\mathbb{E},\s)$ be
an extriangulated category with coproducts. Consider the following
statements. 
\begin{enumerate}
\item $\mathcal{C}$ is AET4. 
\item For any set $I\neq\emptyset,$  the map $\tau:\mathbb{E}(\coprod_{i\in I}A_{i},B)\rightarrow\prod_{i\in I}\mathbb{E}(A_{i},B)$,
$\epsilon\mapsto\left(\epsilon\cdot\mu_{i}^{A}\right)_{i\in I}$,
defines a natural isomorphism. 
\item For any objects $V,D$ in $\mathcal{C}$ there is
a universal $\mathbb{E}$-extension of $V$ by $D$. That is, for $V,D\in\C,$ there is an 
$\eta\in\mathbb{E}(V^{(X)},D)$, for some set $X\neq\emptyset$, such
that the map $\text{Hom}_{\mathcal{C}}(V,V^{(X)})\to\E(V,D),\;f\mapsto \eta\cdot f,$
is surjective. 
\end{enumerate}
Then, the following implications hold true: $(a)\Rightarrow(b)$ and $(b)\Leftrightarrow(c)$.
Moreover, if $\mathcal{C}$ is coproduct-compatible, then the three statements above are
equivalent. 
\end{thmA}

The second main result in the paper is the following one (see Theorems \ref{MainTB} and \ref{MainTB2} for details).
\vspace{0.2cm}

{\bf Theorem B.} Let $(\mathcal{C},\E,\s)$ be an essentially small
AET4 (resp. AET4{*})  extriangulated category. Then, for any $n\geq 1,$ we have the natural isomorphism
\begin{alignat*}{1}
\mathbb{E}^{n}\Big(\coprod_{i\in I}D_{i},C\Big) & \stackrel{}{\simeq}\prod_{i\in I}\mathbb{E}^{n}(D_{i},C).\\
\Big(\text{resp. }\; \mathbb{E}^{n}\Big(C,\prod_{i\in I}D_{i}\Big) & \stackrel{}{\simeq}\prod_{i\in I}\mathbb{E}^{n}(C,D_{i})\Big)
\end{alignat*}

Having proved Theorem A, we will use it to study the condition AET4 in different settings. In the case of exact categories we generalize (see Corollary \ref{equvExactAET}) the description given in \cite{A} for AB4 abelian categories. In the description of the AET4 condition for extended hearts, we will also look at some results
for hearts of intervals of s-torsion pairs  in extriangulated categories
with negative extensions (see \cite{AET1}). Our main result in this
direction tells us that: for a $(n-1)$-smashing $t$-structure $\te=(\mathcal{X},\mathcal{Y})$
in a triangulated category with coproducts, the extended heart (of
length $n$) of $\te$ is AET4 if, and only if, $\te$ is smashing. Furthermore,
we will see an example of an extended heart of length 2 in a triangulated
category with coproducts that is not AET4. This example shows (in a non-abelian context) that
a full subcategory of an AET4 extriangulated category is not necessarily
AET4 and, moreover, that there can be AET4 extriangulated subcategories
$\mathcal{H}$ and $\mathcal{H}'$ of a triangulated category with
coproducts $\mathcal{D}$ such that the category $\mathcal{H}\star\mathcal{H}'$
is not AET4. 

Lastly, we will seek to show that, in the case of having a recollement
of extriangulated categories $(\mathcal{A},\mathcal{B},\mathcal{C})$,
the AET4 condition in $\mathcal{B}$ is inherited in $\mathcal{A}$
and $\mathcal{C}$ under certain conditions. It is worth mentioning
that the definition of recollement of extriangulated categories, that
we introduce in Definition \ref{def:rec}, is different from others
that can be found in the literature, but it coincides with the usual
recollement for abelian and triangulated
categories (see Remark \ref{rem:natrec}(a)). The difference is that one of our recollement
conditions is stated in terms of a torsion-torsion-free triple which
is better suited to our methods. Our main results for a recollement $(\mathcal{A},\mathcal{B},\mathcal{C})$ of extriangulated
categories are Theorem \ref{M1RETC} and Theorem \ref{M2RETC}.

It should be noted that all our results are dualizable. In particular,
Theorem A can be dualized for characterizing the AET4{*} condition
which is related to having extriangulated products, see Theorem \ref{thm:mainAop}. 

The structure of the article is as follows. In section 2, we will
look at the preliminaries necessary for the presentation of our results.
Specifically, we describe the matrix expression of the elements
of an additive bifunctor, we also recall some essential results on recollements
of additive categories and  facts and concepts about the
theory of extriangulated categories. In the remaining sections, we
will present the results described above. Specifically, in Section
3 we introduce condition AET4 and prove Theorem A.  Section 4
contains our results on condition AET4 in extended hearts of t-structures.
In Section 5, we will present our results on recollements. 

Finally,
the article includes an Appendix where we study in detail the properties
of adjoint pairs between extriangulated categories. Specifically,
without assuming having enough projective or injective objects, we
will seek conditions for an adjoint pair $(S,T)$ to induce a natural
isomorphism $\mathbb{E}_{\mathcal{A}}(SX,Y)\rightarrow\mathbb{E}_{\mathcal{B}}(X,TY).$ It is also considered the case of higher extension groups and adjoint pair of functors which are needed in the proof of Theorem B.

\subsection*{Acknowledgements}

Part of this work was developed in a research stay of the first named
author at the the Instituto de Ciencias F{\'i}sicas y Matem{\'a}ticas of the
Universidad Austral de Chile in 2025. The first named author would
like to thank the academic and administrative staff of this institution
for their warm hospitality and support. 

\section{Preliminaries}

\subsection{Functors, coproducts and matrices}

Let $\C$ be a category and $\mathcal{C}^{op}$ be its opposite category.
 Given a functor $F:\mathcal{C}\rightarrow\mathcal{D}$,
we have its opposite functor $F^{op}:\mathcal{C}^{op}\rightarrow\mathcal{D}^{op}.$
Note that a natural transformation $\gamma:F\rightarrow G$ defines
a natural transformation $\gamma^{op}:G^{op}\rightarrow F^{op}$.

Given two natural transformations $a:F_{1}\rightarrow F_{2}$ and
$b:G_{1}\rightarrow G_{2}$, with $G_{1},G_{2}:\mathcal{Y}\rightarrow\mathcal{Z}$
and $F_{1},F_{2}:\mathcal{X}\rightarrow\mathcal{Y}$, the Godement
product $b\cdot a$ of $b$ by $a$ is the natural transformation
$b\cdot a:G_{1}\circ F_{1}\rightarrow G_{2}\circ F_{2}$ defined as
follows 
\[
(b\cdot a)_{X}:=b_{F_{2}X}\circ G_{1}(a_{X})=G_{2}(a_{X})\circ b_{F_{1}X}\quad\forall X\in\mathcal{X}.
\]
We will use the notation $b\cdot F_{1}:=b\cdot1_{F_{1}}$ and $G_{1}\cdot a:=1_{G_{1}}\cdot a$.

Let $\C$ be an additive category and $X=\coprod_{i\in I}X_{i}$ be
a coproduct in $\C,$ via the family $\{\mu_{i}^{X}:X_{i}\to X\}_{i\in I}$
of \textbf{canonical inclusions}. By using the universal property of
coproducts, we get the family $\{\pi_{i}^{X}:X\to X_{i}\}_{i\in I}$
of \textbf{natural projections} satisfying that $\pi_{i}^{X}\circ\mu_{i}^{X}=1_{X_{i}}$ for each $i,$
and $\pi_{i}^{X}\circ\mu_{j}^{X}=0$ for $j\neq i.$
In case $X_{i}=B_{i}$ for all $i\in I$, we set $B^{(I)}:=\coprod_{i\in I}X_{i}$.
Moreover, from the universal property of coproducts, we can define
the \textbf{codiagonal} morphism $\nabla_{B}:B^{(I)}\rightarrow B$
as the one that $\nabla_{B}\circ\mu_{i}^{X}=1_{B},$ for all $i\in I$.

Assume now that the additive category $\C$ has coproducts. Let $X=\coprod_{i\in I}X_{i}$
be a coproduct in $\C,$ via the family $\{\mu_{i}^{X}:X_{i}\to X\}_{i\in I}.$
In case $|I|\geq 2,$  we will use the following manipulation,
for each $i_{0}\in I$. Consider the coproduct $Y_{2}=\coprod_{j\in I-\{i_{0}\}}X_{j}$
in $\C$ with natural inclusions $\{\mu'_{j}:X_{i}\rightarrow Y_{2}\}_{j\in I-\{i_{0}\}}$.
Then, we have the coproduct $X=Y_{1}\coprod Y_{2}$ in $\C,$ where
$Y_{1}:=X_{i_{0}}$ and the natural inclusions $\mu_{Y_{1}}:=\mu_{i_{0}}^{X}:Y_{1}\to X$
and $\mu_{Y_{2}}:Y_{2}\to X$ is the morphism satisfying that $\mu_{Y_{2}}\circ\mu'_{j}=\mu_{j}^{X}$
for all $j\in I-\{i_{0}\}$.

We introduce now the matrices which are related with additive bifunctors.
Let $F:\mathcal{C}^{op}\times\mathcal{C}\rightarrow\Ab$ be an additive
bifunctor. For $x\in F(A,B)$, $f\in\Hom_{\mathcal{C}}(A',A)$ and
$g\in\Hom_{\mathcal{C}}(B,B')$, we will use the notation 
\begin{center}
$x\cdot f:=F(f^{op},B)(x)$ and $g\cdot x:=F(A,g)(x).$ 
\par\end{center}

Now let $A=\coprod_{i=1}^{n}A_{i}$ and $B=\coprod_{i=1}^{m}B_{i}$
in $\C,$ and let $\Mat_{m\times n}^{F}(A,B)$ be the set whose elements
are matrices $\alpha$ of size $m\times n$ with $i,j$-entry $[\alpha]_{i,j}$
in $F(A_{j},B_{i}).$ Notice that $\Mat_{m\times n}^{F}(A,B)$ is
an abelian group with the usual sum of matrices, and each $x\in F(A,B)$
can be identified with a matrix in $\Mat_{m\times n}^{F}(A,B).$ More
precisely, we have the following result that can be found in \cite[Prop.1.5.17]{BN}.

\begin{lem}\label{FtoM} Let $A=\coprod_{i=1}^{n}A_{i}$ and $B=\coprod_{i=1}^{m}B_{i}$
in $\C.$ Then, the map $\Phi_{B,A}^{F}:F(A,B)\to\Mat_{m\times n}^{F}(A,B),$
defined by  
$
[\Phi_{B,A}^{F}(x)]_{i,j}:=\pi_{i}^{B}\cdot x\cdot\mu_{j}^{A},
$
is an isomorphism of abelian groups whose inverse is 
$
(\Phi_{B,A}^{F})^{-1}(\alpha)=\sum_{i,j}\mu_{i}^{B}\cdot[\alpha]_{i,j}\cdot\pi_{j}^{A}.
$
\end{lem}

If $F=\Hom_{\C}(-,?),$ the isomorphism $\varphi_{B,A}:\Hom_{\C}(A,B)\to\Mat_{m\times n}(A,B),$
where $\varphi_{B,A}:=\Phi_{B,A}^{\Hom_{\C}}$ and $\Mat_{m\times n}(A,B):=\Mat_{m\times n}^{\Hom_{\C}}(A,B),$
is well known. Moreover, the actions $x\cdot f$ and $g\cdot x$ described
above can be computed as the usual product of matrices. More precisely,
we have the following result that can be found in \cite[Prop.1.5.20]{BN},
where the product of matrices is given in the usual way.

\begin{prop}\label{proFtoM}  Let $A=\coprod_{i=1}^{n}A_{i}$ and $B=\coprod_{i=1}^{m}B_{i}$
in $\C.$ Then, for $x\in F(A,B),$ $f\in\Hom_{\C}(A',A)$ and $g\in\Hom_{\C}(B,B'),$
we have that 
\begin{center}
$\Phi_{B,A'}^{F}(x\cdot f)=\Phi_{B,A}^{F}(x)\varphi_{A,A'}(f)$ and
$\Phi_{B',A}^{F}(g\cdot x)=\varphi_{B',B}(g)\Phi_{B,A}^{F}(x).$ 
\par\end{center}
\end{prop}

As a consequence of the above, we can identify each $x\in F(A,B),$
$f\in\Hom_{\C}(A',A)$ and $g\in\Hom_{\C}(B,B')$ with its corresponding
matrix and to use those matrices to compute $x\cdot f$ and $g\cdot x.$

\subsection{Recollements}

\begin{defn}\label{def:recad} Let $\mathcal{A}$, $\mathcal{B}$ and $\mathcal{C}$
be additive categories. A \textbf{recollement} of $\mathcal{B}$ by
$\mathcal{A}$ and $\mathcal{C}$ is a diagram of additive functors
\\
\noindent\begin{minipage}[t]{1\columnwidth}%
\[
\begin{tikzpicture}[-,>=to,shorten >=1pt,auto,node distance=2cm,main node/.style=,x=2cm,y=-2cm]

\node (1) at (1,0) {$\mathcal{A}$};
\node (2) at (2,0) {$\mathcal{B}$};
\node (3) at (3,0) {$\mathcal{C}$};

\draw[-> , thin]  (1)  to  node  {$i_*$} (2);
\draw[-> , thin]  (2)  to  node  {$j^*$} (3);

\draw[-> , thin,in=45,out=135, above]  (3)  to  node  {$j_!$} (2);
\draw[-> , thin,in=45,out=135, above]  (2)  to  node  {$i^*$} (1);

\draw[-> , thin,in=-45,out=-135, below]  (3)  to  node  {$j_*$} (2);
\draw[-> , thin,in=-45,out=-135, below]  (2)  to  node  {$i^!$} (1);

\end{tikzpicture}
\]%
\end{minipage}\\
satisfying the following conditions.
\begin{enumerate}
\item [(AR1)]$(i^{*},i_{*},i^{!})$ and $(j_{!},j^{*},j_{*})$ are adjoint
triples;
\item [(AR2)]the functors $i_{*}$, $j_{!}$ and $j_{*}$ are fully faithful;
\item [(AR3)]$\im(i_{*})=\Ker(j^{*})$.
\end{enumerate}
\end{defn}

\begin{rem}\label{rem:rec} Observe the following facts for a recollement of
additive categories as in definition above.
\begin{enumerate}
\item It follows from (AR1) and (AR3) that $i^{*}\circ j_{!}=i^{!}\circ j_{!}=0$. 
\item Recall that an adjoint pair $(S:\mathcal{Y}\rightarrow\mathcal{X},T:\mathcal{X}\rightarrow\mathcal{Y})$
induces natural transformations $\varphi:1_{\mathcal{Y}}\rightarrow T\circ S$
and $\psi:S\circ T\rightarrow1_{\mathcal{X}}$ such that 
\[
\psi_{SY}\circ S(\varphi_{Y})=1_{SY}\text{ and }T(\psi_{X})\circ\varphi_{TX}=1_{TX}
\]
 for all $Y\in\mathcal{Y}$ and for all $X\in\mathcal{X}$ (see \cite[Thm.3.1.5]{B}).
The natural transformation $\varphi$ is called the \textbf{unit}
of $(S,T)$, and $\psi$ is called the \textbf{co-unit} of $(S,T)$.
Throughout the paper, we will denote as follows the respective units
and co-units induced by the adjoint pairs in the recollement: 
\begin{alignat*}{1}
_{1}\varphi:1_{\mathcal{B}}\rightarrow j_{*}\circ j^{*}\text{, } & \quad{}{}_{1}\psi:j^{*}\circ j_{*}\rightarrow1_{\mathcal{C}}\text{, }\\
{}{}_{2}\varphi:1_{\mathcal{B}}\rightarrow i_{*}\circ i^{*}\text{, } & \quad{}{}_{2}\psi:i^{*}\circ i_{*}\rightarrow1_{\mathcal{A}},\\
_{3}\varphi:1_{\mathcal{A}}\rightarrow i^{!}\circ i_{*}\text{, } & \quad{}{}_{3}\psi:i_{*}\circ i^{!}\rightarrow1_{\mathcal{B}}\text{, }\\
{}{}_{4}\varphi:1_{\mathcal{C}}\rightarrow j^{*}\circ j_{!}\text{, } & \quad{}{}_{4}\psi:j_{!}\circ j^{*}\rightarrow1_{\mathcal{B}}.
\end{alignat*}
Here, note that $_{1}\psi$, $_{2}\psi$, $\varphi_{3}$ and $\varphi_{4}$
are isomorphisms since $i_{*}$, $j_{!}$, and $j_{*}$ fully faithful
(see \cite[Prop.3.4.1]{B}). Moreover, this implies that $_{1}\varphi_{j_{*}C}$,
$j^{*}({}_{1}\varphi_{B})$, $_{2}\varphi_{i_{*}A}$, $i^{*}({}_{2}\varphi_{B})$,
$_{3}\psi_{i_{*}A}$, $i^{!}({}_{3}\psi_{B})$, $_{4}\psi_{j_{!}C}$
and $j^{*}({}_{4}\psi_{B})$ are isomorphisms for all $A\in\mathcal{A}$,
$B\in\mathcal{B}$ and $C\in\mathcal{C}$.

\item If $\mathcal{B}$ has coproducts, then $\mathcal{A}$ has coproducts.
Indeed, for a set of objects $\{A_{i}\}_{i\in I}$ in $\mathcal{A}$,
consider the coproduct $\coprod_{i\in I}^{\mathcal{B}}i_{*}A_{i}$ in $\mathcal{B}.$
Hence, since $(i^{*},i_{*})$ is an adjoint pair and $i_{*}$ is fully
faithful, we have that $i^{*}\left(\coprod_{i\in I}^{\mathcal{B}}i_{*}A_{i}\right)$
is the coproduct of $\{A_{i}\}_{i\in I}$ in $\mathcal{A}$. That
is, 
\[
i^{*}\left(\coprod_{i\in I}^{\mathcal{B}}i_{*}A_{i}\right)=\coprod_{i\in I}^{\mathcal{A}}A_{i}.
\]
\item Similarly, if $\mathcal{B}$ has coproducts, then $\mathcal{C}$ has
coproducts. Moreover, for any set of objects $\{C_{i}\}_{i\in I}$
in $\mathcal{C}$, we have that 
\[
j^{*}\left(\coprod_{i\in I}^{\mathcal{B}}j_{*}C_{i}\right)=\coprod_{i\in I}^{\mathcal{C}}C_{i}.
\]
\item If $\mathcal{A}$ has coproducts, then $i_{*}$ commutes with coproducts.
To see this, consider a set of objects $\{A_{i}\}_{i\in I}$ in $\mathcal{A}$.
Then, using the adjoint pair $(i_{*},i^{!})$, one can check that
$i_{*}\left(\coprod_{i\in I}^{\mathcal{A}}A_{i}\right)$ is the coproduct
of $\{i_{*}A_{i}\}_{i\in I}$ in $\mathcal{B}$. Hence, 
\[
i_{*}\left(\coprod_{i\in I}^{\mathcal{A}}A_{i}\right)=\coprod_{i\in I}^{\mathcal{B}}i_{*}A_{i}.
\]
\item Similarly, if $\mathcal{C}$ has coproducts, then $j_{!}$ commutes
with coproducts.
\end{enumerate}
\end{rem}

\subsection{Extriangulated categories and functors}

An extriangulated category is a triple $(\mathcal{C},\mathbb{E},\s)$
consisting of an additive category $\mathcal{C},$ an additive bifunctor
$F:\mathcal{C}^{op}\times\mathcal{C}\rightarrow\Ab$ and a realization
$\s$ which sends each $\eta\in\mathbb{E}(A,B)$ to an equivalence
class $[\suc[B][E][A][a][b]]$ of a sequence of morphisms $\suc[B][E][A][a][b]$
in $\C$ satisfying a series of axioms (see \cite[Def.2.12]{NP}). For
each $A,B\in\C,$ we will refer to the elements of the abelian group
$\mathbb{E}(A,B)$ as \textbf{$\mathbb{E}$-extensions}. An \textbf{$\s$-conflation}
in $\C$ is a sequence $\suc[B][E][A][f][g]$ of morphisms in $\C$
such that $\s(\eta)=[\suc[B][E][A][f][g]]$ for some $\eta\in\mathbb{E}(A,B),$
and it is also denoted by $\eta:\;\suc[B][E][A][f][g]$ or $B\xrightarrow{a}E\xrightarrow{b}A\xrightarrow{\eta}.$
In such case, we say that $f$ is an \textbf{$\s$-inflation} and
$g$ is an \textbf{$\s$-deflation}.

\begin{defn}\cite[Cond.5.8]{NP} Let $(\mathcal{C},\mathbb{E},\s)$ be an extriangulated
category. We say that $\mathcal{C}$ satisfies the \textbf{WIC condition}
if, for composable morphisms $g$ and $f$, we have that $f$ is an
$\s$-inflation (resp. $g$ is a $\s$-deflation) if $g\circ f$ is
an $\s$-inflation (resp. $\s$-deflation). 
\end{defn}

\begin{prop}\label{prop:WIC}\cite[Prop.2.5]{BHST} Let $(\mathcal{C},\mathbb{E},\s)$
be an extriangulated category. Then, the following conditions are
equivalent. 
\begin{enumerate}
\item Every split-epi has a kernel in $\mathcal{C}$. 
\item Every split-mono has a cokernel in $\mathcal{C}$. 
\item Every split-epi is an $\s$-deflation.
\item Every split-mono is an $\s$-inflation. 
\item $(\mathcal{C},\mathbb{E},\s)$ satisfies the WIC condition. 
\end{enumerate}
\end{prop}

\begin{lem}\label{lemita} For an extriangulated category $(\mathcal{C},\E,\s),$
the following statements hold true. 
\begin{enumerate}
\item $\eta:\:$ $B\amalg D\overset{\left[\begin{smallmatrix}f & 0\\
0 & 1
\end{smallmatrix}\right]}{\rightarrow}E\amalg D\overset{\left[\begin{smallmatrix}g & 0\end{smallmatrix}\right]}{\rightarrow}A$ is an $\s$-conflation if, and only if, $\suc[B][E][A][f][g]$ is
the realization of an $\E$-extension $\eta'$ and $\left[\begin{smallmatrix}1\\
0
\end{smallmatrix}\right]\cdot\eta'=\left[\begin{smallmatrix}\eta'\\
0
\end{smallmatrix}\right]=\eta$. 
\item $\eta:\:B\overset{\left[\begin{smallmatrix}f\\
0
\end{smallmatrix}\right]}{\rightarrow}E\amalg D\overset{\left[\begin{smallmatrix}g & 0\\
0 & 1
\end{smallmatrix}\right]}{\rightarrow}A\amalg D$ is an $\s$-conflation if, and only if, $\suc[B][E][A][f][g]$ is
the realization of an $\E$-extension $\eta'$ and $\eta=\left[\begin{smallmatrix}\eta' & 0\end{smallmatrix}\right]=\eta'\cdot\left[\begin{smallmatrix}1 & 0\end{smallmatrix}\right]$. 
\item $\eta:\:$ $D\amalg B\overset{\left[\begin{smallmatrix}1 & 0\\
0 & f
\end{smallmatrix}\right]}{\rightarrow}D\amalg E\overset{\left[\begin{smallmatrix}0 & g\end{smallmatrix}\right]}{\rightarrow}A$ is an $\s$-conflation if, and only if, $\suc[B][E][A][f][g]$ is
the realization of an $\E$-extension $\eta'$ and $\left[\begin{smallmatrix}0\\
1
\end{smallmatrix}\right]\cdot\eta'=\left[\begin{smallmatrix}0\\
\eta'
\end{smallmatrix}\right]=\eta$. 
\item $\eta:\:B\overset{\left[\begin{smallmatrix}0\\
f
\end{smallmatrix}\right]}{\rightarrow}D\amalg E\overset{\left[\begin{smallmatrix}1 & 0\\
0 & g
\end{smallmatrix}\right]}{\rightarrow}D\amalg A$ is an $\s$-conflation if, and only if, $\suc[B][E][A][f][g]$ is
the realization of an $\E$-extension $\eta'$ and $\eta=\left[\begin{smallmatrix}0 & \eta'\end{smallmatrix}\right]=\eta'\cdot\left[\begin{smallmatrix}0 & 1\end{smallmatrix}\right]$. 
\end{enumerate}
\end{lem}

\begin{proof}
We only prove (a) since the others statements follow by dual or similar
arguments.

$(\Rightarrow)$ Consider the $\s$-conflations $\eta:\:B\amalg D\overset{\left[\begin{smallmatrix}f & 0\\
0 & 1
\end{smallmatrix}\right]}{\rightarrow}E\amalg D\overset{\left[\begin{smallmatrix}g & 0\end{smallmatrix}\right]}{\rightarrow}A,$\\
 $\eta_{0}:\:B\overset{\left[\begin{smallmatrix}1\\
0
\end{smallmatrix}\right]}{\rightarrow}B\amalg D\overset{\left[\begin{smallmatrix}0 & 1\end{smallmatrix}\right]}{\rightarrow}D$ and $\eta_{1}:\:E\overset{\left[\begin{smallmatrix}1\\
0
\end{smallmatrix}\right]}{\rightarrow}E\amalg D\overset{\left[\begin{smallmatrix}0 & 1\end{smallmatrix}\right]}{\rightarrow}D$. Since $\left[\begin{smallmatrix}0 & 1\end{smallmatrix}\right]\circ\left[\begin{smallmatrix}f & 0\\
0 & 1
\end{smallmatrix}\right]=\left[\begin{smallmatrix}0 & 1\end{smallmatrix}\right]$, by \cite[Prop.3.17]{NP}, there is an $\s$-conflation $\eta':\:\suc[B][E][A][x][y]$
such that $\left[\begin{smallmatrix}\eta'\\
0
\end{smallmatrix}\right]=\left[\begin{smallmatrix}1\\
0
\end{smallmatrix}\right]\cdot\eta'=\eta$ and such that the following diagram commutes: 
\[
\xymatrix{B\ar[r]^{\left[\begin{smallmatrix}1\\
0
\end{smallmatrix}\right]}\ar[d]^{x} & B\amalg D\ar[r]^{\left[\begin{smallmatrix}0 & 1\end{smallmatrix}\right]}\ar[d]^{\left[\begin{smallmatrix}f & 0\\
0 & 1
\end{smallmatrix}\right]} & D\ar@{=}[d]\\
E\ar[r]^{\left[\begin{smallmatrix}1\\
0
\end{smallmatrix}\right]}\ar[d]^{y} & E\amalg D\ar[r]^{\left[\begin{smallmatrix}0 & 1\end{smallmatrix}\right]}\ar[d]^{\left[\begin{smallmatrix}g & 0\end{smallmatrix}\right]} & D\\
A\ar@{=}[r] & A
}
\]
Lastly, note that the commutativity of the diagram implies that $x=f$
and $y=g$.

$(\Leftarrow)$ Observe that $\left[\begin{smallmatrix}\eta' & 0\\
0 & 0
\end{smallmatrix}\right]$ is realized by $[B\amalg D\overset{\left[\begin{smallmatrix}f & 0\\
0 & 1
\end{smallmatrix}\right]}{\rightarrow}E\amalg D\overset{\left[\begin{smallmatrix}g & 0\\
0 & 0
\end{smallmatrix}\right]}{\rightarrow}A\amalg0]$. On the other hand, there are isomorphisms $1:B\amalg D\rightarrow B\amalg D$
and $\left[\begin{smallmatrix}1\\
0
\end{smallmatrix}\right]:A\rightarrow A\amalg0$. By \cite[Prop.3.7]{NP}, $\eta=\left[\begin{smallmatrix}\eta' & 0\\
0 & 0
\end{smallmatrix}\right]\cdot\left[\begin{smallmatrix}1\\
0
\end{smallmatrix}\right]$ is realized by $[B\amalg D\overset{\left[\begin{smallmatrix}f & 0\\
0 & 1
\end{smallmatrix}\right]}{\rightarrow}E\amalg D\overset{\left[\begin{smallmatrix}g & 0\end{smallmatrix}\right]}{\rightarrow}A]$. 
\end{proof}

\begin{lem}\label{lemita2} For an extriangulated category $(\mathcal{C},\E,\s)$
and the $\s$-conflations 
\[
\eta:\:B\amalg B'\overset{\left[\begin{smallmatrix}f & 0\\
0 & f'
\end{smallmatrix}\right]}{\rightarrow}E\amalg E'\overset{\left[\begin{smallmatrix}g & 0\\
0 & g'
\end{smallmatrix}\right]}{\rightarrow}A\amalg A'\text{, }\eta_{1}:\:\suc[B][E][A][f][g]\text{ and }\eta_{2}:\:\suc[B'][E'][A'][f'][g']
\]
the following statements hold true. 
\begin{enumerate}
\item There are isomorphisms $v:A\amalg A'\rightarrow A\amalg A'$ and $u:B\amalg B'\rightarrow B\amalg B'$
such that $\eta\cdot v=\left[\begin{smallmatrix}\eta_{1} & 0\\
0 & \eta_{2}
\end{smallmatrix}\right]=u\cdot\eta$. 
\item If $f$ and $f'$ are monic (or $g$ and $g'$ are epic), then $\eta=\left[\begin{smallmatrix}\eta_{1} & 0\\
0 & \eta_{2}
\end{smallmatrix}\right]$. 
\end{enumerate}
\end{lem}
\begin{proof}
(a)  Observe that $\s(\eta)=\s(\eta_{1}\amalg\eta_{2})$. Hence, (a) follows
from \cite[Rem.3.10]{NP} and its dual. 
\

(b) If $f$ and $f'$ are monic, then $\left[\begin{smallmatrix}f & 0\\
0 & f'
\end{smallmatrix}\right]$ is monic, and thus $u=1_{B\amalg B'}$. Similarly, if $g$ and $g'$
are epic, then $v=1_{A\amalg A'}$. 
\end{proof}

Now we recall from \cite{B} the notion of extriangulated functor
and also the composition of such functors. For an additive functor
$F:\C\to\D$ and an equivalence class $[\suc[B][E][A][a][b]]$ of
a sequence of morphisms $\suc[B][E][A][a][b]$ in $\C,$ we set $F([\suc[B][E][A][a][b]]):=[\suc[FB][FE][FA][Fa][Fb]].$

\begin{defn}\label{def:funextr}\cite[Def.2.32]{BS} Let $(\mathcal{C},\mathbb{E},\s)$
and $(\mathcal{D},\mathbb{F},\t)$ be extriangulated categories, and
let $F:\mathcal{C}\rightarrow\mathcal{D}$ be a functor. We say that
$F$ is \textbf{extriangulated} if it is additive and there is a natural
transformation $\Gamma_{F}:\mathbb{E}\rightarrow\mathbb{F}\circ(F^{op}\times F)$
such that 
$\t((\Gamma_{F})_{C,A}(\eta))=F(\s(\eta)),$ for $\eta\in\Ebb(C,A).$ 
\end{defn}

Observe that,  for $a\in\Hom_{\C}(A,A'),$
$c\in\Hom_{\C}(C',C)$ and $\eta\in\Ebb(C,A),$ the naturality of
$\Gamma_{F}$ implies that 
\[
F(a)\cdot(\Gamma_{F})_{C,A}(\eta)\cdot F(c)=(\Gamma_{F})_{C',A'}(a\cdot\eta\cdot c).
\]
Let $(\mathcal{E},\mathbb{G},\mathfrak{o})$ be a third extriangulated
category and consider an extriangulated functor $G:\mathcal{D}\rightarrow\mathcal{E}$.
As in \cite[Def.2.11]{NOS}, define (by using Godement product) the
\textbf{composition of extriangulated functors} $G\circ F,$ via $\Gamma_{G\circ F}:=(\Gamma_{G}\cdot(F^{op}\times F))\circ\Gamma_{F},$
see diagram below.\\
\[
\begin{tikzpicture}[-,>=to,shorten >=1pt,auto,node distance=2.5cm,main node/.style=,x=2cm,y=-2cm]
\coordinate (U) at (0,0.5);
\coordinate (V) at (1,0);
\coordinate (U') at (0,-0.5);
\coordinate (V') at (1,1);
\node (1) at (0,0.5)  {$\mathcal{D}^{op} \times \mathcal{D} $};
\node (2) at (1,0)  {$\mathbf{Ab} $};
\node (2') at (1,1)  {$\mathcal{E}^{op} \times \mathcal{E} $};
\node (3) at (2,0.5)  {$\mathbf{Ab}$};
\node (4) at (0,-.5)  {$\mathbf{Ab}$};
\node (5) at (-1,0)  {$\mathcal{C}^{op} \times \mathcal{C} $};
\node (A1) at (barycentric cs:U=1 ,U'=.5)  {$$};
\node (A2) at (barycentric cs:U=.5 ,U'=1)  {$$};
\node (B1) at (barycentric cs:V'=.5 ,V=1)  {$$};
\node (B2) at (barycentric cs:V'=1 ,V=.5)  {$$};
\draw[->, thin]  (1)  to [below left] node  {$\scriptstyle {G}^{op} \times G $} (2');
\draw[->, thin]  (1)  to [below right] node  {$\scriptstyle \mathbb{F} $} (2);
\draw[-, double]  (2)  to  node  {$$} (3);
\draw[->, thin]  (2')  to [below right] node  {$\scriptstyle \mathbb{G} $} (3);
\draw[->, thin]  (5)  to [below left] node  {$\scriptstyle F^{op}\times F$} (1);
\draw[->, thin]  (5)  to  node  {$\scriptstyle \mathbb{E} $} (4);
\draw[-, double]  (4)  to  node  {$$} (2);
\draw[->, thin]  (A2)  to node  {$\scriptstyle \Gamma_F $} (A1);
\draw[->, thin]  (B1)  to node  {$\scriptstyle \Gamma_G$} (B2);
\end{tikzpicture} \]

\medskip{}

Lastly, for extriangulated functors $H:\mathcal{C}\rightarrow\mathcal{D}$
and $F:\mathcal{C}\rightarrow\mathcal{D}$, define a\textbf{ natural
transformation of extriangulated functors} $\alpha:(F,\Gamma_{F})\rightarrow(H,\Gamma_{H})$
as a natural transformation $\alpha:F\rightarrow H$ such that 
\[
(\mathbb{F}\cdot(\alpha^{op}\times1_{H}))\circ\Gamma_{H}=(\mathbb{F}\cdot(1_{F}^{op}\times\alpha))\circ\Gamma_{F}.
\]
That is, for any $\eta\in\mathbb{E}(A,B),$ we have that $\Gamma_{H}(\eta)\cdot\alpha_{A}=\alpha_{B}\cdot\Gamma_{F}(\eta)$.

\begin{prop}\label{prop:isovsextr}Let $(\mathcal{C},\mathbb{E},\s)$ and $(\mathcal{D},\mathbb{F},\t)$
be extriangulated categories, $F:\mathcal{C}\rightarrow\mathcal{D}$
and $G:\mathcal{C}\rightarrow\mathcal{D}$ functors, and $\varphi:F\rightarrow G$
a natural isomorphism. If $F$ is
extriangulated and $G$ is additive, then $G$ is extriangulated and $\varphi$ is a natural
transformation of extriangulated functors. 
\end{prop}

\begin{proof}
Since $F$ is extriangulated, there is a natural transformation $\Gamma_{F}:\mathbb{E}\rightarrow\mathbb{F}\circ(F^{op}\times F)$
such that $\t((\Gamma_{F})_{C,A}(\eta))=F(\s(\eta)),$ for $\eta\in\Ebb(C,A).$
Define $\Gamma_{G}:\mathbb{E}\rightarrow\mathbb{F}\circ(G^{op}\times G)$
as $(\mathbb{F}\cdot((\varphi^{op})^{-1}\times\varphi))\circ\Gamma_{F}$.
That is, for $X,Y\in\mathcal{C}$, $\Gamma_{G}$ is the natural transformation
\[
\mathbb{E}(X,Y)\rightarrow\mathbb{F}(G(X),G(Y))\text{, }\epsilon\mapsto\varphi_{Y}\cdot\Gamma_{F}(\epsilon)\cdot\varphi_{X}^{-1}.
\]
Observe that, for $\epsilon:\:\suc[Y][Z][X][a][b]$, we have that
$\Gamma_{G}(\epsilon)\cdot\varphi_{X}=\varphi_{Y}\cdot\Gamma_{F}(\epsilon)$.
Therefore, we have a morphism of conflations $(\varphi_{X},\varphi_{Y}):\Gamma_{F}(\epsilon)\rightarrow\Gamma_{G}(\epsilon)$.
The realization of this morphism together with the natural transformation
$\varphi$ gives us the following commutative diagram: 
\[
\xymatrix{FY\ar[r]^{Fa}\ar[d]^{\varphi_{Y}} & FZ\ar[r]^{Fb}\ar[d]^{h} & FX\ar[d]^{\varphi_{X}}\\
GY\ar[r]^{a'}\ar@{=}[d] & M'\ar[r]^{b'}\ar[d]^{\varphi_{Z}\circ h^{-1}} & GX\ar@{=}[d]\\
GY\ar[r]^{Ga} & GZ\ar[r]^{Gb} & GX
}
\]
Here, $h$ is an isomorphism since $\varphi_{Y}$ and $\varphi_{X}$
are isomorphisms (see \cite[Cor.3.6]{NP}); and the bottom squares
commute since $Ga=\varphi_{Z}\circ Fa\circ\varphi_{Y}^{-1}$ and $Gb=\varphi_{X}\circ Fb\circ\varphi_{Z}^{-1}$.
In conclusion, $G$ is extriangulated and $\varphi$ is a natural
transformation of extriangulated functors.
\end{proof}

\begin{defn}
Let $(\mathcal{C},\mathbb{E},\s)$ and $(\mathcal{D},\mathbb{F},\t)$
be extriangulated categories, and let $F:\mathcal{C}\rightarrow\mathcal{D}$
be an additive functor. We say that $F$ preserves \textbf{inflations} (deflations)
if $F(f)$ is an inflation (deflation) in $\mathcal{D}$ for any inflation (deflation) $f$
in $\mathcal{C}$. 
\end{defn}

The following lemma implies that a natural isomorphism, between functors
that preserve inflations (resp. deflations), maps inflations into
inflations (resp. deflations into deflations).

\begin{lem}\label{lem:infdef} Let $(\mathcal{C},\mathbb{E},\s)$ be an extriangulated
category and $\epsilon:\:\suc[][][][a][b]$ be an $\s$-conflation. Then, for a morphism $f:X\rightarrow Y,$ the following statements hold true. 
\begin{enumerate}
\item If there are isomorphisms $\alpha:N\rightarrow X$ and $\beta:M\rightarrow Y$
such that $\beta\circ a=f\circ\alpha$, then $\eta:\; X\xrightarrow{f}Y\xrightarrow{b\circ\beta^{-1}}K$
is an $\s$-conflation such that $\eta=\alpha\cdot\epsilon$.
\item If there are isomorphisms $\beta:X\rightarrow M$ and $\gamma:Y\rightarrow K$
such that $\gamma\circ f=b\circ\beta$, then $\eta:\;N\xrightarrow{\beta^{-1}\circ a}X\xrightarrow{f} Y$
is an $\s$-conflation such that $\eta=\epsilon\cdot\gamma$. 
\end{enumerate}
\end{lem}

\begin{proof}
We only prove (b) since the proof of (a) follows by duality.

Let $\suc[N][M'][Y][a'][b']$
be an $\E$-extension representing $\epsilon\cdot\gamma$. By definition,
we have a morphism of $\s$-conflations $\epsilon\cdot\gamma\rightarrow\epsilon$
realized by a triple of morphisms $(1_{N},\gamma_{0},\gamma)$. Now,
by \cite[Lem.2.5]{E}, we get that the diagram 
\[\xymatrix{M'\ar[d]_{\gamma_0}\ar[r]^{b'} & K\ar[d]^{\gamma}\\ 
M\ar[r]^b & K}\]
 is a weak pull-back. This implies that there is a
morphism $\beta_{0}:X\rightarrow M'$ such that $\gamma_{0}\circ\beta_{0}=\beta$
and $b'\circ\beta_{0}=f$. Hence, it remains to prove that $\beta_{0}$
is an isomorphism and that $\beta_{0}\circ\beta^{-1}\circ a=a'$ (see
diagram below)
\[
\xymatrix{N\ar[r]^{\beta^{-1}\circ a}\ar@{=}[d] & X\ar[r]^{f}\ar[d]^{\beta_{0}} & Y\ar@{=}[d]\\
N\ar[r]^{a'}\ar@{=}[d] & M'\ar[r]^{b'}\ar[d]^{\gamma_{0}} & Y\ar[d]^{\gamma}\\
N\ar[r]^{a} & M\ar[r]^{b} & K.
}
\]
For this, note that $\gamma_{0}$ is an isomorphism
by \cite[Cor.3.6]{NP}; and thus, $\beta_{0}$ is an isomorphism since
$\beta_{0}=\gamma_{0}^{-1}\circ\beta$. Moreover 
$$a'=\gamma_{0}^{-1}\circ a=\gamma_{0}^{-1}\circ\beta\circ\beta^{-1}\circ a=\beta_{0}\circ\beta^{-1}\circ a$$
proving the result.
\end{proof}

\subsection{Extriangulated categories with negative first extension}

Let $\mathcal{D}=(\mathcal{D},\mathbb{E},\s)$ be an extriangulated
category with negative first extension (e.g. an exact or a triangulated
category, see \cite[Def.2.3]{AET1}). An \textbf{$s$-torsion pair}
in $\mathcal{D}$ is a pair $\u=(\mathcal{X},\mathcal{Y})$ of full
subcategories of $\mathcal{D},$ which are closed under isomorphisms in $\C,$ such that $\Hom_{\mathcal{D}}(\mathcal{X},\mathcal{Y})=0$,
$\mathcal{D}=\mathcal{X}*\mathcal{Y}$ and $\mathbb{E}^{-1}(\mathcal{X},\mathcal{Y})=0$
(see \cite[Def.3.1]{AET1}). Observe that: for an exact category $\mathcal{D}$,
$\u=(\mathcal{X},\mathcal{Y})$ is an $s$-torsion pair if, and only
if, it is a torsion pair in the usual sense; and, for a triangulated
category $\mathcal{D},$ the pair $(\mathcal{U},\mathcal{W})=(\mathcal{U},\Sigma\mathcal{U}^{\bot_{0}})$
is a $t$-structure if, and only if, $(\mathcal{U},\Sigma^{-1}\mathcal{W})$
is an $s$-torsion pair. 

An important property of an $s$-torsion pair $\u=(\mathcal{X},\mathcal{Y})$
is that the inclusion functor $\mathcal{X}\rightarrow\mathcal{D}$
admits a right adjoint $\te_{\u}:\mathcal{D}\rightarrow\mathcal{X}$
and the inclusion functor $\mathcal{Y}\rightarrow\mathcal{D}$ admits
a left adjoint $(1:\te_{\u}):\mathcal{D}\rightarrow\mathcal{Y}$.
Moreover, for any $D\in\mathcal{D}$, there is an $\s$-conflation
$\suc[\te_{\mathbf{u}}D][D][(1:\te_{\mathbf{u}})D]$ which is known as the canonical 
$\s$-conflation attached to the object $D.$

Define $\stors(\mathcal{D})$ as the class of all the $s$-torsion pairs
in $\mathcal{D}$. Given $\u=(\mathcal{X},\mathcal{Y}),\u'=(\mathcal{X}',\mathcal{Y}')\in\stors(\mathcal{D})$,
we say that $\u\leq\u'$ if $\mathcal{X}\subseteq\mathcal{X}'$. In
this case, it is defined the interval $[\u,\u']:=\{\v\in\stors(\mathcal{D})\,|\:\u\leq\v\leq\u'\}$.
The class $\mathcal{H}_{[\u,\u']}:=\mathcal{X}'\cap\mathcal{Y}$ is
an extriangulated category with negative first extension known as
the \textbf{heart of the interval} $[\u,\u']$ (see \cite{AET1}). 

Part of the relevance of the hearts of intervals of $s$-torsion pairs
is that they help to parametrize the $s$-torsion pairs in the interval as follows. 

\begin{thm}\cite[Thm.3.9]{AET1}\label{thm:proceso aet} 
Let $\mathcal{D}$
be an extriangulated category with negative first extension, and let
$\mathbf{u}_{1}=(\mathcal{X}_{1},\mathcal{Y}_{1})$ and $\mathbf{u}_{2}=(\mathcal{X}_{2},\mathcal{Y}_{2})$
in $\stors\,\mathcal{\mathcal{D}}$ be such that $\mathbf{u}_{1}\leq\mathbf{u}_{2}$.
Then, there exist an isomorphism of posets 
\[
\Phi:[\mathbf{u}_{1},\mathbf{u}_{2}]\rightarrow\stors\,\mathcal{H}_{[\mathbf{u}_{1},\mathbf{u}_{2}]},\;(\mathcal{X},\mathcal{Y})\mapsto(\mathcal{X}\cap\mathcal{Y}_{1},\mathcal{X}_{2}\cap\mathcal{Y})
\]
with inverse 
\[
\Psi:\stors\,\mathcal{H}_{[\mathbf{u}_{1},\mathbf{u}_{2}]}\rightarrow[\mathbf{u}_{1},\mathbf{u}_{2}],\;(\mathcal{T},\mathcal{F})\mapsto(\mathcal{X}_{1}\star\mathcal{T},\mathcal{F}\star\mathcal{Y}_{2}).
\]
\end{thm}

A particular case of the heart of an interval is well-known and studied
in triangulated categories. Namely, for a $t$-structure $\mathbf{x}=(\mathcal{U},\mathcal{W})$
in a triangulated category $\mathcal{D}$, set $\u_{1}:=(\Sigma\mathcal{U},\mathcal{W})$,
$\u_{2}=(\mathcal{U},\Sigma^{-1}\mathcal{W})$ and $\u_{3}=(\Sigma^{-1}\mathcal{U},\Sigma^{-2}\mathcal{W})$.
Then, we have that $\mathcal{H}_{\mathbf{x}}:=\mathcal{H}_{[\u_{1},\u_{2}]}$
is the usual \textbf{heart} of $\mathbf{x}$ and $\mathcal{C}_{\mathbf{x}}:=\mathcal{H}_{[\u_{1},\u_{3}]}$
is the \textbf{extended heart} of $(\mathcal{U},\mathcal{W})$ (see \cite[Def.4.1]{AMP}). In
this context, the correspondence in Theorem \ref{thm:proceso aet}
is known as the Happel-Reiten-Smal{\o} tilting process (see \cite{HRS}).

An important feature of a $t$-structure $\mathbf{x}$ is that it
comes equipped with a cohomological functor $H_{\mathbf{x}}:\mathcal{D}\rightarrow\mathcal{H}_{\mathbf{x}}$.
A similar functor can be defined for any interval $[\u,\u']$ of $s$-torsion
pairs in an extriangulated category with negative first extension.
For this, observe that $\te_{\u'}\circ(1:\te_{\u})(D)\in\mathcal{H}_{[\u,\u']}$
for any $D\in\mathcal{D}$. Indeed, for $\u=(\mathcal{X},\mathcal{Y})$
and $D\in\mathcal{D}$, consider the canonical $\s$-conflation 
\[
\suc[\te_{\u'}\circ(1:\te_{\u})D][(1:\te_{\u})D][(1:\te_{\u'})\circ(1:\te_{\u})D].
\]
Since $(1:\te_{\u})D,(1:\te_{\u'})\circ(1:\te_{\u})D\in\mathcal{Y}$
and $\mathcal{Y}$ is closed under co-cones (see \cite[Prop.2.9]{AMP}),
we have that $\te_{\u'}\circ(1:\te_{\u})D\in\mathcal{H}_{[\u,\u']}$.

\begin{defn}\label{def:H} Let $\mathcal{D}=(\mathcal{D},\mathbb{E},\s)$ be an
extriangulated category with negative first extension and $[\u,\u']$
be an interval in $\stors(\mathcal{D})$, with $\u=(\mathcal{X},\mathcal{Y})$
and $\u'=(\mathcal{X}',\mathcal{Y}')$. Define the functor $H_{[\u,\u']}:\mathcal{D}\rightarrow\mathcal{H}_{[\u,\u']},\;D\mapsto\te_{\u'}\circ(1:\te_{\u})(D).$ 
\end{defn}

The following remarks will be useful. 

\begin{rem}\label{rem:adjuntos} Assume the conditions of the definition above.
\begin{enumerate}
\item It can be proved that $H_{[\u,\u']}(D)\cong(1:\te_{\u})\circ\te_{\u'}(D)$
for all $D\in\mathcal{D}$. 
\item Consider the inclusion functors 
\[
i_{\mathcal{X}'}:\mathcal{X}'\rightarrow\mathcal{D}\text{, }\quad i_{\mathcal{Y}}:\mathcal{Y}\rightarrow\mathcal{D}\text{, }\quad j_{\mathcal{X}'}:\mathcal{H}_{[\u,\u']}\rightarrow\mathcal{X}'\text{, and }\quad j_{\mathcal{Y}}:\mathcal{H}_{[\u,\u']}\rightarrow\mathcal{Y}.
\]
 For any $X'\in\mathcal{X}'$ and $D\in\mathcal{H}_{[\u,\u']}=\Y\cap\X',$ we have that 
\[
_{\mathcal{H}_{[\u,\u']}}(H_{[\u,\u']}(X'),D)\stackrel{(a)}{\cong}{}{}_{\mathcal{Y}}((1:\te_{\u})(X'),D)\cong{}{}_{\mathcal{D}}(X',D)={}{}_{\mathcal{X}'}(X',j_{\mathcal{X}'}D).
\]
Hence, the mapping $X'\mapsto H_{[\u,\u']}(X')$ defines a left adjoint
$L:\mathcal{X}'\rightarrow\mathcal{H}_{[\u,\u']}$ for $j_{\mathcal{X}'}$.
Similarly, $Y\mapsto H_{[\u,\u']}(Y)$ defines a right adjoint $R:\mathcal{Y}\rightarrow\mathcal{H}_{[\u,\u']}$
for $j_{\mathcal{Y}}$.
\end{enumerate}
\end{rem}

\begin{rem}\cite[Lem.3.1]{PS1}\label{rem:L} Let $\mathbf{x}=(\mathcal{U},\mathcal{W})$
be a $t$-structure in a triangulated category $(\mathcal{D},\Sigma,\triangle)$.
Set $\u_{1}:=(\Sigma\mathcal{U},\mathcal{W})$, $\u_{2}:=(\mathcal{U},\Sigma^{-1}\mathcal{W})$
and $\Sigma^{-n}\u_{2}:=(\Sigma^{-n}\mathcal{U},\Sigma^{-(n+1)}\mathcal{W})$
for all $n\geq1$. 
\begin{enumerate}
\item For $D\in\mathcal{U}$, we have that $H_{[\u_{1},\u_{2}]}(D)\cong\Sigma\left((1:\te_{\u_{2}})(\Sigma^{-1}D)\right).$\\
Indeed, $\Sigma^{-1}D\in\Sigma^{-1}\mathcal{U}$ and
thus, from the canonical $\s$-conflation 
\[
\suc[\te_{\u_{2}}(\Sigma^{-1}D)][\Sigma^{-1}D][(1:\te_{\u_{2}})(\Sigma^{-1}D)],
\]
we can deduce that $\Sigma\left((1:\te_{\u_{2}})(\Sigma^{-1}D)\right)\cong(1:\te_{\u_{1}})D.$
Therefore 
\[
H_{[\u_{1},\u_{2}]}(D)=(1:\te_{\u_{1}})\circ\te_{\u_{2}}(D)=(1:\te_{\u_{1}})(D)
\]
 since $\te_{\u_{2}}(D)=D$.
\item Similarly, for any $n\geq1$ and $D\in\Sigma^{-n+1}\mathcal{U}$,
we have that $$H_{[\u_{1},\Sigma^{-n+1}\u_{2}]}(D)\cong\Sigma\left((1:\te_{\u_{2}})(\Sigma^{-1}D)\right).$$
\end{enumerate}
\end{rem}

\begin{lem}\label{lem:hrs} Let $\mathcal{D}$ be an extriangulated category with
negative first extension, $\mathbf{u}_{1}=(\mathcal{X}_{1},\mathcal{Y}_{1})$
and $\mathbf{u}_{2}=(\mathcal{X}_{2},\mathcal{Y}_{2})$ in $\stors\,\mathcal{\mathcal{D}}$
be such that $\mathbf{u}_{1}\leq\mathbf{u}_{2}$, and consider $(\mathcal{T},\mathcal{F})\in\stors\mathcal{H}_{[\u_{1},\u_{2}]}$.
If $(\mathcal{X},\mathcal{Y})=\Psi(\mathcal{T},\mathcal{F}):=(\mathcal{X}_{1}\star\mathcal{T},\mathcal{F}\star\mathcal{Y}_{2})$ (see Theorem \ref{thm:proceso aet}),  then
\[
\mathcal{X}=\left\{ D\in\mathcal{X}_{2}\,|\:H_{[\u_{1},\u_{2}]}(D)\in\mathcal{T}\right\} \;\text{ and }\;\mathcal{Y}=\left\{ D\in\mathcal{Y}_{1}\,|\:H_{[\u_{1},\u_{2}]}(D)\in\mathcal{F}\right\} .
\]
\end{lem}

\begin{proof}
Let us prove the first equality, the second one follows by similar
arguments. Indeed, for $D\in\mathcal{X}=\mathcal{X}_{1}\star\mathcal{T}$,
there is an $\s$-conflation $\suc[X_{1}][D][T]$ with $X_{1}\in\mathcal{X}_{1}$
and $T\in\mathcal{T}$. Since $\mathcal{T}\subseteq\mathcal{X}_{2}\cap\mathcal{Y}_{1}$,
we can assume that $T=(1:\te_{\u_{1}})D$. Hence, using that $(1:\te_{\u_{2}})T=0$,
we have that $$T=\te_{\u_{2}}T=\te_{\u_{2}}\circ(1:\te_{\u_{1}})D=H_{[\u_{1},\u_{2}]}(D).$$
Now, consider $D\in\mathcal{X}_{2}$ with $H_{[\u_{1},\u_{2}]}(D)\in\mathcal{T}$.
Since $$H_{[\u_{1},\u_{2}]}(D)=\te_{\u_{1}}\circ(1:\te_{\u_{2}})(D)\cong(1:\te_{\u_{1}})\circ\te_{\u_{2}}(D),$$
we know that there is an $\s$-conflation $\suc[\te_{\u_{1}}\circ\te_{\u_{2}}D][\te_{\u_{2}}D][H_{[\u_{1},\u_{2}]}D]$.
Therefore, since $\te_{\u_{2}}D=D$ and $\te_{\u_{1}}\circ\te_{\u_{2}}D=\te_{\u_{1}}D\in\mathcal{X}_{1}$,
we conclude that $D\in\mathcal{X}$. 
\end{proof}

\section{The AET4 condition for extriangulated categories}

\subsection{The product category of an extriangulated category}\label{pcofexts}

Let $I$ be a non-empty set and $\mathcal{C}$ be a category. Observe
that $I$ can be viewed as a discrete category. That is, a category
where the objects are the elements of $I$ and the only morphisms
are the identities of the objects. Let us consider $\Fun(I,\mathcal{C})$,
the category of functors $I\rightarrow\mathcal{C}$ (also known as
the product category $\prod_{i\in I}\mathcal{C}$). Observe that $\Fun(I,\mathcal{C})$
is equipped with a family of functors 
$$\{(-)_{i}:\Fun(I,\mathcal{C})\rightarrow\mathcal{C}\}_{i\in I},$$
where $X_{i}:=X(i)$ for every $X\in\Fun(I,\mathcal{C})$, satisfying
the following universal property: for any family of functors $\{E_{i}:\mathcal{C}'\rightarrow\mathcal{C}\}_{i\in I}$
there is a unique functor $E:\mathcal{C}'\rightarrow\Fun(I,\mathcal{C})$
such that $(-)_{i}\circ E=E_{i}$ for all $i\in I$. In particular,
there is a functor $T:\mathcal{C}\rightarrow\Fun(\mathcal{C},I)$
such that $(-)_{i}\circ T=1_{\mathcal{C}}$ for all $i\in I$. We
will refer to $T$ as the \textbf{constant functor}.\textbf{ }Moreover,
in case $\mathcal{C}$ has all the products indexed by $I$, there
is a functor $P:=\prod_{i\in I}(-):\Fun(I,\mathcal{C})\rightarrow\mathcal{C}$
defined via $\prod_{i\in I}F=\prod_{i\in I}F_{i}$. Similarly, if
$\mathcal{C}$ has all coproducts indexed by $I$, one can define
a functor $S:=\coprod_{i\in I}(-):\Fun(I,\mathcal{C})\rightarrow\mathcal{C}$
via $\coprod_{i\in I}F=\coprod_{i\in I}F_{i}$.

The goal of this section is to briefly discuss how the category $\mathcal{D}:=\Fun(I,\mathcal{C})$
is extriangulated in case $\mathcal{C}$ is extriangulated.

Let $(\mathcal{C},\mathbb{E},\s)$ be an extriangulated category.
Note that $\mathcal{D}:=\Fun(I,\mathcal{C})$ is additive since $\mathcal{C}$
is additive  (see \cite[Chap.IV, Sec.7]{S}). Consider the additive functor
$\mathbb{F}:\mathcal{D}^{op}\times\mathcal{D}\rightarrow\Ab$, defined
as $\mathbb{F}(-,?):=\prod_{i\in I}\mathbb{E}((-)_{i},(?)_{i}).$ It can be seen that a sequence
$\suc[F][G][H][f][g]$ in $\mathcal{D}$ is equivalent to $\suc[F][G'][H][f'][g']$
if, and only if, $\suc[F_{i}][G_{i}][H_{i}][f_{i}][g_{i}]$ is equivalent
to $\suc[F_{i}][G'_{i}][H_{i}][f_{i}'][g_{i}']$ for all $i\in I$.
Therefore, we can define a realization $\t$ of $\mathbb{F}$ as the
one that associates the $\mathbb{F}$-extension $\eta=(\eta_i)_{i\in I}\in\mathbb{F}(F,G)$
the equivalence class induced by the equivalence classes $\{\s(\eta_{i})\}_{i\in I}$.
We will use the notation $\t=\prod_{i\in I}\s$.

Finally, axioms ET2, ET3 and ET4 are satisfied by $(\mathcal{D},\mathbb{F},\t)$.
To verify this, it is enough to apply the functors $\{(-)_{i}:\Fun(I,\mathcal{C})\rightarrow\mathcal{C}\}_{i\in I}$
and consider the corresponding axiom in $\mathcal{C}$.

\subsection{The adjoint pair associated to the coproduct}\label{adjpcops}

In this section, we consider the following setting. Let $I$ be a non-empty set and $(\mathcal{C},\mathbb{E},\s)$ be
an extriangulated category with all the coproducts indexed by $I$.
Consider the category $\mathcal{D}=\Fun(I,\C)$ and the functors $T:\mathcal{C}\rightarrow\Fun(I,\C)$
and $S:=\coprod_{i\in I}(-):\Fun(I,\C)\rightarrow\mathcal{C}$ defined
in section \ref{pcofexts}. It is well-known that $(S,T)$ is an adjoint pair (see \cite[Chap.IV, Sec.8]{S}
and \cite[Chap.IV, Sec.9, Ex.1]{S}). Specifically, note that for
$D\in\mathcal{D}$ the canonical inclusions $\{\mu_{i}^{D}:D_{i}\rightarrow\coprod_{i\in I}D_{i}\}_{i\in I}$
define a natural transformation $\varphi:1_{\mathcal{D}}\rightarrow T\circ S$;
and the co-diagonal morphism $\nabla:C^{(I)}\rightarrow C$ define
a natural transformation $\psi:S\circ T\rightarrow1_{\mathcal{C}}$.
Moreover, $\psi$ and $\varphi$ satisfy that $\psi_{SD}\circ S(\varphi_{D})=1_{SD}$
and $T(\psi_{C})\circ\varphi_{TC}=1_{TC}$ for all $D\in\mathcal{D}$
and $C\in\mathcal{C}.$ We also recall, see in section \ref{pcofexts} for details, that $(\mathcal{D},\mathbb{F},\t)$ is an extriangulated category, where $\mathbb{F}(-,?):=\prod_{i\in I}\mathbb{E}((-)_{i},(?)_{i})$ and $\t=\prod_{i\in I}\s.$ 
\vspace{0.2cm}

The following result is a consequence of the results obtained in the
Appendix.
 
\begin{prop}\label{prop:props} For the extriangulated categories $(\mathcal{C},\E,\s)$ and  $(\mathcal{D},\mathbb{F},\t),$ and the functors $T:\C\to \D$ and $S:\D\to \C$ defined above, the following statements hold true. 
\begin{enumerate}
\item The constant functor $T:\mathcal{C}\rightarrow\D$ is extriangulated
via the natural transformation 
\[
\Gamma_{T}:\mathbb{E}(?,-)\rightarrow\prod_{i\in I}\mathbb{E}(?,-)
\]
defined as $\Gamma_{T}(\eta)=(\eta_{i})_{i\in I}$ for $\eta\in\mathbb{E}(C',C)$,
where $\eta_{i}=\eta$ for all $i\in I$. 
\item There is a monic natural transformation 
\[
\tau:\mathbb{E}(\coprod_{i\in I}(?)_{i},-)\rightarrow\prod_{i\in I}\mathbb{E}((?)_{i},-)
\]
defined as $\tau(\eta):=(\eta\cdot\mu_{i}^{D})_{i\in I}$
for all $\eta\in\mathbb{E}(\coprod_{i\in I}D_{i},C)$. 
\item If $S:\mathcal{D}\rightarrow\mathcal{C}$ is $(1,T)$-extriangulated,
then there is a monic natural transformation 
\[
\sigma:\prod_{i\in I}\mathbb{E}((?)_{i},-)\rightarrow\mathbb{E}(\coprod_{i\in I}(?)_{i},-)
\]
defined as $\sigma(\eta)=\nabla_{C}\cdot\Gamma_{S}^{(1,T)}(\eta)$
for all $\eta\in\mathbb{F}(D,T(C))$. 
\item Let $S:\mathcal{D}\rightarrow\mathcal{C}$ be
$(1,T)$-extriangulated. 
\begin{enumerate}
\item If $\varphi:1\rightarrow T\circ S$ is $(1,T)$-extriangulated or
$\psi:S\circ T\rightarrow1$ is $(S,1)$-extriangulated, then $\tau$
and $\sigma$ are natural isomorphisms and $\tau=\sigma^{-1}$. 
\item If $\varphi:1\rightarrow T\circ S$ is $(1,T)$-extriangulated, then
$\psi:S\circ T\rightarrow1$ is $(S,1)$-extriangulated. 
\item If $S$ and $\varphi$ are extriangulated, then $\psi$
is extriangulated. 
\end{enumerate}
\end{enumerate}
\end{prop}

\begin{proof}
(a) Observe that $(-)_{i}\circ T=1_{\mathcal{C}}$ for all $i\in I$.
Hence, $\mathbb{F}(T(?),T(-))=\prod_{i\in I}\mathbb{E}(?,-)$. Now,
there is a natural transformation 
\[
\Gamma_{T}:\mathbb{E}(?,-)\rightarrow\mathbb{F}(T(?),T(-))=\prod_{i\in I}\mathbb{E}(?,-)
\]
defined by the universal property of products via the family of functors
$\{1_{\mathbb{E}}:\mathbb{E}(?,-)\rightarrow\mathbb{E}(?,-)\}_{i\in I}$.
Moreover, for an $\s$-conflation $\eta:\:\suc[A][B][C][f][g]$, we
have that the realization of $\Gamma_{T}(\eta)$ is given by $\{\s(\Gamma_{T}(\eta)_{i})=\s(\eta)=[\suc[A][B][C][f][g]]\}_{i\in I}$,
which is equal to $[\suc[TA][TB][TC][Tf][Tg]]$. Therefore, $T$ is
extriangulated.
 
(b) It follows from (a) and Propositions \ref{prop:existe} and \ref{prop:mono}. 
\

(c) It follows from Propositions \ref{prop:existe} and \ref{prop:mono}. 
\

(d) The item (d1) follows from Proposition \ref{prop:0}. Let us prove 
(d2). For this, observe that $T$ is extriangulated by (a). Hence,
it is enough to show, by Lemma \ref{lem:phi}(c), that $TST\psi_{SD}\circ\varphi_{TSTSD}=1_{TSTSD}$ for each $D\in\mathcal{D}.$ Now, $\varphi_{TSTSD}$
consists of the canonical inclusions 
\[
\{\mu_{k}:(\coprod_{i\in I}D_{i})^{(I)}\rightarrow((\coprod_{i\in I}D_{i})^{(I)})^{(I)}\}{}_{k\in I}
\]
and $ST\psi_{SD}$ consists of the coproduct 
\[
\nabla^{(I)}:((\coprod_{i\in I}D_{i})^{(I)})^{(I)}\rightarrow(\coprod_{i\in I}D_{i})^{(I)},
\]
where $\nabla:(\coprod_{i\in I}D_{i})^{(I)}\rightarrow\coprod_{i\in I}D_{i}$
is the codiagonal morphism. Hence, it follows from the universal property
of coproducts that 
\[
(TST\psi_{SD})_{i}\circ(\varphi_{TSTSD})_{i}=\nabla^{(I)}\circ\mu_{i}=1
\]
for all $i\in I$. Therefore $TST\psi_{SD}\circ\varphi_{TSTSD}=1_{TSTSD}$. 

Lastly, item (d3) follows from similar arguments
using Lemma \ref{lem:phi-1}(c) instead of Lemma \ref{lem:phi}(c).
\end{proof}

\subsection{The AET4 condition for extriangulated categories}

In this section we will make use of the notions and developments given in the appendix.

\begin{defn}
Let $(\mathcal{C},\mathbb{E},\s)$ and  $(\mathcal{D},\mathbb{F},\t)$ be the extriangulated categories introduced in section \ref{adjpcops}, where 
$I$ is a non-empty set and $\D=\Fun(I,\C).$
\begin{enumerate}
\item We say that $\mathcal{C}$ is \textbf{AET3($I$)} if, for any set
of objects $\{C_{i}\}_{i\in I}$ in $\C,$ the coproduct $\coprod_{i\in I}C_{i}$ in $\C$
exists. If $\mathcal{C}$ is AET3($X$) for every non-empty set $X$,
we say that it is \textbf{AET3}. 

\item We say that $\mathcal{C}$ is \textbf{AET3.5($I$)} if it is AET3($I$)
and the functor $S:=\coprod_{i\in I}(-):\mathcal{D}\rightarrow\mathcal{C}$
is $(1,T)$-extriangulated, where $T:\mathcal{C}\rightarrow\D$
is the constant functor. If $\mathcal{C}$ is AET3.5($X$) for every
non-empty set $X$, we say that it is \textbf{AET3.5}. 

\item We say that $\mathcal{C}$ is AET4\textbf{($I$)} if it is AET3.5($I$)
and $\varphi:1\rightarrow T\circ S$ is $(1,T)$-extriangulated, where $T$, $S$ and $\varphi$ are as above. If
$\mathcal{C}$ is AET4($X$) for every non-empty set $X$, we say
that it is AET4. Note that this notion of AET4 coincides with the one given in Definition \ref{ET4Def}.

\item We say that $\mathcal{C}$ is \textbf{AET3{*}} (resp. \textbf{AET3.5}{*},
\textbf{AET4}{*}) if $\mathcal{C}^{op}$ is \textbf{AET3} (resp. \textbf{AET3.5},
\textbf{AET4}). 
\end{enumerate}
\end{defn}

\begin{prop}\label{prop:2implica3} Let $(\mathcal{C},\E,\s)$ be an AET4($I$)
extriangulated category, $\{A_{i}\}_{i\in I}$ a family of objects in $\C$
and $B\in\mathcal{C}$. Then, the morphism 
\[
\tau:\mathbb{E}(\coprod_{i\in I}A_{i},B)\rightarrow\prod_{i\in I}\mathbb{E}(A_{i},B)\text{, }\eta\mapsto\{\eta\cdot\mu_{i}^{A}\}_{i\in I}
\]
is an isomorphism  of abelian groups whose inverse is 
$$\sigma:\prod_{i\in I}\mathbb{E}(A_{i},B)\to \mathbb{E}(\coprod_{i\in I}A_{i},B),\;\{\eta_i\}_{i\in I}\mapsto \nabla_C\cdot\Gamma^{(1,T)}_S(\{\eta_i\}_{i\in I}).$$
\end{prop}

\begin{proof}
It follows from Proposition \ref{prop:props}(d1). 
\end{proof}

\begin{prop}\label{prop:0implica1}
For an  AET3($I$)  extriangulated category $(\mathcal{C},\E,\s),$  $\D:=\Fun(I,\mathcal{C})$ and the functor $S:=\coprod_{i\in I}(-):\D\rightarrow\mathcal{C},$ the following statements are equivalent.
\begin{itemize}
\item[(a)] $\C$ is AET3.5($I$).
\item[(b)] The functor $S:\D\rightarrow\mathcal{C}$
is extriangulated. 
\end{itemize}
Moreover, if one of the above conditions holds true, then for any family of $\s$-conflations
$\{\eta_{i}:\:\suc[A_{i}][B_{i}][F_{i}][f_{i}][g_{i}]\}_{i\in I}$ in $\C,$
we have that 
\[
\suc[\coprod_{i\in I}A_{i}][\coprod_{i\in I}B_{i}][\coprod_{i\in I}F_{i}][\coprod_{i\in I}f_{i}][\coprod_{i\in I}g_{i}]
\]
is an $\s$-conflation in $\C.$
\end{prop}

\begin{proof} (a) $\Rightarrow$ (b)
Let $T:\C\to\D$ be the constant functor. 
Consider the natural transformation $\varphi:1\rightarrow T\circ S.$ We
recall that $\varphi_{X}$ is defined by the canonical inclusions
$\{\mu_{i}^{X}:X_{i}\rightarrow\coprod_{i\in I}X_{i}\}_{i\in I}$
for all $X\in\D.$

Let us consider the
natural transformation 
\[
\mathbb{F}\cdot(1\times\varphi):\mathbb{F}(?,-)\rightarrow\mathbb{F}(?,TS(-)).
\]
Since $S$ is $(1,T)$-extriangulated, there is
a natural transformation 
\[
\Gamma_{S}^{(1,T)}:\mathbb{F}(?,T(-))\rightarrow\mathbb{E}(S(?),ST(-))
\]
such that $\Gamma_{S}^{(1,T)}(\eta)$ is realized by $[\suc[STX][SY][SZ][Sa][Sb]]$
for every $\t$-conflation $\eta:\:\suc[TX][Y][Z][a][b]$ in $\D.$ Now, consider
\[
\Gamma:=\Gamma_{S}^{(1,T)}\circ(\mathbb{F}\cdot(1\times\varphi)):\mathbb{F}(?,-)\rightarrow\mathbb{E}(S(?),STS(-)).
\]
Let us examine how $\Gamma$ maps the realizations. For this, consider
a $\t$-conflation $\eta:\:\suc[A][B][C][f][g]$ in $\D.$ Then, by Lemma \ref{lemita} (a), we get that $\varphi_{A}\cdot\eta=(\mu_{i}^{A}\cdot\eta_{i})_{i\in I}$
is realized by 
\begin{equation}
\left\{ \left[A_{i}\amalg(\coprod_{j\neq i}A_{j})\overset{\left[\begin{smallmatrix}f_{i} & 0\\
0 & 1
\end{smallmatrix}\right]}{\rightarrow}B_{i}\amalg(\coprod_{j\neq i}A_{j})\overset{\left[\begin{smallmatrix}g_{i} & 0\end{smallmatrix}\right]}{\rightarrow}F_{i}\right]\right\} _{i\in I}\label{eq:}
\end{equation}
Then, by using (\ref{eq:}), it follows that
$\Gamma(\eta)=\Gamma_{S}^{(1,T)}(\varphi_{A}\cdot\eta)$ is realized
by 
\[
\left[(\coprod_{i\in I}A_{i})\amalg(\coprod_{i\in I}(\coprod_{j\neq i}A_{j}))\overset{\left[\begin{smallmatrix}\coprod_{i\in I}f_{i} & 0\\
0 & 1
\end{smallmatrix}\right]}{\rightarrow}(\coprod_{i\in I}B_{i})\amalg(\coprod_{i\in I}(\coprod_{j\neq i}A_{j}))\overset{\left[\begin{smallmatrix}\coprod_{i\in I}g_{i} & 0\end{smallmatrix}\right]}{\rightarrow}\coprod_{i\in I}C_{i}\right].
\]
Observe now that there is a natural transformation $\pi:STS\rightarrow S$
defined as follows: for $X\in\D$, 
\[
\pi_{X}:\coprod_{i\in I}\left(\coprod_{i\in I}X_{i}\right)\rightarrow\coprod X_{i}
\]
is the morphism satisfying that $\pi_{X}\circ\mu_{i}^{TSX}=\mu_{i}^{X}\circ\pi_{i}^{X}$.
In other words, rearranging $\coprod_{i\in I}\left(\coprod_{i\in I}X_{i}\right)$
as $(\coprod_{i\in I}X_{i})\amalg\left(\coprod_{i\in I}\coprod_{j\neq i}X_{i}\right)$,
$\pi_{X}$ is the morphism 
\[
\left[\begin{smallmatrix}1 & 0\end{smallmatrix}\right]:(\coprod_{i\in I}X_{i})\amalg(\coprod_{i\in I}\coprod_{j\neq i}X_{i})\rightarrow(\coprod_{i\in I}X_{i}).
\]
Thus by Lemma \ref{lemita}, we conclude that $((\mathbb{E}\circ(S\times STS))\cdot(1\times\pi))\circ\Gamma$
defines the natural transformation 
\[
\Gamma_{S}:\mathbb{F}(?,-)\rightarrow\mathbb{E}(S(?),S(-))
\]
such that $\Gamma_{S}(\eta)$ is realized by $[\suc[SX][SY][SZ][Sa][Sb]]$
for every $\t$-conflation $\eta$ with realization $[\suc[X][Y][Z][a][b]]$ in $\D.$
\

(b) $\Rightarrow$ (a) It follows from Example \ref{ExA} (c).
\end{proof}

\begin{example}\label{exa:ab}$ $ 
\begin{enumerate}
\item If $\mathcal{A}$ is an AET3 exact category with exact coproducts,
then the derived category $\mathcal{D}(\mathcal{A})$ is an AET3 triangulated
category (see \cite[Lem.4.1.15]{K}). 

\item Every AET3 triangulated category $(\mathcal{C},\Sigma,\triangle)$
is AET4. Indeed, by \cite[Prop.1.1.6]{N}, for every set $I\neq\emptyset$ there
is a natural isomorphism 
\[
\upsilon:\coprod_{i\in I}\circ\overline{\Sigma}\rightarrow\Sigma\circ\coprod_{i\in I},
\]
where $\overline{\Sigma}:\D\rightarrow\D$
is the functor satisfying that $(-)_{i}\circ\overline{\Sigma}=\Sigma\circ(-)_{i}$
for all $i\in I$. Hence, there is a natural transformation 
\[
\Gamma_S:\prod_{i\in I}\Hom_{\mathcal{C}}((?)_{i},\Sigma(-)_{i})\rightarrow\Hom_{\mathcal{C}}(\coprod_{i\in I}(?)_{i},\Sigma\coprod_{i\in I}(-)_{i})
\]
defined as the composition 
\[
\prod_{i\in I}{}_{\mathcal{C}}((?)_{i},\Sigma(-)_{i})\overset{\coprod_{i\in I}}{\rightarrow}{}_{\mathcal{C}}(\coprod_{i\in I}(?)_{i},\coprod_{i\in I}\Sigma(-)_{i})\overset{\Hom\cdot(1\times\upsilon)}{\rightarrow}{}_{\mathcal{C}}(\coprod_{i\in I}(?)_{i},\Sigma\coprod_{i\in I}(-)_{i}),
\]
where $_{\mathcal{C}}(-,?):=\Hom_{\mathcal{C}}(-,?)$. Moreover, by
the dual of \cite[Prop.1.2.1]{N}, for every $(h_{i})_{i\in I}\in\prod_{i\in I}\Hom_{\mathcal{C}}(A{}_{i},\Sigma B{}_{i})$
realized by 
\[
\{[\suc[B_{i}][C_{i}][A_{i}][f_{i}][g_{i}]\overset{h_{i}}{\rightarrow}\Sigma B{}_{i}]\}_{i\in I},
\]
we have that $h:=\Gamma_S(h_{i})_{i\in I}$ is realized by 
\[
[\suc[\coprod_{i\in I}B_{i}][\coprod_{i\in I}C_{i}][\coprod_{i\in I}A_{i}][\coprod_{i\in I}f_{i}][\coprod_{i\in I}g_{i}]\overset{\Gamma_S(h_{i})}{\rightarrow}\Sigma\coprod_{i\in I}B{}_{i}].
\]
Therefore, $S=\coprod_{i\in I}(-):\mathcal{D}\rightarrow\mathcal{C}$
is extriangulated for every non-empty set $I.$ Finally, from the commutativity of the diagram below
\[\xymatrix{
A_i\ar[d]_{\mu^{A}_i}\ar[rr]^{h_i} && \Sigma B_i\ar[d]^{\Sigma\mu^{B}_i}\\
\coprod_{i\in I}A_i\ar[rr]^h && \Sigma\coprod_{i\in I}B_i,
}\]
it follows that $\varphi:1\rightarrow T\circ S$ is $(1,T)$-extriangulated.

\item If $\mathcal{A}$ is an AET3 extriangulated category and $(\mathcal{T},\mathcal{F})$
is an $s$-torsion pair in $\mathcal{A}$, then $\mathcal{T}$ is
an AET3 extriangulated category. 
\end{enumerate}
\end{example}

\begin{lem}\label{prop:0implica1-1} Let $(\mathcal{C},\E,\s)$
be an AET4($I$) extriangulated category, $\D:=\Fun(I,\C),$ $S:=\coprod_{i\in I}(-):\D\to\C$ and $T:\C\to \D$ be the constant functor. Then, the natural transformations
$\varphi:1\rightarrow T\circ S$ and $\psi:S\circ T\rightarrow1$
are extriangulated, where $\varphi_{X}$ is defined by the canonical
inclusions $\{\mu_{i}^{X}:X_{i}\rightarrow\coprod_{i\in I}X_{i}\}_{i\in I}$
for all $X\in\Fun(I,\mathcal{C})$ and $\psi_{A}$ is the co-diagonal
morphism $\nabla:A^{(I)}\rightarrow A$.
\end{lem}

\begin{proof}
In the proof of Proposition \ref{prop:0implica1},
it was shown that there is a natural transformation 
$
\Gamma_{S}:\mathbb{F}(?,-)\rightarrow\mathbb{E}(S(?),S(-))
$
such that $\Gamma_{S}(\eta)$ is realized by $[\suc[SX][SY][SZ][Sa][Sb]]$
for every $\mathbb{F}$-extension $\eta$ with realization $[\suc[X][Y][Z][a][b]]$.
Moreover, for $A,C\in\D$ and $\eta\in\mathbb{F}(C,A)$,
it was shown that $\Gamma_{S}(\eta)=\pi_{C}\cdot\Gamma_{S}^{(1,T)}(\mu_{i}^{A}\cdot\eta)_{i\in I}$,
where $\pi_{A}:\coprod_{i\in I}\left(\coprod_{i\in I}A_{i}\right)\rightarrow\coprod_{i\in I} A_{i}$
is the morphism satisfying that $\pi_{A}\circ\mu_{i}^{TSA}=\mu_{i}^{A}\circ\pi_{i}^{A}$.
Let us prove that $\Gamma_{S}(\eta)\cdot\mu_{i}^{C}=\mu_{i}^{A}\cdot\eta$
for all $i\in I$. For this, observe that 
\begin{alignat*}{1}
\Gamma_{S}(\eta)\cdot\mu_{i}^{C} & =\pi_{A}\cdot\Gamma_{S}^{(1,T)}(\mu_{i}^{A}\cdot\eta)_{i\in I}\cdot\mu_{i}^{C}\\
 & =\pi_{A}\cdot\mu_{i}^{TSA}\cdot\mu_{i}^{A}\cdot\eta\\
 & =\mu_{i}^{A}\cdot\pi_{i}^{A}\cdot\mu_{i}^{A}\cdot\eta\\
 & =\mu_{i}^{A}\cdot\eta.
\end{alignat*}
Lastly, it follows from Proposition \ref{prop:props}(d3) that $\psi$
is extriangulated. 
\end{proof}

Now, we are ready to state and prove other of the main results of the paper.

\begin{thm}\label{MainTB}
Let $(\mathcal{C},\E,\s)$ be an essentially small
AET4($I$) extriangulated category. Then, for any $n\geq 1,$ the maps 
\begin{alignat*}{1}
\tau_{D,C}^{n}:\mathbb{E}^{n}\Big(\coprod_{i\in I}D_{i},C\Big) & \stackrel{}{\rightarrow}\prod_{i\in I}\mathbb{E}^{n}(D_{i},C)\text{, }\eta\mapsto(\eta\cdot\mu_{i}^{D})_{i\in I}\text{ and }\\
\sigma_{D,C}^{n}:\prod_{i\in I}\mathbb{E}^{n}(D_{i},C) & \stackrel{}{\rightarrow}\mathbb{E}^{n}\Big(\coprod_{i\in I}D_{i},C\Big)\text{, }(\eta_{i})_{i\in I}\mapsto\nabla_{C}\cdot\Gamma_{S}^{n}(\eta_{i})_{i\in I}
\end{alignat*}
define natural isomorphisms such that $\tau^{n}$ is the inverse of
$\sigma^{n}$. 
\end{thm}

\begin{proof} Let $\D:=\Fun(I,\C).$ 
Consider the functors $S:=\coprod_{i\in I}(-):\D\rightarrow\mathcal{C}$ and $T:\mathcal{C}\rightarrow\D.$ We
know that $T$ is extriangulated, and it follows from Proposition
\ref{prop:0implica1} that $S$ is extriangulated. Moreover,
by Lemma \ref{prop:0implica1-1}, the natural transformations
$\varphi:1\rightarrow T\circ S$ and $\psi:S\circ T\rightarrow1$
are extriangulated. Therefore, by Proposition \ref{prop:higher},
we get that $\tau^{n}$ and $\sigma^{n}$ are natural isomorphisms
inverse to each other. 
\end{proof}

For the sake of completeness (and to have it available for the reader) we state the dual version of Theorem \ref{MainTB2} which is also true.

\begin{thm}\label{MainTB2}
Let $(\mathcal{C},\E,\s)$ be an essentially small
AET4{*}($I$) extriangulated category. Then, for any $n\geq 1,$ the maps 
\begin{alignat*}{1}
\tau_{C,D}^{n}:\mathbb{E}^{n}\Big(C,\prod_{i\in I}D_{i}\Big) & \stackrel{}{\rightarrow}\prod_{i\in I}\mathbb{E}^{n}(C,D_{i})\text{, }\eta\mapsto(\pi_{i}^{D}\cdot\eta)_{i\in I}\text{ and }\\
\sigma_{C,D}^{n}:\prod_{i\in I}\mathbb{E}^{n}(C,D_{i}) & \stackrel{}{\rightarrow}\mathbb{E}^{n}\Big(C,\prod_{i\in I}D_{i}\Big)\text{, }(\eta_{i})_{i\in I}\mapsto \Gamma_{P}^{n}(\eta_{i})_{i\in I}\cdot\Delta_{C}
\end{alignat*}
define natural isomorphisms such that $\tau^{n}$ is the inverse of
$\sigma^{n}$ and $P:=\prod_{i\in I}(-):\Fun(I,\C)\to \C.$
\end{thm}

\begin{prop}\label{prop:1implica2}
For an AET3.5 extriangulated category $(\mathcal{C},\E,\s),$ the following statements hold true. 
\begin{enumerate}
\item For every non-empty family of $\s$-conflations $\{\eta_{i}:\:\suc[B_{i}][E_{i}][A_{i}][f_{i}][g_{i}]\}_{i\in I},$ we have that 
 $\eta:\:\suc[\coprod_{i\in I}B_{i}][\coprod_{i\in I}E_{i}][\coprod_{i\in I}A_{i}][\coprod_{i\in I}f_{i}][\coprod_{i\in I}g_{i}]$ is an $\s$-conflation 
and, for each $j\in I,$ there are two isomorphisms in $\C$
\[u_{j}:\coprod_{i\in I}B_{i}\xrightarrow{\sim}\coprod_{i\in I}B_{i}\quad\text{ and }v_{j}:\coprod_{i\in I}B_{i}\xrightarrow{\simeq}\coprod_{i\in I}B_{i}
\]
such that $(u_j\cdot\eta)\cdot\mu_{j}^{A}=(\eta\cdot v_{j})\cdot\mu_{j}^{A}=\mu_{j}^{B}\cdot\eta_{}.$

\item If every $\s$-inflation is a monomorphism (or every $\s$-deflation is an epimorphism),
then $\mathcal{C}$ is AET4. 
\end{enumerate}
\end{prop}

\begin{proof}
Let $I$ be a non-empty set and $\{\eta_{i}:\:\suc[B_{i}][E_{i}][A_{i}][f_{i}][g_{i}]\}_{i\in I}$ be a family of $\s$-conflations. By Proposition \ref{prop:0implica1} we have that 
$$\eta:\:\suc[\coprod_{i\in I}B_{i}][\coprod_{i\in I}E_{i}][\coprod_{i\in I}A_{i}][\coprod_{i\in I}f_{i}][\coprod_{i\in I}g_{i}]$$ 
is an $\s$-conflation. 
\

Consider $i_{0}\in I.$ Let us show that (a) holds true for $j=i_0$ and that $\C$ is AET4$(I).$ Notice that if $I=\{i_0\}$ then there is nothing to prove. Thus, we can assume that $I_{0}:=I-\{i_{0}\}\neq \emptyset.$
Observe that, if an object $X\in\mathcal{C}$ is a coproduct $\coprod_{i\in I}X_{i}$,
then $X=X_{1}\amalg X_{2}$ with $X_{1}=X_{i_{0}}$ and $X_{2}=\coprod_{i\in I_{0}}X_{i}$.
Moreover, the inclusion $\mu'_{1}:X_{1}\rightarrow X$ of the coproduct
$X_{1}\amalg X_{2}$ is equal to the inclusion $\mu{}_{i_{0}}^{X}:X_{1}\rightarrow X$
of the coproduct $\coprod_{i\in I}X_{i}$.

Consider the constant functors $T:\mathcal{C}\rightarrow\Fun(I,\mathcal{C})$
and $T':\mathcal{C}\rightarrow\Fun(I_{0},\mathcal{C})$, together
with the functors $S=\coprod_{i\in I}(-)_{i}:\Fun(\mathcal{C},I)\rightarrow\mathcal{C}$
and $S'=\coprod_{i\in I_{0}}(-)_{i}:\Fun(\mathcal{C},I_{0})\rightarrow\mathcal{C}$.
Since $\C$ is AET3.5, by Proposition \ref{prop:0implica1}, there is a natural transformation
\[
\Gamma_{S}:\mathbb{F}(?,-)\rightarrow\mathbb{E}(S(?),S(-))
\]
such that $\Gamma_{S}(\theta)$ is realized by $[\suc[SX][SY][SZ][Sa][Sb]]$
for every $\mathbb{F}$-extension $\theta$ with realization $[\suc[X][Y][Z][a][b]]$.
Similarly, there is a natural transformation 
\[
\Gamma_{S'}:\mathbb{F}'(?,-)\rightarrow\mathbb{E}(S'(?),S'(-))
\]
with a similar property, where $\mathbb{F}'(?,-)=\prod_{i\in I_{0}}\mathbb{E}((?)_{i},(-)_{i})$.

Now, for the family of $\s$-conflations $\{\eta_{i}:\:\suc[B_{i}][E_{i}][A_{i}][f_{i}][g_{i}]\}_{i\in I}$,
we have 
\[
\s(\Gamma_{S}(\eta_{i})_{i\in I})=\left[\suc[\coprod_{i\in I}B_{i}][\coprod_{i\in I}E_{i}][\coprod_{i\in I}A_{i}][\coprod_{i\in I}f_{i}][\coprod_{i\in I}g_{i}]\right]
\]
and 
\[
\s(\Gamma_{S'}(\eta_{i})_{i\in I_{0}})=\left[\suc[\coprod_{i\in I_{0}}B_{i}][\coprod_{i\in I_{0}}E_{i}][\coprod_{i\in I_{0}}A_{i}][\coprod_{i\in I_{0}}f_{i}][\coprod_{i\in I_{0}}g_{i}]\right].
\]
Observe that the $\s$-conflation $\suc[\coprod_{i\in I}B_{i}][\coprod_{i\in I}E_{i}][\coprod_{i\in I}A_{i}][\coprod_{i\in I}f_{i}][\coprod_{i\in I}g_{i}]$
can be expressed as 
\[
\suc[B_{i_{0}}\amalg(\coprod_{i\in I_{0}}B_{i})][E_{i_{0}}\amalg(\coprod_{i\in I_{0}}E_{i})][A_{i_{0}}\amalg(\coprod_{i\in I_{0}}A_{i})][f_{i_{0}}\amalg(\coprod_{i\in I_{0}}f_{i})][g_{i_{0}}\amalg(\coprod_{i\in I_{0}}g_{i})].
\]
Let $\eta':=\Gamma_{S'}(\eta_{i})_{i\in I_{0}}.$ Then, by Lemma \ref{lemita2} (a) there are isomorphisms 
\[u_{i_0}:\coprod_{i\in I}B_{i}\xrightarrow{\sim}\coprod_{i\in I}B_{i}\quad\text{ and }v_{i_0}:\coprod_{i\in I}B_{i}\xrightarrow{\sim}\coprod_{i\in I}B_{i}
\]
such that $u_{i_0}\cdot \eta=\left[\begin{smallmatrix} \eta_{i_0} & 0\\0 & \eta'  \end{smallmatrix}\right]=\eta\cdot v_{i_0}.$ Therefore, by Lemma \ref{FtoM}, we have
$$\mu^{B}_{i_0}\cdot\eta_{i_0}=\left[\begin{smallmatrix} 1 \\0 \end{smallmatrix}\right]\eta_{i_0}=\left[\begin{smallmatrix} \eta_{i_0} \\0 \end{smallmatrix}\right]=
\left[\begin{smallmatrix} \eta_{i_0} &0\\0 &\eta' \end{smallmatrix}\right]\cdot\mu^{A}_{i_0}$$
and thus  
$(u_{i_0}\cdot \eta)\cdot\mu^{A}_{i_0}=\mu^{B}_{i_0}\cdot\eta_{i_0}=(\eta\cdot v_{i_0})\cdot\mu^{A}_{i_0}.$ Therefore, we get (a). Finally, by considering the case where $B_i=B$ $\forall\,i\in I$ and assuming that the hypothesis in (b) holds true, then from the above and Lemma \ref{lemita2} (b) we get that $\C$ is AET4.
\end{proof}

\subsection{Universal $\mathbb{E}$-extensions}
\begin{defn}
\cite[Def.5.7]{AMP}\label{def:univ} Let $(\mathcal{C},\E,\s)$ be
an extriangulated category and $V,D\in\mathcal{C}$. A \textbf{universal
$\mathbb{E}$-extension} of $V$ by $D$ is an 
$\eta\in\mathbb{E}(V^{(X)},D)$, for some set $X\neq\emptyset$, such
that $\eta\cdot-:\text{Hom}_{\mathcal{C}}(V,V^{(X)})\to\E(V,D),\;f\mapsto \eta\cdot f,$
is surjective. 
\end{defn}

\begin{lem}\label{prop:3implica4} Let $(\mathcal{C},\E,\s)$ be an AET3 extriangulated
category and $V,D\in\mathcal{C}$. If the natural transformation
$
\tau:\mathbb{E}(V^{(X)},D)\rightarrow\prod_{x\in X}\mathbb{E}(V,D)\text{, }\eta\mapsto(\eta\cdot\mu_{i}^{V^{(X)}})_{i\in I}
$
is surjective for any set $X\neq\emptyset$, then there is a universal $\mathbb{E}$-extension
of $V$ by $D$. 
\end{lem}

\begin{proof}
Let $\E(V,D)=\{\eta_{i}\}_{i\in I}.$ Since 
$\tau:\mathbb{E}(V^{(I)},D)\rightarrow\prod_{i\in I}\mathbb{E}(V,D)$ is surjective, 
there exists $\eta\in\mathbb{E}(V^{(I)},D)$
such that $\eta\cdot\mu_{i}^{V^{(I)}}=\eta_{i}$ for all $i\in I.$ Therefore, we get that 
 $\eta\cdot-:\text{Hom}_{\mathcal{C}}(V,V^{(I)})\to\E^{1}(V,D)$
is surjective. 
\end{proof}

The connection between the universal $\mathbb{E}$-extensions and the natural transformation 
$\tau:\mathbb{E}(\coprod_{i\in I}(?)_i,-)\rightarrow\prod_{i\in I}\mathbb{E}((?)_i,-)$
from Proposition \ref{prop:props} (b) is as follows.
\begin{prop}\label{prop:tauiso} Let $(\mathcal{C},\E,\s)$ be an AET3 extriangulated
category. Then, the following statements are equivalent. 
\begin{itemize}
\item[(a)] For every $V,D\in\mathcal{C}$, there is a universal
$\mathbb{E}$-extension of $V$ by $D.$
\item[(b)]  $\tau:\mathbb{E}\Big(\coprod_{i\in I}(?)_i,-\Big)\xrightarrow{\sim}\prod_{i\in I}\mathbb{E}((?)_i,-)$ for any set $I\neq\emptyset.$
\end{itemize}  
\end{prop}

\begin{proof}
(a) $\Rightarrow$ (b) By Proposition \ref{prop:props} (b), it is enough to show that $\tau$
is epic. For this, consider $(\eta_{i})_{i\in I}\in\prod_{i\in I}\mathbb{E}(B_{i},A)$
and choose a realization $\eta_{i}:\suc[A][E_{i}][B_{i}][f_{i}][g_{i}]$
for all $i\in I$. By hypothesis, there is a universal extension $\eta:\:\suc[A][E][B^{(Y)}][a][b]$,
where $B=\coprod_{i\in I}B_{i}.$ Hence, for each $i\in I$, there
exists $u_{i}\in B\rightarrow B^{(Y)}$ such that $\eta\cdot u_{i}=\eta_{i}\cdot\pi_{i}^{B}$,
where $\pi_{i}^{B}:B\rightarrow B_{i}$ is the natural projection.
Moreover, by the universal property of coproducts, there is $u\in\Hom_{\mathcal{C}}(B,B^{(Y)})$
such that $u\circ\mu_{i}^{B}=u_{i}\circ\mu_{i}^{B}$ $\forall i\in I$.
Then $\eta\cdot u\in\mathbb{E}(B,A)$ is an $\mathbb{E}$-extension
such that 
\[
\tau(\eta\cdot u)=(\eta\cdot(u\circ\mu_{i}^{B}))_{i\in I}=((\eta\cdot u_{i})\cdot\mu_{i}^{B})_{i\in I}=(\eta_{i}\cdot(\pi_{i}^{B}\circ\mu_{i}^{B}))_{i\in I}=(\eta_{i})_{i\in I},
\]
and thus we get that $\tau$ is epic.
\

(b) $\Rightarrow$ (a) It follows from Lemma \ref{prop:3implica4}.
\end{proof}

\subsection{AET4 for coproduct-compatible extriangulated categories with negative
first extension}

\begin{defn}\label{def:cc} Let $(\mathcal{C},\E,\s)$ be an extriangulated
category. We will say that $\mathcal{C}$
is \textbf{coproduct-compatible (}resp.\textbf{ product-compatible)}
if: (1) $\C$ is an AET3 (resp. AET$3^{*}$)  extriangulated category with negative first extension, and (2) for any object $X\in\mathcal{C},$ any set $I\neq\emptyset$ and any family $\{Y_{i}\}_{i\in I}$ of objects in $\mathcal{C}$,
we have that the natural transformation 
\begin{alignat*}{1}
 & \tilde{\tau}:\mathbb{E}^{-1}(\coprod_{i\in I}Y_{i},X)\rightarrow\prod_{i\in I}\mathbb{E}^{-1}(Y_{i},X)\text{, }\chi\mapsto\left(\mathbb{E}^{-1}(\mu_{i}^{Y},X)(\chi)\right)_{i\in I}\\
\text{(resp.} & \tilde{\tau}:\mathbb{E}^{-1}(X,\prod_{i\in I}Y_{i})\rightarrow\prod_{i\in I}\mathbb{E}^{-1}(X,Y_{i})\text{, }\chi\mapsto\left(\mathbb{E}^{-1}(X,\pi_{i}^{Y})(\chi)\right)_{i\in I})
\end{alignat*}
is an epimorphism. 
\end{defn}

\begin{example}\label{exCCETC} We have the following examples of extriangulated categories which are coproduct-compatible.
\

(1) Let $(\mathcal{C},\mathbb{E},\s)$ be an AET3 extriangulated category
with negative first extension. If $\mathbb{E}^{-1}=0$ (e.g. $\mathcal{C}$ is an exact category)  then $\mathcal{C}$ is coproduct-compatible. 
\

(2) Let $(\mathcal{C},\Sigma,\triangle)$ be a triangulated category with coproducts. 
Then, $\mathcal{C}$ is a coproduct-compatible extriangulated category.\\
Indeed, since $\mathbb{E}^{-1}(Y,X)=\Hom_{\mathcal{C}}(Y,\Sigma^{-1}X),$
we get that $\tilde{\tau}$ is an isomorphism. 
\

(3) Let $(\mathcal{D},\Sigma,\triangle)$ be a triangulated category with coproducts, and let $\te_{1}=(\mathcal{T}_{1},\mathcal{F}_{1})$ and $\te_{2}=(\mathcal{T}_{2},\mathcal{F}_{2})$ in $\stors(\mathcal{D})$. Then the heart $\mathcal{C}:=\mathcal{H}_{[\te_{1},\te_{2}]}$ is a coproduct-compatible extriangulated category.\\
 Indeed, it is clear that $\C$ is an extriangulated category with negative first extension. To show that $\C$ has coproducts, observe that the coproduct
in $\mathcal{C}$ for a family $\{C_{i}\}_{i\in I}$ in $\mathcal{C}$
is given by $\coprod_{i\in\mathcal{C}}^{\mathcal{C}}C_{i}=(1:\te_{1})\left(\coprod_{i\in I}^{\mathcal{D}}C_{i}\right),$ where $\coprod_{i\in I}^{\mathcal{D}}C_{i}$ is the coproduct in $\D.$
This follows from the fact that the map $X\mapsto(1:\te_{1})X$ defines
a left adjoint of the canonical inclusion $\mathcal{C}\rightarrow\mathcal{T}_{2}$
(see Remarks \ref{rem:adjuntos} and \ref{rem:L}(a)). Moreover, the
canonical inclusions of the coproduct $\coprod_{i\in\mathcal{C}}^{\mathcal{C}}C_{i}$
are the morphisms $\{y_{C}\circ\mu_{i}^{C}\}_{i\in I}$, where $\mu_{i}^{C}:C_{i}\rightarrow\coprod_{i\in I}^{\mathcal{D}}C_{i}$
is the $i$-th canonical inclusion in $\D$ and $y_{C}$
is the morphism appearing in the canonical deflation 
\[
\suc[\te_{1}\left(\coprod_{i\in\mathcal{C}}^{\mathcal{D}}C_{i}\right)][\coprod_{i\in\mathcal{C}}^{\mathcal{D}}C_{i}][(1:\te_{1})\left(\coprod_{i\in\mathcal{C}}^{\mathcal{D}}C_{i}\right)][x_{C}][y_{C}]
\]
induced by the $\s$-torsion pair $(\mathcal{T}_{1},\mathcal{F}_{1})$. Hence, in this context,
$\tilde{\tau}$ is the natural transformation
\[
\Hom_{\mathcal{\mathcal{D}}} \Big((1:\te_{1})\left(\coprod_{i\in I}^{\mathcal{D}}C_{i}\right),\Sigma^{-1}X\Big)\rightarrow\prod_{i\in I}\Hom_{\mathcal{\mathcal{D}}}(C_{i},\Sigma^{-1}X)\text{, }
\]
given by $\chi\mapsto\left(\chi\circ y_{C}\circ\mu_{i}^{C}\right)_{i\in I}.$
Note that this natural transformation is the composition 
\[
_{\mathcal{D}}\left((1:\te_{1})\left(\coprod_{i\in I}^{\mathcal{D}}C_{i}\right),\Sigma^{-1}X\right)\rightarrow{}{}_{\mathcal{D}}\left(\left(\coprod_{i\in I}^{\mathcal{Y}}C_{i}\right),\Sigma^{-1}X\right)\rightarrow\prod_{i\in I}{}{}_{\mathcal{D}}\left(C_{i},\Sigma^{-1}X\right)
\]
given by $\chi\mapsto\chi\circ y_{C}\mapsto\left(\chi\circ y_{C}\circ\mu_{i}^{C}\right)_{i\in I}$.
It is well-known that the natural transformation on the right is an
isomorphism. Moreover, the natural transformation on the left is surjective
since $\Sigma^{-1}X\in\Sigma^{-1}\mathcal{C}\subseteq\mathcal{F}_{1}$.
Therefore, $\tilde{\tau}$ is an epimorphism. 
\end{example}

\begin{thm}\label{thm:tauvsAB4} For a coproduct-compatible
extriangulated category $(\mathcal{C},\E,\s),$ the following statements are equivalent.
\begin{itemize}
\item[(a)] For any set $I\neq\emptyset$ the natural transformation (see Proposition \ref{prop:props} (b))
$$\tau:\mathbb{E}\Big(\coprod_{i\in I}(?)_i,-\Big)\rightarrow\prod_{i\in I}\mathbb{E}((?)_i,-)$$ is an isomorphism.
\item[(b)] $\C$ is AET4.
\end{itemize} 
\end{thm}

\begin{proof}
(a) $\Rightarrow$ (b)
Let us show that the condition AET3.5 holds true. For this,
consider the natural transformation 
$
\tau^{-1}:\prod_{i\in I}\mathbb{E}(B_{i},A^{(I)})\rightarrow\mathbb{E}(\coprod_{i\in I}B_{i},A^{(I)}),
$
together with the natural transformation 
\[
\omega:=\prod_{i\in I}\left(\mathbb{E}(B_{i},\mu_{i}^{A})\right):\prod_{i\in I}\mathbb{E}(B_{i},A)\rightarrow\prod_{i\in I}\mathbb{E}(B_{i},A^{(I)})\text{, }(\eta_{i})_{i\in I}\mapsto(\mu_{i}^{A}\cdot\eta_{i})_{i\in I},
\]
where $\{\mu^{A}_i:A\to A^{(I)}\}_{i\in I}$ are the canonical inclusions of the coproduct $A^{(I)}.$ We claim that $\Gamma:=\tau^{-1}\circ\omega$ is the natural transformation
satisfying the AET3.5 condition. Indeed, for $(\eta_{i})_{i\in I}\in\prod_{i\in I}\mathbb{E}(B_{i},A)$,
we know that $(\tau^{-1}\circ\omega)(\eta_{i})_{i\in I}$ is an $\mathbb{E}$-extension
$\eta\in\mathbb{E}(\coprod_{i\in I}B_{i},A^{(I)})$ such that 
\begin{equation}
(\eta\cdot\mu_{i}^{B})_{i\in I}=\tau(\eta)=\omega(\eta_{i})_{i\in I}=(\mu_{i}^{A}\cdot\eta_{i})_{i\in I},\label{eq:AB4}
\end{equation}
where $\{\mu_{i}^{B}:B_{i}\rightarrow\coprod_{i\in I}B_{i}\}_{i\in I}$
are the canonical inclusions of $\coprod_{i\in I}B_{i}.$ In particular, for all $i\in I$, we
have a morphism of $\mathbb{E}$-extensions $(\mu_{i}^{A},\mu_{i}^{B}):\eta_{i}\rightarrow\eta$.
Let us consider the following realization of such morphism: 
\[
\xymatrix{A\ar[r]^{f_{i}}\ar[d]^{\mu_{i}^{A}} & E_{i}\ar[r]^{g_{i}}\ar[d]^{\theta_{i}} & B_{i}\ar[d]^{\mu_{i}^{B}}\\
A^{(I)}\ar[r]^{f} & E\ar[r]^{g} & \coprod_{i\in I}B_{i}.
}
\]
 We claim that the $\s$-conflation $\suc[A^{(I)}][E][\coprod_{i\in I}B_{i}][f][g]$
is equivalent to the sequence $\suc[A^{(I)}][\coprod_{i\in I}E_{i}][\coprod_{i\in I}B_{i}][\coprod_{i\in I}f_{i}][\coprod_{i\in I}g_{i}]$.
Indeed, by the universal property of coproducts 
$\exists\,\theta:\coprod_{i\in I}E_{i}\rightarrow E$ such that $\theta\circ\mu_{i}^{E}=\theta_{i}$ $\;\forall\,i\in I.$
Observe that $\theta$ makes the diagram below commute since 
\begin{alignat*}{1}
g\circ\theta\circ\mu_{i}^{E} & =g\circ\theta_{i}=\mu_{i}^{B}\circ g=(\coprod_{i\in I}g_{i})\circ\mu_{i}^{E}\text{\ensuremath{\quad\forall i\in I}}\text{ and }\\
\theta\circ(\coprod_{i\in I}f_{i})\circ\mu_{i}^{A} & =\theta\circ\mu_{i}^{E}\circ f_{i}=\theta_{i}\circ f_{i}=f\circ\mu_{i}^{A}\text{\ensuremath{\quad\forall i\in I}}.
\end{alignat*}
 
\[
\xymatrix{A\ar[r]^{f_{i}}\ar[d]^{\mu_{i}^{A}} & E_{i}\ar[r]^{g_{i}}\ar[d]^{\mu_{i}^{E}} & B_{i}\ar[d]^{\mu_{i}^{B}}\\
A^{(I)}\ar[r]^{\coprod f_{i}}\ar@{=}[d] & \coprod_{i\in I}E_{i}\ar[r]^{\coprod g_{i}}\ar[d]^{\theta} & \coprod_{i\in I}B_{i}\ar@{=}[d]\\
A^{(I)}\ar[r]^{f} & E\ar[r]^{g} & \coprod_{i\in I}B_{i}.
}
\]
We proceed as follows to show that $\theta$ is an isomorphism. Observe (see in Section \ref{pcofexts}) that $(\D,\mathbb{F},\t)$ is an extriangulated category with negative first extension since $(\C,\mathbb{E},\s)$ is so. Therefore, for any $X\in\mathcal{C}$, we have the following commutative
diagram with exact rows:
\[
\xymatrix{\mathbb{E}^{-1}(\coprod A_{i},X)\ar[r]^{\Delta}\ar[d]^{\tilde{\tau}} & (\coprod B_{i},X)\ar[r]^{{}(g,X)}\ar[d]^{(\mu_{i}^{B},X)_{i\in I}} & (E,X)\ar[r]^{{}(f,X)}\ar[d]^{(\theta_{i},X)_{i\in I}} & (\coprod A_{i},X)\ar[r]^{{}-\cdot\eta}\ar[d]^{(\mu_{i}^{A},X)_{i\in I}} & \mathbb{E}(\coprod B_{i},X)\ar[d]^{\tau}\\
\prod\mathbb{E}^{-1}(A_{i},X)\ar[r]^{\prod\Delta} & \prod_{i\in I}(B_{i},X)\ar[r]_{{}\prod_{i\in I}(g_{i},X)} & \prod(E_{i},X)\ar[r]_{{}\prod(f_{i},X)} & \prod(A_{i},X)\ar[r]_{{}\prod(-\cdot\eta_{i})} & \prod\mathbb{E}(B_{i},X)
}
\]
We have in the above diagram that the first column is an epimorphism. Moreover, the second, forth and
fifth columns are isomorphisms. Hence, by the Five Lemma we get that $(\theta_{i},X)_{i\in I}$ is an isomorphism. This means that $E$ is a coproduct
of $\{E_{i}\}_{i\in I}$ and that the morphisms $\{\theta_{i}:E_i\to E\}_{i\in I}$
are the canonical inclusions. In particular, for $X=\coprod_{i\in I}E_{i}$,
there is a unique morphism $\nu:E\rightarrow\coprod_{i\in I}E_{i}$
such that $\nu\circ\theta_{i}=\mu_{i}^{E}$ for all $i\in I$. Moreover,
by the universal property of coproducts, we have that $\nu\circ\theta=1_{\coprod_{i\in I}E_{i}}$
and that $\theta\circ\nu=1_{E}$. Hence, $\theta$ is an isomorphism.
Therefore $\eta$ is realized by $\suc[A^{(I)}][\coprod_{i\in I}E_{i}][\coprod_{i\in I}B_{i}][\coprod_{i\in I}f_{i}][\coprod_{i\in I}g_{i}]$
as desired. Lastly, note that the AET4 condition holds true by the
equality (\ref{eq:AB4}). 
\

(b) $\Rightarrow$ (a) It follows from Proposition \ref{prop:2implica3}.
\end{proof}
We can now give a proof of Theorem A. 
\begin{proof}
[Proof of Theorem A] The equivalence $(b)\Leftrightarrow(c)$ follows
from Proposition \ref{prop:tauiso}. The
implication $(a)\Rightarrow(b)$ can be obtained from Proposition \ref{prop:2implica3}. Finally, in case $\mathcal{C}$ is coproduct-compatible, we get from Theorem \ref{thm:tauvsAB4} that (a) and (b) are equivalent.
\end{proof}

Let us (for the sake of completeness) state the dual statement of Theorem A which is also true.

\begin{thm}\label{thm:mainAop} Let $(\mathcal{C},\mathbb{E},\s)$ be
an extriangulated category with products. Consider the following
statements. 
\begin{enumerate}
\item $\mathcal{C}$ is AET$4^{*}.$ 
\item For any set $I\neq\emptyset,$  the map $\tau:\mathbb{E}(A,\prod_{i\in I}B_i)\rightarrow\prod_{i\in I}\mathbb{E}(A,B_i)$,
$\epsilon\mapsto\left(\pi_{i}^{B}\cdot\epsilon\right)_{i\in I}$,
defines a natural isomorphism. 
\item For any objects $V,D$ in $\mathcal{C}$ there is
a universal $\mathbb{E}$-coextension of $V$ by $D$. That is, for $V,D\in\C,$ there is an 
$\eta\in\mathbb{E}(D,V^{X})$, for some set $X\neq\emptyset$, such
that the map $\text{Hom}_{\mathcal{C}}(V^{X},V)\to\E(D,V),\;g\mapsto g\cdot\eta,$
is surjective. 
\end{enumerate}
Then, the following implications hold true: $(a)\Rightarrow(b)$ and $(b)\Leftrightarrow(c)$.
Moreover, if $\mathcal{C}$ is a product-compatible (see Definition \ref{def:cc}), then the three statements above are
equivalent. 
\end{thm}

\begin{cor}\label{cor:ThmA}
Let $(\mathcal{C},\mathbb{E},\s)$ be a coproduct-compatible (product-compatible) extriangulated category.
If $\mathcal{C}$ has enough $\E$-injectives ($\E$-projectives), then
$\mathcal{C}$ is AET4 (AET4{*}). 
\end{cor}

\begin{proof} Suppose that $\C$ has enough $\E$-injectives. To show that $\mathcal{C}$ is AET4, by Theorem A it is enough to prove that there exists a universal $\mathbb{E}$-extension of $V$ by $D,$ for any $V,D\in\mathcal{C}.$ Indeed, for $D\in\C$ there is a conflation $\rho:\;\suc[D][E][D']$
with $E\in\Inj_\E(D).$ Since $\E(V,E)=0,$ the map $\Hom_{\mathcal{C}}(V,D')\rightarrow\mathbb{E}(V,D),\;h\mapsto\rho\cdot h,$ is epic. By the universal
property of coproducts, there is a morphism $f:V^{(\Hom_{\mathcal{C}}(V,D'))}\rightarrow D'$
such that $f\circ\mu_{h}^{V}=h$ for any $h\in\Hom_{\mathcal{C}}(V,D')$.
Hence, for any $\epsilon\in\mathbb{E}(V,D)$ there is $h\in\Hom_{\mathcal{C}}(V,D')$
such that $\epsilon=\rho\cdot h$ and thus $\epsilon=\rho\cdot\left(f\circ\mu_{h}^{V}\right)=\left(\rho\cdot f\right)\cdot\mu_{h}^{V}$.
Therefore $\rho\cdot f$ is a universal $\mathbb{E}$-extension of
$V$ by $D$. 
\end{proof}

Recall that an exact category $(\mathcal{C},\mathcal{E})$ (in the
sense of Quillen) is $\Ext$-small if the equivalence classes of conflations between $A$ and $B$ form a set for every $A,B\in\mathcal{C}$
(i.e. $\Ext(A,B)$ is a set for all $A,B\in\mathcal{C}$). In this
case, $\mathcal{C}$ has a canonical structure of extriangulated category
with $\mathbb{E}=\Ext.$

\begin{cor}\label{equvExactAET}
For an $\Ext$-small exact category $(\mathcal{C},\mathcal{E}),$ the following statements
are equivalent. 
\begin{enumerate}
\item $\mathcal{C}$ is an AET4 extriangulated category. 
\item $\mathcal{C}$ is an AET3.5 extriangulated category. 
\item $\C$ has coproducts and the functor $\coprod_{i\in I}(-):\Fun(I,\C)\rightarrow\mathcal{C}$
is extriangulated  for any set $I\neq\emptyset.$
\item $\C$ has coproducts and for any set $I\neq\emptyset$ and a family $\{\suc[A_{i}][B_{i}][C_{i}][f_{i}][g_{i}]\}_{i\in I}$ of conflations  in $\C,$
we have that 
\[
\suc[\coprod_{i\in I}A_{i}][\coprod_{i\in I}B_{i}][\coprod_{i\in I}C_{i}][\coprod_{i\in I}f_{i}][\coprod_{i\in I}g_{i}]
\]
 is a conflation in $\C.$
\item $\C$ has coproducts and for any set $I\neq\emptyset,$ the maps $\tau_{C,A}:\Ext(\coprod_{i\in I}C_{i},A)\rightarrow\prod_{i\in I}\Ext(C_{i},A)$,
$\eta\mapsto(\eta\cdot\mu_{i}^{C})_{i\in I},$ define a natural isomorphism. 
\end{enumerate}
Moreover if one of the above conditions holds true, then we have the natural isomorphism 
$\Ext^n\Big(\coprod_{i\in I}D_i,C\Big)\simeq \prod_{i\in I}\Ext^n(D_i,C)$ for any $n\geq 1$ and any set $I\neq\emptyset.$
\end{cor}

\begin{proof} $(a)\Leftrightarrow(b)$ It follows from Theorem \ref{thm:tauvsAB4} since $\C$ is coproduct-compatible. 

$(b)\Leftrightarrow(c)$ It follows from Proposition
\ref{prop:0implica1}. 

$(c)\Leftrightarrow(d)$ It follows from the fact
that, in this context, a functor is extriangulated if, and only if,
it is exact (see \cite[Ex.3.3]{BGLS}).

$(e)\Leftrightarrow (a)$ It follows from Theorem \ref{thm:tauvsAB4} since $\C$ is coproduct-compatible. 

Finally, if one of the above conditions holds true, we get from Theorem \ref{MainTB} the natural isomorphism 
$\Ext^n\Big(\coprod_{i\in I}D_i,C\Big)\simeq \prod_{i\in I}\Ext^n(D_i,C).$
\end{proof}

\section{AET4 for hearts of intervals of $s$-torsion pairs}

For an additive category $\C$ and $\X\subseteq\C,$ we denote by $\Free_\C(\X)$ the class of all the objects $C\in\C$ such that there exists a family $\{X_i\}_{i\in I}$ in $\X$ satisfying that $C=\coprod_{i\in I}X_i$ in $\C.$ Note that $\X\subseteq\Free_\C(\X)$ and in case $\X=\Free_\C(\X)$ it is said that $\X$ is closed under coproducts in $\C.$

Let $\mathcal{D}=(\mathcal{D},\mathbb{E},\s)$ be an extriangulated
category with negative first extension. Following \cite{PS1}, we say that 
$\u=(\mathcal{X},\mathcal{Y})\in\stors(\D)$ is {\bf smashing} if $\mathcal{Y}=\Free_\D(\mathcal{Y}).$ For the $\s$-torsion pair $\u=(\mathcal{X},\mathcal{Y})$ we always have that $\mathcal{X}=\Free_\D(\mathcal{X})$ (see \cite[Prop.3.2]{AET1}).
Let now $[\u,\u']$ be an interval
in $\stors(\mathcal{D})$ with $\u=(\mathcal{X},\mathcal{Y})$ and
$\u'=(\mathcal{X}',\mathcal{Y}')$. Our goal in this section is to find conditions for the heart $\mathcal{H}_{[\u,\u']}$ to
be AET4. A first result follows from the arguments of \cite[Props.3.2,3.3]{PS1}.
We include the proof for completeness.

\begin{lem}\label{lem:lemitaAB} Let $(\mathcal{C},\mathbb{E},\s)$
be an extriangulated category with negative first extension and $\X\subseteq\mathcal{C}$
be closed under extensions and coproducts in $\C.$ If $\mathcal{C}$
is AET4 (resp. AET3, AET3.5), then $\X$ is AET4 (resp. AET3,
AET3.5). 
\end{lem}

\begin{proof} If $\mathcal{C}$ is AET3, it is clear that $\X$ is AET3 since $\X=\Free_\C(\X).$  Let $\mathcal{C}$ be AET3.5 Then, by Proposition \ref{prop:0implica1} we have that $S=\coprod_{i\in I}(-):\Fun(I,\mathcal{C})\rightarrow\mathcal{C}$ is extriangulated. Now, by using that
$\X=\Free_\C(\X),$  we get that $S$
can be restricted to the extriangulated functor $S'=\coprod_{i\in I}(-):\Fun(I,\X)\rightarrow\X$ and thus, by Proposition \ref{prop:0implica1}, it follows that 
$\X$ is AET3.5.

Suppose that $\mathcal{C}$ is AET4. We know that $\Gamma(\eta_{i})_{i\in I}\cdot\mu_{i}^{B}=\mu_{i}^{A}\cdot\eta_{i}$
for all $(\eta_{i})_{i\in I}\in\prod_{i\in I}\mathbb{E}(B_{i},A_{i})$, where
$\mu_{i}^{A}:A_{i}\rightarrow\coprod_{i\in I}A_{i}$
and $\mu_{i}^{B}:B_{i}\rightarrow\coprod_{i\in I}B_{i}$
are the canonical inclusions of coproducts in $\C.$  Since the coproducts in $\X$
are coproducts in $\mathcal{C},$ then $\Gamma$ restricts to $\X$ and thus
 $\X$ is AET4.
\end{proof}

\begin{prop}\label{prop:smash} Let $(\mathcal{D},\mathbb{E},\s)$
be an AET3 extriangulated category with negative first extension and $[\u,\u']$
be an interval in $\stors(\mathcal{D})$, with $\u=(\mathcal{X},\mathcal{Y})$
and $\u'=(\mathcal{X}',\mathcal{Y}')$. Then, for the heart $\mathcal{H}:=\mathcal{H}_{[\u,\u']}=\mathcal{Y}\cap\mathcal{X}'$ the following statements hold true. 
\begin{enumerate}
\item $\mathcal{H}$ and $\mathcal{X}'$ are AET3. 

\item If $\mathcal{X}'$ is AET4 (resp. AET3.5) and $\mathcal{H}=\Free_\D(\mathcal{H}),$
then $\mathcal{H}$ is AET4
(resp. AET3.5).

\item If $\mathcal{D}$ is AET4 (resp. AET3.5) and $\u$ is smashing,
then $\mathcal{H}$ is AET4 (resp. AET3.5). 
\item If $\mathcal{D}$ is a triangulated category and $\u$ is smashing,
then $\mathcal{H}$ is AET4. 
\end{enumerate}
\end{prop}
\begin{proof}
(a) Let $\{H_{i}\}_{i\in I}$ be in $\mathcal{H}.$ 
Since $\mathcal{H}\subseteq\mathcal{X}'=\Free_\D(\mathcal{X}')$ and $\D$ is AET3, the coproduct $\coprod_{i\in I}^{\mathcal{D}}H_{i}$ in $\D$
is a coproduct in $\mathcal{X}'$. Now, by Remark \ref{rem:adjuntos}, there is a left adjoint $L$
of the inclusion $j_{\mathcal{X}'}:\mathcal{H}\rightarrow\mathcal{X}';$ and therefore
$L(\coprod_{i\in I}^{\mathcal{D}}H_{i})$ is the coproduct in $\mathcal{H}.$
\

(b) It follows from Lemma \ref{lem:lemitaAB}.
\

(c) Since $\u$ is smashing, we get that $\mathcal{H}=\Free_\D(\mathcal{H}).$  Then (c)  follows from Lemma \ref{lem:lemitaAB}. 
\

(d) It follows from (c) and Example \ref{exa:ab}(b).
\end{proof}

\begin{lem}\label{lem:Hpreservacoprods} Let $(\mathcal{D},\mathbb{E},\s)$
be an AET3.5 extriangulated category with negative first extension
and $[\u,\u']$ be and interval in $\stors(\mathcal{D})$, with $\u=(\mathcal{X},\mathcal{Y})$
and $\u'=(\mathcal{X}',\mathcal{Y}')$. If $\u$ and $\u'$ are smashing,
then the functor 
$$H_{[\u,\u']}:\mathcal{D}\rightarrow\mathcal{H}_{[\u,\u']},\;D\mapsto \te_{\u'}\circ(1:\te_{\u})(D)$$
preserves coproducts. 
\end{lem}

\begin{proof}
Let $\{D_{i}\}_{i\in I}$ be in $\mathcal{D}.$ For each $i\in I,$ we consider the canonical $\s$-conflation $\suc[\te_{\u}D_{i}][D_{i}][(1:\te_{\u})D_{i}]$.
Since $\mathcal{D}$ is AET3.5, we have by Proposition \ref{prop:0implica1} the $\s$-conflation 
$
\suc[\coprod_{i\in I}\te_{\u}D_{i}][\coprod_{i\in I}D_{i}][\coprod_{i\in I}(1:\te_{\u})D_{i}].
$
Observe that: $\coprod_{i\in I}\te_{\u}D_{i}\in\Free_\D(\mathcal{X})=\mathcal{X}$ and $\coprod_{i\in I}(1:\te_{\u})D_{i}\in\Free_\D(\mathcal{Y})=\mathcal{Y}$
because $\u$ is smashing. Therefore $\te_{\u}\coprod_{i\in I}D_{i}=\coprod_{i\in I}\te_{\u}D_{i}$
and $(1:\te_{\u})\coprod_{i\in I}D_{i}=\coprod_{i\in I}(1:\te_{\u})D_{i}$.
Hence, $\te_{\u}$ and $(1:\te_{\u})$ preserve coproducts. Similarly,
one can show that $\te_{\u'}$ and $(1:\te_{\u'})$ preserve coproducts.
Therefore $H_{[\u,\u']}=\te_{\u'}\circ(1:\te_{\u})$ preserve coproducts. 
\end{proof}

\subsection{AET4 for hearts and extended hearts of $t$-structures}

For an additive category $\C$ and $\X\subseteq\C,$ we denote by $\coFree_\C(\X)$ the class of all the objects $C\in\C$ such that there exists a family $\{X_i\}_{i\in I}$ in $\X$ satisfying that $C=\prod_{i\in I}X_i$ in $\C.$ Note that $\X\subseteq\coFree_\C(\X)$ and in case $\X=\coFree_\C(\X)$ it is said that $\X$ is closed under products in $\C.$

Let $\mathbf{x}=(\mathcal{U},\mathcal{W})$ be a $t$-structure in
a triangulated category $(\mathcal{D},\Sigma,\triangle)$. We will
use the following $s$-torsion pairs in $\mathcal{D}$ throughout
this section: $\u_{1}:=(\Sigma\mathcal{U},\mathcal{W})$, $\u_{2}:=(\mathcal{U},\Sigma^{-1}\mathcal{W})$
and $\Sigma^{-n+1}\u_{2}:=(\Sigma^{-n+1}\mathcal{U},\Sigma^{-n}\mathcal{W})$
for all $n\geq1$. Recall that $\mathcal{H}:=\mathcal{H}_{[\u_{1},\u_{2}]}=\mathcal{W}\cap\mathcal{U}$
is the heart of $\mathbf{x}$ and 
\[
\mathcal{C}_{n}:=\mathcal{H}_{[\u_{1},\Sigma^{-n+1}\u_{2}]}=\mathcal{W}\cap\Sigma^{-n+1}\mathcal{U}=\mathcal{H}\star\Sigma^{-1}\mathcal{H}\star\cdots\star\Sigma^{-n+1}\mathcal{H}
\]
 is the {\bf extended heart} of length $n$ of $\mathbf{x}$ (see \cite[Cor.3.4]{AMP}). Note that $\C_n$ is AET4 if $\D$ has coproducts and $\mathbf{x}$ is smashing (see Proposition \ref{prop:smash} (d)).

An important fact to note is that, for any $1<m<n$, $(\mathcal{H},\Sigma^{-1}\mathcal{C}_{n-1})$
and $(\mathcal{C}_{m},\Sigma^{-m}\mathcal{C}_{n-m})$ are $s$-torsion
pairs in $\mathcal{C}_{n}$ (see \cite[Lem.3.2]{AMP}). In particular, $\mathcal{H}$
and $\mathcal{C}_{m}$ are closed under cones, extensions, direct
summands and coproducts in $\mathcal{C}_{n}$ (see \cite[Prop.2.9]{AMP}
and \cite[Prop.3.2]{AET1}). Similarly, $\Sigma^{-n+1}\mathcal{H}$
and $\Sigma^{-m}\mathcal{C}_{n-m}$ are closed under co-cones, extensions,
direct summands and products in $\mathcal{C}_{n}$.

\begin{lem}\label{lem:lemita2} Let $\mathbf{x}=(\mathcal{U},\mathcal{W})$ be
a $t$-structure in a triangulated category $(\mathcal{D},\Sigma,\triangle)$,
$\mathcal{H}$ be the heart of $\mathbf{x}$ and $\mathcal{C}_{n}:=\mathcal{W}\cap\Sigma^{-n+1}\mathcal{U}$
for some $n\geq2$. If $\mathcal{C}_{n}$ is AET3.5 and $\{H_{\lambda}\}_{\lambda\in\Lambda}$
is a family in $\mathcal{H}$, then 
\[
\coprod_{\lambda\in\Lambda}^{\mathcal{C}_{n}}\left(\Sigma^{-n+1}H_{\lambda}\right)=\Sigma^{-n+1}\left(\coprod_{\lambda\in\Lambda}^{\mathcal{H}}H_{\lambda}\right)=\Sigma^{-n+1}\left(\coprod_{\lambda\in\Lambda}^{\mathcal{C}_{n}}H_{\lambda}\right).
\]
\end{lem}

\begin{proof}
Note that $\coprod_{\lambda\in\Lambda}^{\mathcal{H}}H_{\lambda}=\coprod_{\lambda\in\Lambda}^{\mathcal{C}_{n}}H_{\lambda}$ since $\mathcal{H}=\Free_{\C_n}(\mathcal{H}).$
Hence, it is enough to prove the first equality. For this, we proceed
by induction on $n\geq2$.

Let $n=2.$ Consider the family of split $\s$-conflations $\{\suc[H_{\lambda}][H_{\lambda}^{2}][H_{\lambda}][f_{\lambda}][g_{\lambda}]\}_{\lambda\in\Lambda}$.
By rotation of triangles in $\D,$ we get that $\{\suc[\Sigma^{-1}H_{\lambda}][H_{\lambda}][H_{\lambda}^{2}][0][f_{\lambda}]\}_{\lambda\in\Lambda}$
is a family of $\s$-conflations in $\mathcal{C}_{n}.$ Hence, by Proposition \ref{prop:0implica1}, we have an $\s$-conflation
\[
\suc[\coprod_{\lambda\in\Lambda}^{\mathcal{C}_{n}}\left(\Sigma^{-1}H_{\lambda}\right)][\coprod_{\lambda\in\Lambda}^{\mathcal{C}_{n}}H_{\lambda}][\coprod_{\lambda\in\Lambda}^{\mathcal{C}_{n}}H_{\lambda}^{2}][0][\coprod_{\lambda\in\Lambda}^{\mathcal{C}_{n}}f_{\lambda}].
\]
On the other hand, by Lemma \ref{lem:lemitaAB} we know that $\mathcal{H}$ is AET3.5. Thus, by Proposition \ref{prop:0implica1}, we have
$\suc[\coprod_{\lambda\in\Lambda}^{\mathcal{H}}H_{\lambda}][\coprod_{\lambda\in\Lambda}^{\mathcal{H}}H_{\lambda}^{2}][\coprod_{\lambda\in\Lambda}^{\mathcal{H}}H_{\lambda}][\coprod_{\lambda\in\Lambda}^{\mathcal{H}}f_{\lambda}][\coprod_{\lambda\in\Lambda}^{\mathcal{H}}g_{\lambda}]$
is a split $\s$-conflation in $\mathcal{H}$. Then, by rotating this
triangle in $\D,$ we get the $\s$-conflation 
\[
\suc[\Sigma^{-1}\left(\coprod_{\lambda\in\Lambda}^{\mathcal{\mathcal{H}}}H_{\lambda}\right)][\coprod_{\lambda\in\Lambda}^{\mathcal{\mathcal{H}}}H_{\lambda}][\coprod_{\lambda\in\Lambda}^{\mathcal{\mathcal{H}}}H_{\lambda}^{2}][0][\coprod_{\lambda\in\Lambda}^{\mathcal{H}}f_{\lambda}].
\]
Note that $\coprod_{\lambda\in\Lambda}^{\mathcal{C}_{n}}f_{\lambda}=\coprod_{\lambda\in\Lambda}^{\mathcal{H}}f_{\lambda}$ and thus 
 $\Sigma^{-1}\left(\coprod_{\lambda\in\Lambda}^{\mathcal{\mathcal{H}}}H_{\lambda}\right)=\coprod_{\lambda\in\Lambda}^{\mathcal{C}_{n}}\left(\Sigma^{-1}H_{\lambda}\right)$. 

Let $n>2$ and assume that $\coprod_{\lambda\in\Lambda}^{\mathcal{C}_{n-1}}\left(\Sigma^{-n+2}H_{\lambda}\right)=\Sigma^{-n+2}\left(\coprod_{\lambda\in\Lambda}^{\mathcal{H}}H_{\lambda}\right)$.
Consider the family of split $\s$-conflations 
\[
\{\Sigma^{-n+2}\suc[H_{\lambda}][\Sigma^{-n+2}H_{\lambda}^{2}][\Sigma^{-n+2}H_{\lambda}][f_{\lambda}][g_{\lambda}]\}_{\lambda\in\Lambda}.
\]
 By similar arguments as before, we can prove that 
\[
\Sigma^{-1}\left(\coprod_{\lambda\in\Lambda}^{\mathcal{\mathcal{C}}_{n-1}}\Sigma^{-n+2}H_{\lambda}\right)=\coprod_{\lambda\in\Lambda}^{\mathcal{C}_{n}}\left(\Sigma^{-n+1}H_{\lambda}\right).
\]
And thus, by the induction hypothesis, we get that 
\[
\coprod_{\lambda\in\Lambda}^{\mathcal{C}_{n}}\left(\Sigma^{-n+1}H_{\lambda}\right)=\Sigma^{-1}\left(\coprod_{\lambda\in\Lambda}^{\mathcal{\mathcal{C}}_{n-1}}\Sigma^{-n+2}H_{\lambda}\right)=\Sigma^{-n+1}\left(\coprod_{\lambda\in\Lambda}^{\mathcal{H}}H_{\lambda}\right).
\]
 
\end{proof}
It will be useful to introduce the following definition from \cite[Def.5.4]{V}.

\begin{defn} Let $\mathbf{x}=(\mathcal{U},\mathcal{W})$ be a $t$-structure in
a triangulated category $(\mathcal{D},\Sigma,\triangle).$
For $n\geq0$, we say that $\mathbf{x}$ is \textbf{$n$-smashing}
if $\Free_\D(\mathcal{W})\subseteq\Sigma^{n}\mathcal{W}.$ Dually, we say
that $\x$ is \textbf{$n$-co-smashing} if $\coFree_\D(\mathcal{U})\subseteq\Sigma^{-n}\mathcal{U}$.
\end{defn}

\begin{rem}\label{rem:smash} Let $\mathbf{x}=(\mathcal{U},\mathcal{W})$ be a
$t$-structure in a triangulated category $(\mathcal{D},\Sigma,\triangle)$
and $n,m\geq0$.
\begin{enumerate}
\item The notion of smashing (resp. co-smashing) coincides with $0$-smashing
(resp. 0-co-smashing). 
\item If $\mathbf{x}$ is $n$-smashing (resp. $n$-co-smashing), then it
is also $(n+1)$-smashing (resp. $(n+1)$-co-smashing). 
\item If $\mathbf{x}$ is $n$-smashing and $\mathbf{y}=(\mathcal{U}',\mathcal{W}')$
is a $t$-structure such that $\mathcal{U}\subseteq\mathcal{U}'\subseteq\Sigma^{-m}\mathcal{U}$,
then $\mathbf{y}$ is $(n+m)$-smashing. Indeed, we have that $\mathcal{W}'\subseteq\mathcal{W}\subseteq\Sigma^{m}\mathcal{W}'$;
and thus, $\Free_\D(\mathcal{W}')\subseteq\Free_\D(\mathcal{W})\subseteq\Sigma^{n}\mathcal{W}\subseteq\Sigma^{n+m}\mathcal{W}'$. 
\item Similarly, if $\mathbf{x}$ is $n$-co-smashing and $\mathbf{y}=(\mathcal{U}',\mathcal{W}')$
is a $t$-structure such that $\mathcal{U}\subseteq\mathcal{U}'\subseteq\Sigma^{-m}\mathcal{U}$,
then $\mathbf{y}$ is $(n+m)$-co-smashing. 
\end{enumerate}
\end{rem}

In the following lemma we show that the $t$-structures parametrized
by the Happel-Reiten-Smal{\o} tilting process are $1$-smashing.

\begin{lem}\label{lem:HRSsmash} For an AB3 abelian category $\mathcal{A}$ and the standard $t$-structure $\mathbf{s}=(\mathcal{D}^{\leq0},\mathcal{D}^{\geq0})$ in the derived category $\D(\mathcal{A}),$ the following statements hold true. 
\begin{enumerate}
\item $\D(\mathcal{A})$ is AET4.
\item $\mathbf{s}=(\mathcal{D}^{\leq0},\mathcal{D}^{\geq0})$
is smashing if, and only if, $\mathcal{A}$ is AB4. 
\item If $\mathcal{A}$ is AB4 and $(\mathcal{T},\mathcal{F})\in\stors(\mathcal{A})$,
then $(\Sigma\mathcal{D}^{\leq0}\star\mathcal{T},\Sigma\mathcal{F}\star\mathcal{D}^{\geq0})$
is a $1$-smashing $t$-structure in $\D(\mathcal{A}).$
\end{enumerate}
\end{lem}

\begin{proof} (a) It is well known that $\D(\mathcal{A})$ has coproducts since $\mathcal{A}$ is AB3. Then, by Example \ref{exa:ab} (b), we get that $\D(\mathcal{A})$ is AET4.

(b) If $\mathbf{s}$ is smashing, then $\mathcal{A}\cong\mathcal{H}_{\mathbf{s}}$
is AET4 by Proposition \ref{prop:smash} (d); and thus by Corollary  \ref{equvExactAET} $\mathcal{A}$ is AB4. If $\mathcal{A}$ is AB4,
then the coproduct commutes with the co-homology functors. Therefore
$\Free_{\D(\mathcal{A})}(\mathcal{D}^{\leq0})=\mathcal{D}^{\leq0}$ (see also \cite[Prop.3.3]{PS1}). 

(c) Let $\mathcal{A}$ be AB4. Then, by Corollary \ref{equvExactAET} we know that $\mathcal{A}$ is AET4. Hence (c) follows from (b) and Remark \ref{rem:smash} (c).
\end{proof}
\begin{lem}\label{lem:lemita1} For $n\geq2$, 
 an $(n-1)$-smashing $t$-structure $\mathbf{x}=(\mathcal{U},\mathcal{W})$ in a triangulated category
$(\mathcal{D},\Sigma,\triangle)$ with coproducts, the heart $\mathcal{H}$  of $\mathbf{x}$, and $\mathcal{C}_{n}:=\mathcal{W}\cap\Sigma^{-n+1}\mathcal{U},$
the following statements hold true.
\begin{itemize}
\item[(a)] $\coprod_{\lambda\in\Lambda}^{\mathcal{C}_{n}}\Sigma^{-n+1}H_{\lambda}=\coprod_{\lambda\in\Lambda}^{\mathcal{D}}\Sigma^{-n+1}H_{\lambda}$ for $\{H_{\lambda}\}_{\lambda\in\Lambda}$ in $\mathcal{H}.$
\item[(b)] $\Free_\D(\mathcal{H})\subseteq \C_n$ if $\C_n$ is AET3.5.
\end{itemize}
\end{lem}

\begin{proof} (a)
Firstly, note that $\coprod_{\lambda\in\Lambda}^{\mathcal{D}}H_{\lambda}\in\Sigma^{n-1}\mathcal{W}$
since $\mathbf{x}$ is $(n-1)$-smashing. Thus 
\[
\coprod_{\lambda\in\Lambda}^{\mathcal{D}}\left(\Sigma^{-n+1}H_{\lambda}\right)\cong\Sigma^{-n+1}\left(\coprod_{\lambda\in\Lambda}^{\mathcal{D}}H_{\lambda}\right)\in\mathcal{W}.
\]
On the other hand, by Remark \ref{rem:adjuntos} (b), we know that there is a left
adjoint $L$ of the inclusion $j:\mathcal{C}_{n}\rightarrow\Sigma^{-n+1}\mathcal{U}$.
Hence $\coprod_{\lambda\in\Lambda}^{\mathcal{C}_{n}}\Sigma^{-n+1}H_{\lambda}=L(\coprod_{\lambda\in\Lambda}^{\mathcal{D}}\Sigma^{-n+1}H_{\lambda})$
(see the proof of \cite[Prop.3.2]{PS1}). Moreover, by Remark \ref{rem:L} (b) we have that 
\[
L(\coprod_{\lambda\in\Lambda}^{\mathcal{D}}\Sigma^{-n+1}H_{\lambda})=\Sigma(1:\te_{\u_{2}})(\Sigma^{-1}\coprod_{\lambda\in\Lambda}^{\mathcal{D}}\left(\Sigma^{-n+1}H_{\lambda}\right)).
\]
Therefore, since $\Sigma^{-1}\coprod_{\lambda\in\Lambda}^{\mathcal{D}}\left(\Sigma^{-n+1}H_{\lambda}\right)\in\Sigma^{-1}\mathcal{W}$,
we conclude that 
\[
\coprod_{\lambda\in\Lambda}^{\mathcal{C}_{n}}\Sigma^{-n+1}H_{\lambda}=L(\coprod_{\lambda\in\Lambda}^{\mathcal{D}}\Sigma^{-n+1}H_{\lambda})=\coprod_{\lambda\in\Lambda}^{\mathcal{D}}\left(\Sigma^{-n+1}H_{\lambda}\right).
\]
(b) Let $\{H_{\lambda}\}_{\lambda\in\Lambda}$ be in $\mathcal{H}.$ Then, by (a) and Lemma \ref{lem:lemita2}, we have that
\[
\Sigma^{-n+1}\coprod_{\lambda\in\Lambda}^{\mathcal{D}}H_{\lambda}=\coprod_{\lambda\in\Lambda}^{\mathcal{C}_{n}}\left(\Sigma^{-n+1}H_{\lambda}\right)=\Sigma^{-n+1}\left(\coprod_{\lambda\in\Lambda}^{\mathcal{H}}H_{\lambda}\right)=\Sigma^{-n+1}\left(\coprod_{\lambda\in\Lambda}^{\mathcal{C}_{n}}H_{\lambda}\right).
\]
Therefore $\coprod_{\lambda\in\Lambda}^{\mathcal{C}_{n}}H_{\lambda}=\coprod_{\lambda\in\Lambda}^{\mathcal{D}}H_{\lambda}$ proving (b).
\end{proof}

\begin{thm}\label{thm:smashing} Let $n\geq2$, $\mathbf{x}=(\mathcal{U},\mathcal{W})$
be an $(n-1)$-smashing $t$-structure in a triangulated category
$(\mathcal{D},\Sigma,\triangle)$ with coproducts, $\mathcal{H}$
be the heart of $\mathbf{x}$ and $\mathcal{C}_{n}:=\mathcal{W}\cap\Sigma^{-n+1}\mathcal{U}$.
Then, the following statements are equivalent.
\begin{itemize}
\item[(a)] $\mathbf{x}$ is $0$-smashing. 
\item[(b)] $\mathcal{C}_{n}$ is AET4.
\item[(c)] $\mathcal{C}_{n}$ is AET3.5
\end{itemize}
\end{thm}

\begin{proof} Note that the implication (a) $\Rightarrow$ (b) follows from Proposition \ref{prop:smash} (d); and the implication (b) $\Rightarrow$ (c) is trivial.

Assume now that $\mathcal{C}_{n}$ is AET3.5. An important fact to note firstly is that, for any $1<m<n$, we have that $(\mathcal{C}_{m},\Sigma^{-m}\mathcal{C}_{n-m})$ is an $s$-torsion pair in $\mathcal{C}_{n}$ (see \cite[Lem.3.2]{AMP}). Then, by Lemma \ref{lem:lemitaAB}, we conclude that $\mathcal{C}_{m}$ is also AET3.5 for any $1<m<n.$

Now, let $\{V_{\lambda}\}_{\lambda\in\Lambda}$ be in $\mathcal{W}.$ Since $(\U,\Sigma^{-1}\W)$ is an $\s$-torsion pair, we get the canonical $\s$-conflation 
$
\suc[U_{\lambda}][V_{\lambda}][\Sigma^{-1}W_{\lambda}]
$
 with $U_{\lambda}\in\mathcal{U}$ and $W_{\lambda}\in\mathcal{W}$
for every $\lambda\in\Lambda$. One can check that $U_{\lambda}\in\mathcal{H}$
(see \cite{BBD} or \cite[Prop.3.1]{PS1}). Since $\D$ is AET4 (see Example \ref{exa:ab} (b)), by Proposition \ref{prop:0implica1} and the above $\s$-conflations we get  the $\s$-conflation 
\[
\suc[\coprod_{\lambda\in\Lambda}^{\mathcal{D}}U_{\lambda}][\coprod_{\lambda\in\Lambda}^{\mathcal{D}}V_{\lambda}][\coprod_{\lambda\in\Lambda}^{\mathcal{D}}\Sigma^{-1}W_{\lambda}]
\]
Note that $\Free_\D(\Sigma^{-1}\W)\subseteq \Sigma^{n-2}\W$ since 
$
\Sigma\coprod_{i\in I}^{\mathcal{D}}\Sigma^{-1}W_i=\coprod_{i\in I}^{\mathcal{D}}W_i\in\Sigma^{n-1}\W.
$
On the other hand, by Lemma \ref{lem:lemita1}(b) we have that $\coprod_{\lambda\in\Lambda}^{\mathcal{D}}U_{\lambda}\in\mathcal{C}_{n}\subseteq\mathcal{W}.$ Therefore, 
 $\coprod_{\lambda\in\Lambda}^{\mathcal{D}}V_{\lambda}\in\mathcal{W}\star\Sigma^{n-2}\mathcal{W}\subseteq\Sigma^{n-2}\mathcal{W}$ and thus
 $\mathbf{x}$ is $(n-2)$-smashing. Finally, by using that $\mathcal{C}_m$ is AET3.5 for $1<m<n,$ by recursion it follows that $\mathbf{x}$
is $0$-smashing.
\end{proof}

In what follows, we state and prove the dual version of Theorem \ref{thm:smashing}. We do that for the sake of completeness and also since we use it in the example given in section \ref{ExampleNAET4}.

\begin{thm}\label{thm:cosmashing} Let $n\geq2$, $\mathbf{x}=(\mathcal{U},\mathcal{W})$
be an $(n-1)$-co-smashing $t$-structure in a triangulated category
$(\mathcal{D},\Sigma,\triangle)$ with products, $\mathcal{H}$ be
the heart of $\mathbf{x}$ and $\mathcal{C}_{n}:=\mathcal{W}\cap\Sigma^{-n+1}\mathcal{U}.$ Then, the following statements are equivalent.
\begin{itemize}
\item[(a)] $\mathbf{x}$ is $0$-co-smashing. 
\item[(b)] $\mathcal{C}_{n}$ is AET4{*}.
\item[(c)] $\mathcal{C}_{n}$ is AET3.5{*}.
\end{itemize}
\end{thm}

\begin{proof}
This follows by dualizing Theorem \ref{thm:smashing}. For this, recall
that $(\mathcal{D}^{op},T,\triangle^{op})$ is a triangulated category
with $T:=(\Sigma^{-1})^{op}$ and $\triangle^{op}$ consisting of
the sequences of the form $\suc[X^{op}][Y^{op}][Z^{op}][f^{op}][g^{op}]$
for any $\suc[Z][Y][X][g][f]$ in $\triangle$. One can check that
$\x^{op}:=(\mathcal{W}^{op},\mathcal{U}^{op})$ is a $(n-1)$-co-smashing
$t$-structure with heart $\mathcal{H}_{\x^{op}}=\mathcal{H}^{op}$,
and extended heart of length $n$ equal to 
\[
\mathcal{C}_{\x^{op}}=\mathcal{H}^{op}\star T^{-1}\mathcal{H}^{op}\star\cdots\star T^{-n+1}\mathcal{H}^{op}=\mathcal{H}^{op}\star\cdots\star(\Sigma^{n-1}\mathcal{H})^{op}=(\Sigma^{n-1}\mathcal{C}_{n})^{op}.
\]
Now, by Theorem \ref{thm:smashing}, we have that: $\mathcal{C}_{\x^{op}}$
is AET4 (AET3.5) $\Leftrightarrow$ $\x^{op}$ is smashing. Hence, we have that
$(\mathcal{C}_{\x^{op}})^{op}$ is AET4{*} (AET3.5{*}) $\Leftrightarrow$ $\x$ is 0-co-smashing.
Therefore, since $(\mathcal{C}_{\x^{op}})^{op}=\Sigma^{n-1}\mathcal{C}_{n}$
and $\Sigma^{-n+1}:\Sigma^{n-1}\mathcal{C}_{n}\rightarrow\mathcal{C}_{n}$
is an isomorphism of categories, we get that $\mathcal{C}_{n}$ is
AET4{*} (AET3.5{*}) $\Leftrightarrow$ $\mathbf{x}$ is $0$-co-smashing.
\end{proof}

\begin{cor}\label{corTsmsingAB}
Let $\mathcal{A}$ be an AB4 (AB4{*}) abelian category, $\u=(\mathcal{T},\mathcal{F})\in\stors(\mathcal{A})$,
$\mathcal{H}$ be the heart of the $t$-structure $\x=(\Sigma\mathcal{D}^{\leq0}\star\mathcal{T},\Sigma\mathcal{F}\star\mathcal{D}^{\geq0})$ in the derived category 
$\D(\mathcal{A})$
and the extended heart $\mathcal{C}=\mathcal{H}\star\Sigma^{-1}\mathcal{H}$. Then, the
following statements are equivalent. 
\begin{itemize}
\item[(a)] $\x$ is smashing (co-smashing) in $\D(\mathcal{A}).$
\item[(b)] $\u$ is smashing (co-smashing) in $\mathcal{A}.$
\item[(c)] $\mathcal{C}$ is AET4 (AET4{*}). 
\item[(d)] $\mathcal{C}$ is AET3.5 (AET3.5{*}). 
\end{itemize}
\end{cor}

\begin{proof} By Lemma \ref{lem:HRSsmash} we know that $\x$ is 1-smashing in $\D(\mathcal{A}).$ Thus by Theorem \ref{thm:smashing} we get that (a), (c) and (d) are equivalent.

$(a)\Leftrightarrow(b)$ Since $\mathcal{A}$ is AB4, we have that
the co-homology functors preserve coproducts. Thus, $\mathcal{D}^{\geq0}$
is closed under coproducts in $\mathcal{D}$. Now, since $\Sigma\mathcal{F}\star\mathcal{D}^{\geq0}=\{X\in\Sigma\mathcal{D}^{\geq0}\,|\:H^{-1}(X)\in\mathcal{F}\}$,
it follows that $\x$ is smashing if and only if $\mathcal{F}$ is
closed under coproducts in $\mathcal{A}.$
\end{proof}

For a $t$-structure $\te = (\U,\W)$ in a triangulated category $\mathcal{D}$, the \textbf{global dimension} of $\te$ is defined as $\gldim(\te)=\min\{k\in \mathbb{N}\,|\:\Hom _{\mathcal{C}} (\W, \Sigma ^{k+1} \U )=0\}$ (see \cite[Def.3.3]{CLZ}). In particular, if ${\mathcal{D}}$ is the derived category of an abelian category $\mathcal{A}$ with enough projectives or injectives, then the global dimension of the standard $t$-structure $\te_0$ coincides with the global dimension of $\mathcal{A}$. That is, $\gldim(\te_0)$ is equal to the smallest non-negative integer $d$ such that 
$\Ext_{\mathcal{A}}^{i}(-,-)=0\quad\forall i>d$ (see \cite[Prop.3.6]{CLZ}).

\begin{lem}\label{lem:gldim} Let $\mathcal{A}$ be an AB3 abelian category
and  $\te  :=(\mathcal{D}^{\leq0},\mathcal{D}^{\geq0})$  be the standard $t$-structure in $\mathcal{D}(\mathcal{A})$.
 If $\gldim(\te)\leq n$, then $\te$ is $n$-smashing.
\end{lem}

\begin{proof}
Let $\{W_{i}\}_{i\in I}$ be a family of objects in $\mathcal{D}^{\geq0}$. Since $\gldim(\te)\leq n$, we have that $\Hom_{\mathcal{D}(\mathcal{A})}(\mathcal{D}^{\geq0},\Sigma^{n+1}\mathcal{D}^{\leq0})=0.$ Now, by using \emph{strong truncations}, we can find a distinguished
triangle $\suc[\Sigma^{n}W][\coprod_{i\in I}W_{i}][\Sigma^{n+1}U][a][b]$
with $U\in\mathcal{D}^{\leq0}$ and $W\in\mathcal{D}^{\geq0}$. It
follows that $b=0$ since $b\circ\mu_{i}^{W}\in\Hom_{\mathcal{D}(\mathcal{A})}(\mathcal{D}^{\geq0},\Sigma^{n+1}\mathcal{D}^{\leq0})=0$
for all $i\in I$. This implies that $a$ is an split-epimorphism and thus 
$\coprod_{i\in I}W_{i}\in\Sigma^{n}\mathcal{D}^{\geq0}$. 
\end{proof}

\begin{cor}\label{cor:gldim}
Let $n\geq 1$, $\mathcal{A}$ be an AB3 abelian category,   $\te  :=(\mathcal{D}^{\leq0},\mathcal{D}^{\geq0})$  be the standard $t$-structure in $\mathcal{D}(\mathcal{A})$ and $\mathcal{C}_{n+1}:=\mathcal{D}^{\geq0}\cap\Sigma^{-n}\mathcal{D}^{\leq0}.$ If $\gldim(\te)\leq n$, then
$(\mathcal{D}^{\leq0},\mathcal{D}^{\geq0})$ is smashing $\quad\Leftrightarrow\quad$ $\mathcal{C}_{n+1}$ is AET4 (AET3.5).
\end{cor}
\begin{proof}
It follows from Lemma \ref{lem:gldim} that $(\mathcal{D}^{\leq0},\mathcal{D}^{\geq0})$
is $n$-smashing. Then, the result follows from Theorem \ref{thm:smashing}.
\end{proof}

\begin{cor}\label{corAB4DA}
For an AB3 abelian category $\mathcal{A},$ the following statements are equivalent.
\begin{enumerate}
\item $\mathcal{A}$ is AB4. 
\item The standard $t$-structure $(\mathcal{D}^{\leq0},\mathcal{D}^{\geq0})$
in $\mathcal{D}(\mathcal{A})$ is smashing. 
\item $\coprod_{i\in I}^{\mathcal{A}}A_{i}=\coprod_{i\in I}^{\mathcal{D}(\mathcal{A})}A_{i}$ for any set $I\neq\emptyset$ and any family $\{A_{i}\}_{i\in I}$ in $\mathcal{A}.$
\item The natural transformation $\dot{\tau}:\Ext_{\mathcal{A}}^{1}(\coprod_{i\in I}A_{i},B)\rightarrow\prod_{i\in I}\Ext_{\mathcal{A}}^{1}(A_{i},B)$,
$\eta\mapsto(\eta\cdot\mu_{i}^{A})_{i\in I}$, is an isomorphism for any set $I\neq\emptyset.$
\end{enumerate}
\end{cor}
\begin{proof} By Lemma \ref{lem:HRSsmash} (b)  we have that (a) and (b) are equivalent. Furthermore, since $\Ext_{\mathcal{A}}^{1}(-,?)$
is isomorphic to $\Hom_{\D(\mathcal{A})}(-,\Sigma?)$, it can be seen that (c) implies (d). On the other hand, using that $\mathcal{D}^{\leq0}\cap\mathcal{D}^{\geq0}=\mathcal{A}$ and $\Free_{\D(\mathcal{A})}(\mathcal{D}^{\leq0})=\mathcal{D}^{\leq0},$ we get that (b) implies (c). Finally, by Corollary \ref{equvExactAET} we conclude that (a) and (d) are equivalent. Therefore, all of the above
statements are equivalent. 
\end{proof}

\subsection{An example of a non-AET4{*} extended heart}\label{ExampleNAET4}

In this section we will show an example of a non-AET4{*} extended heart.
Specifically, we will consider a $t$-structure associated to a torsion
pair by the Happel-Reiten-Smal{\o} tilting process in a category
of modules, and prove that its heart $\mathcal{H}$ is an abelian AET4{*} category
 such that the extended heart $\mathcal{C}=\mathcal{H}\star\Sigma^{-1}\mathcal{H}$
is not AET4{*}. Here, AET4{*} is the dual notion of AET4. 

Let us begin with the example. Consider the product ring $R:=\prod_{i\in\mathbb{N}}\mathbb{Z}_{2}$ and the ideal $R_{j}:=\left\{ \left(x_{i}\right)_{i\in\mathbb{N}}\in R\,|\:x_{i}=0\;\forall i\neq j\right\}$ for each $j\in \mathbb{N}.$ 
 Note that $R_{i}\in\Proj(R)$
for all $i\in\mathbb{N}$ and thus $I:=\oplus_{i\in\mathbb{N}}R_i\in\Proj(R)$. Moreover, since $I$ is an idempotent ideal of $R,$ the classes 
\begin{alignat*}{1}
\mathcal{C}_{I} & :=\Gen(I)=\{M\in\Mod(R)\,|\:I\cdot M=M\}\text{ and }\\
\mathcal{T}_{I} & :=\Gen(R/I)=\{M\in\Mod(R)\,|\:I\cdot M=0\}
\end{alignat*}
form a torsion pair $(\mathcal{C}_{I},\mathcal{T}_{I})$  (see \cite[Cor.2.2]{jans1965some}). 

\begin{rem}\label{rem:her} We have that $\mathcal{C}_{I}$ is closed under subobjects.
Indeed, consider the set $\{e_{i}\}_{i\in\mathbb{N}}$
of canonical idempotents in $R$. That is, for any $j\in\mathbb{N}$,
$e_{j}$ is the object $(x_{i})_{i\in\mathbb{N}}\in R$ such that
$x_{j}=1$ and $x_{i}=0$ for all $i\neq j$. Observe that 
$M\in\Mod(R)$ satisfies that $M\in\mathcal{C}_{I}$ if, and only if,
$M=\bigoplus_{i\in\mathbb{N}}e_{i}M$. Hence, for $N\leq M$ and $M\in\mathcal{C}_{I}$,
one can show that $N=\bigoplus_{i\in\mathbb{N}}e_{i}N$ and thus
$\mathcal{C}_{I}$ is closed under subobjects.
\end{rem}

Throughout this section, we fix the following notation: $\mathcal{D}=\D(R)$  is the derived category
 of the module category $\Mod(R),$ $(\mathcal{D}^{\leq0},\mathcal{D}^{\geq0})$
is the canonical $t$-structure in $\mathcal{D},$ and $\mathbf{x}=(\mathcal{U},\mathcal{W})$
is the $t$-structure associated to $(\mathcal{C}_{I},\mathcal{T}_{I})$
by the Happel-Reiten-Smal{\o} tilting process. That is
\begin{alignat*}{1}
\mathcal{U} & =\Sigma\mathcal{D}^{\leq0}\star\mathcal{C}_{I}=\{X\in\mathcal{D}^{\leq0}\,|\:H^{0}(X)\in\mathcal{C}_{I}\}\text{ and }\\
\mathcal{W} & =\Sigma\mathcal{T}_{I}\star\mathcal{D}^{\geq0}=\{X\in\Sigma\mathcal{D}^{\geq0}\,|\:H^{-1}(X)\in\mathcal{T}_{I}\}.
\end{alignat*}
Recall that a $t$-structure $\mathbf{v}=(\mathcal{X},\mathcal{Y})$
is \textbf{co-smashing} if $\coFree_\D(\mathcal{X})=\mathcal{X}.$ By the
dual of Proposition \ref{prop:smash} (d), we know that the heart $\mathcal{H}_{\mathbf{v}}$ is AET4{*} if $\mathbf{v}$ is co-smashing. In the following, we will
show that the converse is false. 

\begin{prop}\label{prop:HRS1} $\mathcal{H}_{\mathbf{x}}$ is AET4{*} but $\mathbf{x}$
is not co-smashing.
\end{prop}

\begin{proof}
Let us prove that $\mathbf{x}$ is not co-smashing. For this, observe
that $\mathcal{D}^{\leq0}$ is closed under products and that the
co-homology functor $H^{0}$ preserves products since $\Mod(R)$ is
AB4{*}. Therefore, it is enough to show that $\mathcal{C}_{I}$ is
not closed under products. For this, note that $R_{i}\in\mathcal{C}_{I}$
for all $i\in\mathbb{N}$ (see Remark \ref{rem:her}), but $R\cong\prod_{i\in\mathbb{N}}R_{i}\notin\mathcal{C}_{I}$. 

To prove that $\mathcal{H}_{\mathbf{x}}$ is AET4{*}, it is enough
to show that $\mathcal{H}_{\mathbf{x}}$ has enough projectives (see
Corollary \ref{cor:ThmA}). For this, we claim that $P:=I\amalg\Sigma R/I$
is a projective generator of $\mathcal{H}_{\mathbf{x}}$. Indeed,
for $T\in\mathcal{T}_{I}$, we have that:
\begin{alignat*}{1}
\Ext_{\mathcal{H}_{\mathbf{x}}}^{1}(\Sigma R/I,\Sigma T) & \cong\Hom_{\mathcal{D}}(\Sigma R/I,\Sigma^{2}T)\cong\Ext_{R}^{1}(R/I,T)\cong\Ext_{R/I}^{1}(R/I,T)=0\\
\Ext_{\mathcal{H}_{\mathbf{x}}}^{1}(I,\Sigma T) & \cong\Hom_{\mathcal{D}}(I,\Sigma^{2}T)\cong\Ext_{R}^{2}(I,T)=0
\end{alignat*}
since $\mathcal{T}_{I}=\Mod(R/I)$, $R/I\in\Proj(R/I)$ and $I\in\Proj(R)$
(here, the first isomorphisms are from \cite[Rem.3.1.17]{BBD}); and,
for $C\in\mathcal{C}_{I}$, we have that 
\begin{alignat*}{1}
\Ext_{\mathcal{H}_{\mathbf{x}}}^{1}(\Sigma R/I,C) & \cong\Hom_{\mathcal{D}}(R/I,C)\cong\Hom_{R}(R/I,C)=0\\
\Ext_{\mathcal{H}_{\mathbf{x}}}^{1}(I,C) & \cong\Hom_{\mathcal{D}}(I,\Sigma C)\cong\Ext_{R}^{1}(I,C)=0,
\end{alignat*}
where $\Hom_{\mathcal{D}}(R/I,M)=0$ since $\im(f)\in\mathcal{C}_{I}\cap\mathcal{T}_{I}=0$
(recall that $\mathcal{C}_{I}$ is closed under subobjects and $\mathcal{T}_{I}$
is closed under quotients), and $\Ext_{R}^{1}(I,M)=0$ since $I\in\Proj(R)$.
Therefore, since $\mathcal{H}_{\x}=\Sigma\mathcal{T}_{I}\star\mathcal{C}_{I}$,
we can conclude that $P\in\Proj(\mathcal{H}_{\x})$. 
Finally, it can be proved that $P$ is a generator by using the Horseshoe Lemma together
with the fact that $R/I$ and $I$ are generators in $\mathcal{T}_{I}$
and $\mathcal{C}_{I}$ respectively. 
\end{proof}
Let us prove that the extended heart $\mathcal{C}_{\x}=\mathcal{H}_{\x}\star\Sigma^{-1}\mathcal{H}_{\x}$
is not AET4{*}. Note that this is an example for three $s$-torsion
pairs $\u_{1}\leq\u_{2}\leq\u_{3}$ in an extriangulated category
with negative first extension $\mathcal{D}$ such that $\mathcal{H}_{[\u_{1},\u_{2}]}$
and $\mathcal{H}_{[\u_{2},\u_{3}]}$ are AET4{*}, but $\mathcal{H}_{[\u_{1},\u_{3}]}=\mathcal{H}_{[\u_{1},\u_{2}]}\star\mathcal{H}_{[\u_{2},\u_{3}]}$
is not. 
\begin{prop}\label{exNAET4}
The extended heart $\mathcal{C}_{\x}=\mathcal{H}_{\x}\star\Sigma^{-1}\mathcal{H}_{\x}=\mathcal{W}\cap\Sigma^{-1}\mathcal{U}$
is not AET4{*}. 
\end{prop}

\begin{proof}
We have proved in Proposition \ref{prop:HRS1} that $\x$ is not co-smashing.
Thus,   it follows from Theorem \ref{thm:cosmashing}  that $\mathcal{C}_{\x}$
is not AET4{*}. 
\end{proof}

\section{Recollements and the AET4 condition}

It is known that, if there is a recollement of abelian categories
$(\mathcal{A},\mathcal{B},\mathcal{C})$, then condition AB4 in $\mathcal{B}$
is inherited by categories $\mathcal{A}$ and $\mathcal{C}$ 
(see \cite[Prop.3.5]{PV}). In this section, we will attempt to prove a
similar result for the condition AET4 in the more general setting of extriangulated categories. The following notions are inspired in \cite{WWZ}. 

\begin{defn}\label{def:exacto} Let $(\mathcal{C},\mathbb{E},\s)$ and $(\mathcal{D},\mathbb{F},\t)$
be extriangulated categories and $F:\mathcal{C}\rightarrow\mathcal{D}$ be an additive functor. It is said that $F$ is \textbf{right exact} if, for any conflation $\epsilon:\:\suc[C_{1}][C][C_{2}][a][b]$
in $\mathcal{C}$, there are conflations $\theta':\:\suc[B_{1}][FC_{1}][B_{2}][][x]$
and $\theta:\:\suc[B_{2}][FC][FC_{2}][y][Fb]$ in $\mathcal{D}$ such
that $y\circ x=Fa.$ Dually, we have the notion of a \textbf{left exact} functor between extriangulated categories.
\end{defn}

The following is an example of a right exact functor from a (non-abelian
non-triangulated) extriangulated category to an abelian category. 

\begin{example}
Let $R$ be a ring, $\D=\D(R)$ be the derived category of $\Mod(R)$ and $T\in \Mod(R)$ be a 1-tilting module. That
is, $\Gen(T)=T^{\bot_{1}}$. In this case, it is well-known that $\te=(\mathcal{T},\mathcal{F}):=(\Gen(T),T^{\bot_{0}})$
is a torsion pair. Moreover, $T$ admits a monomorphic projective
presentation $\suca[P_{1}][P_{0}][T][f]$. Let $P$ be the complex  
$
P:\cdots\rightarrow0\rightarrow P_{1}\stackrel{f}{\rightarrow}P_{0}\rightarrow0\rightarrow\cdots
$
with $H^{0}(P)=T.$ Consider $\mathcal{H}_{\te}:=\Sigma\mathcal{F}\star\mathcal{T}$
and $\mathcal{C}_{\te}:=\mathcal{H}_{\te}\star\Sigma^{-1}\mathcal{H}_{\te}$ in $\D.$
We claim that $F:=\Hom_{\C_\te}(\Sigma^{-1}T,-):\C_\te\rightarrow\text{Ab}$
is right exact. For this, we note that: 
\[
\Hom_{\mathcal{D}}(\Sigma^{-1}T,\Sigma H)\cong\Hom_{\mathcal{D}}(T,\Sigma^{2}H)\cong\Hom_{\mathcal{D}}(P,\Sigma^{2}H)\cong\Hom_{\mathcal{K}(R)}(P,\Sigma^{2}H)=0
\]
 for any $H\in\mathcal{H}_{\te}$. Moreover, since $T$ is a projective
generator in $\mathcal{H}_{\te}$ (see the proof of \cite[Thm.4.3]{HRS}),
we have that 
\[
\Hom_{\mathcal{D}(R)}(\Sigma^{-1}T,H)\cong\Hom_{\mathcal{D}(R)}(T,\Sigma H)\cong\Ext_{\mathcal{H}_{\te}}^{1}(T,H)=0.
\]
for any $H\in\mathcal{H}_{\te}$. Therefore
$\Hom_{\mathcal{D}(R)}(\Sigma^{-1}T,\Sigma C)=0$ for any $C\in\mathcal{C}_{\te}.$
Then, for a conflation $\suc[C_{1}][C_{2}][C_{3}][a][b]$
in $\mathcal{C}_{\te}$,  and since $\Hom_{\mathcal{D}(R)}(\Sigma^{-1}T,\Sigma C_{1})=0,$ we get the following exact sequence $FC_{1}\stackrel{Fa}{\rightarrow}FC_{2}\stackrel{Fb}{\rightarrow}FC_{3}\stackrel{}{\rightarrow} 0$ in $\Ab.$
\end{example}

\begin{defn}\label{def:rec} Let $\mathcal{A}$, $\mathcal{B}$ and $\mathcal{C}$
be extriangulated categories. A \emph{recollement} of $\mathcal{B}$
by $\mathcal{A}$ and $\mathcal{C}$ is a recollement of additive
categories (see Definition \ref{def:recad}) \\
\noindent\begin{minipage}[t]{1\columnwidth}%
\[
\begin{tikzpicture}[-,>=to,shorten >=1pt,auto,node distance=2cm,main node/.style=,x=2cm,y=-2cm]

\node (1) at (1,0) {$\mathcal{A}$};
\node (2) at (2,0) {$\mathcal{B}$};
\node (3) at (3,0) {$\mathcal{C}$};

\draw[-> , thin]  (1)  to  node  {$i_*$} (2);
\draw[-> , thin]  (2)  to  node  {$j^*$} (3);

\draw[-> , thin,in=45,out=135, above]  (3)  to  node  {$j_!$} (2);
\draw[-> , thin,in=45,out=135, above]  (2)  to  node  {$i^*$} (1);

\draw[-> , thin,in=-45,out=-135, below]  (3)  to  node  {$j_*$} (2);
\draw[-> , thin,in=-45,out=-135, below]  (2)  to  node  {$i^!$} (1);
 
\end{tikzpicture}
\]%
\end{minipage}\\
satisfying the following conditions:
\begin{enumerate}
\item [(ER4)]$i^{*}$ and $j_{!}$ are right exact; 
\item [(ER5)]$i^{!}$ and $j_{*}$ are left exact; 
\item [(ER6)]$i_{*}$ and $j^{*}$ are extriangulated; 
\item [(ER7)]the triple $(\Ker(i^{*}),\im(i_{*}),\Ker(i^{!}))$ is a TTF
triple in $\mathcal{B}$. That is, the following equalities hold:
$\Hom_{\mathcal{B}}(\im(i_{*}),\Ker(i^{!}))=0$, $\Hom_{\mathcal{B}}(\Ker(i^{*}),\im(i_{*}))=0$,
and $\mathcal{B}=\Ker(i^{*})\star\im(i_{*})=\im(i_{*})\star\Ker(i^{!})$.
\end{enumerate}
\end{defn}

\begin{rem}
\label{rem:natrec}$ $
\begin{enumerate}
\item The definition of right (resp. left) exact functor presented in this
paper is slightly different from the one introduced in \cite{WWZ}.
However, it should be noted that right (resp. left) exact functors
as defined in \cite{WWZ} are right (resp. left) exact as in Definition
\ref{def:exacto}. The same can be said (see Lemma \ref{lem:torsrec}) about the notion of recollement introduced in \cite{WWZ}.
\item In case $\mathcal{A}$, $\mathcal{B}$ and $\mathcal{C}$ are triangulated
(resp. abelian) categories, then the notion of recollement in Definition
\ref{def:exacto} coincides with the usual notion of recollement of
triangulated (resp. abelian) categories (see \cite{BBD,PV}). 
\item Consider a recollement of extriangulated categories as above. Note
that we have a natural isomorphism $_{4}\varphi:1_{\mathcal{C}}\rightarrow j^{*}\circ j_{!}$
(see Remark \ref{rem:rec}). Hence, it follows from Proposition \ref{prop:isovsextr}
that the functor $j^{*}\circ j_{!}:\mathcal{C}\rightarrow\mathcal{C}$
is extriangulated and that the natural transformation $_{4}\varphi:1_{\mathcal{C}}\rightarrow j^{*}\circ j_{!}$
is extriangulated. That is $\Gamma_{j^{*}j_{!}}(\eta)\cdot{}_{4}\varphi_{C_{2}}={}_{4}\varphi_{C_{1}}\cdot\eta\;$
for all $\eta\in\mathbb{E}_{\mathcal{C}}(C_{2},C_{1})$.
\end{enumerate}
\end{rem}

\begin{lem}\label{lem:torsadj} Let $(\mathcal{A},\mathbb{E},\s)$ and  $(\mathcal{B},\mathbb{F},\t)$
be extriangulated categories, and let $(S:\mathcal{A}\rightarrow\mathcal{B},T:\mathcal{B}\rightarrow\mathcal{A})$
be an adjoint pair. Then,
the following statements hold true. 
\begin{enumerate}
\item $\Hom_{\mathcal{B}}(\im(S),\Ker(T))=0$ and $\Hom_{\mathcal{A}}(\Ker(S),\im(T))=0.$ 
\item Let $T$ be left exact, $S$ be fully faithful and $\im(S)$ be closed
under extensions. Then, the equality $\mathcal{B}=\im(S)\star\Ker(T)$
holds true if, and only if, for each $B\in\mathcal{B}$,
$\psi_{B}:STB\rightarrow B$ is an inflation in $\mathcal{B}$ where
$\psi:S\circ T\rightarrow1_{\mathcal{B}}$ is the co-unit associated
to the adjoint pair $(S,T)$. Moreover, in this case, every $B\in\mathcal{B}$
admits a conflation $\suc[STB][B][F][\psi_{B}]$ with $TF=0$.
\item Let $S$ be right exact, $T$ be fully faithful and $\im(T)$ be
closed under extensions. Then, the equality $\mathcal{A}=\Ker(S)\star\im(T)$
holds true if, and only if, for each $A\in\mathcal{A}$, 
$\varphi_{A}:A\rightarrow TSA$ is a deflation in $\mathcal{A}$
where $\varphi:1_{\mathcal{A}}\rightarrow T\circ S$ is the unit associated
to the adjoint pair $(S,T)$. Moreover, in this case, every $A\in\mathcal{A}$
admits a conflation $\suc[F][A][TSA][][\varphi_{A}]$ with $SF=0$.
\end{enumerate}
\end{lem}

\begin{proof}
Note that (a) follows straightforward from the adjunction. Now, we only prove (b) 
since the proof of (c) follows by dual arguments. 

Since $S$ is fully faithful, we get that the unit $\varphi:1_{\mathcal{A}}\rightarrow T\circ S$ is an isomorphism by \cite[Prop.3.4.1]{B}. Therefore, by using that $T(\psi_{B})\circ\varphi_{TB}=1_{TB}$
and $\psi_{SA}\circ S(\varphi_{A})=1_{SA}$ for every $A\in\mathcal{A}$
and $B\in\mathcal{B}$, we have that $T(\psi_{B})$ and $\psi_{SA}$
are isomorphisms for all $A\in\mathcal{A}$ and $B\in\mathcal{B}$. 

$(\Rightarrow)$ Let $B\in\mathcal{B}$. Then, there is a
conflation $\eta:\:\suc[SA][B][F][f][g]$ in $\mathcal{B}$ with $TF=0$.
Recall that the map 
$
\phi:\Hom_{\mathcal{B}}(SA,B)\rightarrow\Hom_{\mathcal{A}}(A,TB)\text{, }h\mapsto Th\circ\varphi_{A},
$
is bijective with inverse 
$
\phi^{-1}:\Hom_{\mathcal{A}}(A,TB)\rightarrow\Hom_{\mathcal{B}}(SA,B)\text{, }w\mapsto\psi_{B}\circ Sw.
$
In particular, there is $f_{0}\in\Hom_{\mathcal{A}}(A,TB)$
such that $f=\psi_{B}\circ(Sf_{0})$. Consider a realization of the
morphism of extensions $(Sf_{0},1):\eta\rightarrow(Sf_{0})\cdot\eta$
\begin{equation}
\xymatrix{SA\ar[r]^{f}\ar[d]^{Sf_{0}} & B\ar[r]^{g}\ar[d]^{\alpha} & F\ar@{=}[d]\\
STB\ar[r]^{f'} & \tilde{B}\ar[r]^{g'} & F
}
\label{eq:diag-1}
\end{equation}
We assert that $(Sf_{0})\cdot\eta$ is realized by the sequence $\suc[STB][B][F][\psi_{B}][g]$.
For this, consider the following diagram
\[
\xymatrix{STB\ar[r]^{\psi_{B}}\ar@{=}[d] & B\ar[r]^{g}\ar[d]^{\alpha} & F\ar@{=}[d]\\
STB\ar[r]^{f'} & \tilde{B}\ar[r]^{g'} & F.
}
\]
We need to show that $\alpha\circ\psi_{B}=f'$ and that $\alpha$
is an isomorphism. Indeed, for the equality we proceed as follows. Note that
\[
g\circ\psi_{B}\in\Hom_{\mathcal{B}}(STB,F)\cong\Hom_{\mathcal{A}}(TB,TF)=\Hom_{\mathcal{A}}(TB,0)=0.
\]
Then by \cite[Prop.3.3]{NP} there is $\theta\in\Hom_{\mathcal{B}}(STB,SA)$
such that $f\circ\theta=\psi_{B}$. Note that 
$
\psi_{B}\circ(Sf_{0})\circ\theta=f\circ\theta={}\psi_{B}
$
 and thus $(Sf_{0})\circ\theta=1_{i_{*}i^{!}B}$ since $\phi^{-1}$
is a bijection. Therefore 
$
\alpha\circ\psi_{B}=\alpha\circ f\circ\theta=f'\circ(Sf_{0})\circ\theta=f'
$
and thus the above diagram commutes. Let us show that $\alpha$
is an isomorphism. For this, it is enough to prove that $Sf_{0}$
is an isomorphism (see diagram (\ref{eq:diag-1}) and \cite[Cor.3.6]{NP}).
Note that $f_{0}=\phi(f)=Tf\circ\varphi_{A}.$ Thus, since $\varphi_{A}$
is an isomorphism, it is enough to show that $Tf$ is an isomorphism.
We proceed as follows. On the one hand, from $f\circ\theta=\psi_{B}$,
we get that $Tf\circ T\theta=T\psi_{B}$ where $T\psi_{B}$ is an
isomorphism. Hence $Tf$ is a split epimorphism. On the other
hand, since $T$ is left exact, we have a conflation of the form $\suc[TSA][TB][W][Tf][x]$.
Now, considering the fact that $\psi_{B}\circ(STf)=f\circ\psi_{SA}$
together with (ET3), we get a morphism of extensions $({}\psi_{SA},q)$ which is
realized by the following diagram
\[
\xymatrix{STSA\ar[r]^{STf}\ar[d]^{\psi_{SA}} & STB\ar[r]^{Sx}\ar[d]^{\psi_{B}} & SW\ar[d]^{q}\\
SA\ar[r]^{f} & B\ar[r]^{g} & F.
}
\]
Since $\Hom_{\mathcal{B}}(SW,F)\cong\Hom_{\mathcal{A}}(W,TF)=0,$ it follows that 
 $q=0.$ Moreover, since $\psi_{SA}$ is an isomorphism,
one can check that $STf$ is a split monomorphism. Lastly, by using that 
$T\cdot{}_{3}\psi:(T\circ S\circ T)\rightarrow T$ is a natural isomorphism,
we conclude that $Tf$ is a split monomorphism and thus an
isomorphism as desired. 

$(\Leftarrow)$ Let $B\in\mathcal{B}$. By hypothesis, we have that
there is a conflation of the form $\eta:\:\suc[STB][B][F][\psi_{B}][g]$
in $\mathcal{B}$. Hence, it is enough to show that $TF=0.$ We assert that $Tg=0$. Indeed, since $T$ is left exact,
$\eta$ induces the conflations 
$
\suc[TSTB][TB][W][T(\psi_{B})][x]\:\text{ and }\;\suc[W][TF][W'][y]
$
 where $y\circ x=Tg$. Note that $W=0$ because $T(\psi_{B})$ is
an isomorphism, and thus $Tg=0$. Now, we show that 
$\eta\cdot\psi_{F}=0$. For this, consider a realization of the morphism
of extensions $(1,\psi_{F}):\eta\cdot\psi_{F}\rightarrow\eta$:
\[
\xymatrix{STB\ar[r]^{\alpha}\ar@{=}[d] & Z\ar[r]^{g'}\ar[d]^{h} & STF\ar[d]^{\psi_{F}}\\
STB\ar[r]^{\psi_{B}} & B\ar[r]^{g} & F.
}
\]
Here $Z=SA$ because $\im(S)$ is closed under extensions. Hence
$\alpha=Sa$ for some $a\in\Hom_{\mathcal{A}}(TB,A)$. Now, using that 
$T(\psi_{B})$ is an isomorphism and $Th\circ TSa=Th\circ T\alpha=T(\psi_{B})$,
we have that $TSa$ is a split monomorphism. Therefore $a$ is a split
monomorphism because $\varphi:1_{\mathcal{A}}\rightarrow T\circ S$
is a natural isomorphism. This implies that $\eta\cdot\psi_{F}=0$
as desired. Lastly, by \cite[Cor.3.5]{NP}, it follows that there
is a morphism $\theta:STF\rightarrow B$ such that $g\circ\theta=\psi_{F}$ 
and thus $0=Tg\circ T\theta=T(\psi_{F})$. This implies that $TF=0$
since $T(\psi_{F})$ is an isomorphism. 
\end{proof}

For an extriangulated category $\D$ and  $\mathcal{X}\subseteq\D$, we have the right $0$-perpendicular class
 $\mathcal{X}^{\bot_{0}}:=\{D\in\D\,|\:\Hom_{\D}(X,D)=0\;\forall X\in\mathcal{X}\}$ and the right $1$-perpendicular class  $\mathcal{X}^{\bot_{1}}:=\{D\in\D\,|\:\mathbb{E}(X,D)=0\;\forall X\in\mathcal{X}\}.$ Dually, we have ${}^{\perp_0}\mathcal{X}$ and ${}^{\perp_1}\mathcal{X}.$

\begin{cor}\label{cor:ttfrec} For a triple $(\mathcal{A},\mathcal{B},\mathcal{C})$ of extriangulated categories satisfying conditions (AR1), (AR2) and (AR3) from Definition \ref{def:recad}, the following statements hold true. 
\begin{enumerate}
\item $\Ker(i^{*})={}^{\bot_{0}}(\im(i_{*}))$ and $\Ker(i^{!})=(\im(i_{*}))^{\bot_{0}}.$
\item $\im(j_{!})\subseteq{}{}^{\bot_{1}}\Ker(j^{*})={}{}^{\bot_{1}}\im(i_{*})$
and $\im(j_{*})\subseteq\im(i_{*})^{\bot_{1}}=\Ker(j^{*})^{\bot_{1}}$. 
\end{enumerate}
\end{cor}

\begin{proof} (a) We have that $\Ker(i^{*})\subseteq{}^{\bot_{0}}(\im(i_{*}))$ and
$\Ker(i^{!})\subseteq(\im(i_{*}))^{\bot_{0}}$.  By Lemma \ref{lem:torsadj} (a). 
For $B\in{}^{\bot_{0}}(\im(i_{*}))$, consider the morphism $B\stackrel{_{2}\varphi_{B}}{\rightarrow}i_{*}i^{*}B$.
Since $_{2}\varphi_{B}=0$ and $i^{*}{_{2}}\varphi_{B}$
is an isomorphism (see  Remark \ref{rem:rec}(b)), we get that   $i^{*}B=0.$
Similarly, for $B\in(\im(i_{*}))^{\bot_{0}}$, we have that $i^{!}B=0$
since $i^{!}{}_{3}\psi_{B}:i^{!}i_{*}i^{!}B\rightarrow i^{!}B$ is
a null isomorphism.
\

(b) Consider $\eta\in\mathbb{E}_{\mathcal{B}}(j_{!}C,T)$ for $C\in\mathcal{C}$
and $T\in\im(i_{*})$, and let $\suc[T][B][j_{!}C][a][b]$ be a realization
of $\eta$. Since $\im(i_{*})=\Ker(j^{*})$, one can check that $j^{*}b$
is an isomorphism. Therefore, since $j_{!}j^{*}b$ and $_{4}\psi_{j_{!}C}$
are isomorphisms (see Remark \ref{rem:rec}(b)) and $_{4}\psi_{j_{!}C}\circ(j_{!}j^{*}b)=b\circ{}_{4}\psi_{B}$,
we have that $b$ is an split-epi  and thus $\eta=0$. \\
Similarly, for $\eta\in\mathbb{E}_{\mathcal{B}}(T,j_{*}C)$ realized
by $\suc[j_{*}C][B][T][a][b]$ with $C\in\mathcal{C}$ and $T\in\im(i_{*})=\Ker(j^{*})$,
one can check that $j^{*}a$ is an isomorphism and thus $a$
is an split-mono. Therefore $\eta=0$. 
\end{proof}

The following lemma proves that condition (ER7) in Definition \ref{def:rec}
can be replaced with the following one: $\mathcal{B}=\Ker(i^{*})\star\im(i_{*})=\im(i_{*})\star\Ker(i^{!})$.
Moreover, it gives an alternative conditions for these equalities. 

\begin{lem}\label{lem:torsrec} For a recollement $(\mathcal{A},\mathcal{B},\mathcal{C})$ of additive categories
as in Definition \ref{def:recad} satisfying conditions (ER4), (ER5) and (ER6) from Definition \ref{def:rec},
the following statements hold true. 
\begin{enumerate}
\item $\Hom_{\mathcal{B}}(\im(i_{*}),\Ker(i^{!}))=0$ and $\Hom_{\mathcal{B}}(\Ker(i^{*}),\im(i_{*}))=0$.
\item The equality $\mathcal{B}=\im(i_{*})\star\Ker(i^{!})$ holds true if, and
only if, for each $B\in\mathcal{B}$, $_{3}\psi_{B}:i_{*}i^{!}B\rightarrow B$
is an inflation in $\mathcal{B}$ where $_{3}\psi:i_{*}\circ i^{!}\rightarrow1_{\mathcal{B}}$
is the co-unit associated to the adjoint pair $(i_{*},i^{!})$. Moreover,
in this case, every $B\in\mathcal{B}$ admits a conflation $\suc[i_{*}i^{!}B][B][F][{_{3}\psi_{B}}]$
with $i^{!}F=0$.
\item The equality $\mathcal{B}=\Ker(i^{*})\star\im(i_{*})$ holds true if, and
only if, for each $B\in\mathcal{B},$ $_{2}\varphi_{B}:B\rightarrow i_{*}i^{*}B$
is a deflation in $\mathcal{B}$ where $_{2}\varphi:1_{\mathcal{B}}\rightarrow i_{*}\circ i^{*}$
is the unit associated to the adjoint pair $(i^{*},i_{*})$. Moreover,
in this case, every $B\in\mathcal{B}$ admits a conflation $\suc[C][B][i_{*}i^{*}B][][{_{2}\varphi_{B}}]$
with $i^{*}C=0$.
\end{enumerate}
\end{lem}

\begin{proof}
Item (a) follows from Corollary \ref{cor:ttfrec} (a). For item (b), consider
the adjoint pair $(i_{*},i^{!})$. By Definition \ref{def:rec}, we
have that $i^{!}$ is left exact and $i_{*}$ is fully faithful.
Moreover $\im(i_{*})=\Ker(j^{*})$ is closed under extensions. Hence
 (b) follows from Lemma \ref{lem:torsadj}(b). Finally (c) follows
with similar arguments. 
\end{proof}

In all that follows, we consider a recollement $(\mathcal{A},\mathcal{B},\mathcal{C})$ of extriangulated
categories as in the Definition \ref{def:rec}. Let us begin with
the following property of recollements. 

\begin{rem}
\label{rem:rec-1}Assume that $\mathcal{B}$ is WIC. By definition,
we have that $j_{!}$ is right exact. That is, for any conflation
$\epsilon:\:\suc[C_{1}][C][C_{2}][a][b]$ in $\mathcal{C}$, there
are conflations $\theta':\:\suc[B_{1}][j_{!}C_{1}][B_{2}][][x]$ and
$\theta:\:\suc[B_{2}][j_{!}C][j_{!}C_{2}][y][j_{!}b]$ in $\mathcal{D}$
such that $y\circ x=j_{!}a$. In what follows, we will show that (if
$\mathcal{B}$ is WIC) we can choose $\theta'$ in such a way that:
$j^{*}x$ is an isomorphism, $B_{1}\in\im(i_{*})$ and $\Gamma_{j^{*}}(\theta)=j^{*}x\cdot\Gamma_{j^{*}j_{!}}(\epsilon)$
(recall that $j^{*}\circ j_{!}$ is extriangulated by Remark \ref{rem:natrec} (c)). 
\

(a) The morphism $_{4}\psi_{B_{2}}:j_{!}j^{*}B_{2}\rightarrow B_{2}$
is a deflation. \\
Indeed, by the isomorphisms of the adjoint pair $(j_{!},j^{*})$,
we have that $x={}{}_{4}\psi_{B_{2}}\circ(j_{!}j^{*}x)\circ(j_{!}{}_{4}\varphi_{C_{1}})$. Hence, $_{4}\psi_{B_{2}}$ is a deflation since $x$ is so and $\mathcal{B}$ is WIC.
\
 
(b) There is an isomorphism $\tilde{\varphi}:C_{1}\rightarrow j^{*}B_{2}$
such that $(j^{*}y)\circ\tilde{\varphi}={}_{4}\varphi_{C}\circ a$.\\
This follows by considering axiom (ET3$^{op}$) and \cite[Cor.3.6]{NP}
in the following commutative diagram
\[
\xymatrix{C_{1}\ar[r]^{a} & C\ar[r]^{b}\ar[d]^{_{4}\varphi_{C}} & C_{2}\ar[d]^{_{4}\varphi_{C_{2}}}\\
j^{*}B_{2}\ar[r]^{j^{*}y} & j^{*}j_{!}C\ar[r]^{j^{*}j_{!}b} & j^{*}j_{!}C_{2}.
}
\]
(c) The morphism $\tilde{x}:={_{4}}\psi_{B_{2}}\circ j_{!}\tilde{\varphi}:j_{!}C_{1}\rightarrow B_{2}$
is a deflation such that $j^{*}\tilde{x}$ is an isomorphism and
$y\circ\tilde{x}=j_{!}a$. \\
Indeed, $_{4}\psi_{B_{2}}$ is a deflation by (a), and thus $\tilde{x}$
is a deflation by Lemma \ref{lem:infdef}. It is clear that $j^{*}x$
is an isomorphism since $j^{*}j_{!}\tilde{\varphi}$ is an isomorphism
by (b) and  $j^{*}{_{4}}\psi_{B_{2}}$ is an isomorphism by Remark
\ref{rem:rec}(b). Lastly, we observe that 
\[
y\circ\tilde{x}=y\circ{_{4}}\psi_{B_{2}}\circ j_{!}\tilde{\varphi}={_{4}}\psi_{j_{!}C}\circ j_{!}j^{*}y\circ j_{!}\tilde{\varphi}={_{4}}\psi_{j_{!}C}\circ j_{!}{_{4}}\varphi_{C}\circ j_{!}a=j_{!}a.
\]
(d) By (c) there is a conflation $\tilde{\theta'}:\:\suc[\tilde{B_{1}}][j_{!}C_{1}][B_{2}][][\tilde{x}]$
such that $j^{*}\tilde{x}$ is an isomorphism and $y\circ\tilde{x}=j_{!}a$.
\

(e) $\tilde{B_{1}}=i_{*}A$ for some $A\in\mathcal{A}$. \\
For this, we can check that $j^{*}\tilde{B_{1}}=0$ since $j^{*}\tilde{x}$
is an isomorphism in $\Gamma_{j^{*}}(\tilde{\theta}')$. Then, $\tilde{B_{1}}=i_{*}A$
by (AR3). 
\

(f) $\Gamma_{j^{*}}(\theta)=j^{*}\tilde{x}\cdot\Gamma_{j^{*}j_{!}}(\epsilon)$.
\\
To prove this, we proceed as follows. By Remark \ref{rem:natrec} (c),
$j^{*}\circ j_{!}$ is extriangulated. Therefore, we can consider
the conflation $\Gamma_{j^{*}j_{!}}(\epsilon):\:\suc[j^{*}j_{!}C_{1}][j^{*}j_{!}C][j^{*}j_{!}C_{2}][j^{*}j_{!}a][j^{*}j_{!}b]$.
Now, by the dual of \cite[Lem.2.5]{E}, the morphism of extensions
$(j^{*}\tilde{x},1):\Gamma_{j^{*}j_{!}}(\epsilon)\rightarrow j^{*}\tilde{x}\cdot\Gamma_{j^{*}j_{!}}(\epsilon)$
is realized by a commutative diagram as the one below
\[
\xymatrix{j^{*}j_{!}C_{1}\ar[r]^{j^{*}j_{!}a}\ar[d]^{j^{*}\tilde{x}} & j^{*}j_{!}C\ar[r]^{j^{*}j_{!}b}\ar[d]^{x'} & j^{*}j_{!}C_{2}\ar@{=}[d]\\
j^{*}B_{2}\ar[r]^{a'} & E\ar[r]^{b'} & j^{*}j_{!}C_{2}
}\]
where the square
on the left is a weak pushout. 
Thus, since $(j^{*}y)\circ(j^{*}\tilde{x})=1_{j^{*}j_{!}C}\circ(j^{*}j_{!}a)$,
there is a morphism $y':E\rightarrow j^{*}j_{!}C$ such that $y'\circ x'=1_{j^{*}j_{!}C}$
and $y'\circ a'=j^{*}y$. Moreover, by \cite[Cor.3.6]{NP}, we can
conclude that $x'$ and $y'$ are isomorphisms. Then, since $j^{*}\tilde{x}\cdot\Gamma_{j^{*}j_{!}}(\epsilon)$
is realized by $\suc[j^{*}B_{2}][E][j^{*}j_{!}C_{2}][a'][b']$ and
$y'$ is an isomorphism, it is enough to prove that the following
diagram commutes
\[
\xymatrix{j^{*}B_{2}\ar[r]^{a'}\ar@{=}[d] & E\ar[r]^{b'}\ar[d]^{y'} & j^{*}j_{!}C_{2}\ar@{=}[d]\\
j^{*}B_{2}\ar[r]^{j^{*}y} & j^{*}j_{!}C\ar[r]^{j^{*}j_{!}b} & j^{*}j_{!}C_{2}.
}
\]
Indeed, the left square on the above diagram commutes by construction.
Moreover $b'\circ x'=(j^{*}j_{!}b)=(j^{*}j_{!}b)\circ y'\circ x'.$
Therefore $b'=(j^{*}j_{!}b)\circ y'$ since $x'$ is an epimorphism.
\

(g) $B_{2}\in\Ker(i^{*})$. \\
To see this, consider the following exact sequence given by $\theta$
for $T\in\im(i_{*})=\Ker(j^{*})$: 
\[
\Hom_{\mathcal{B}}(j_{!}C,T)\rightarrow\Hom_{\mathcal{B}}(B_{2},T)\rightarrow\mathbb{E}_{\mathcal{B}}(j_{!}C_{2},T).
\]
Note that $\Hom_{\mathcal{B}}(j_{!}C,T)\cong\Hom_{\mathcal{C}}(C,j^{*}T)=0$
and  $\mathbb{E}_{\mathcal{B}}(j_{!}C_{2},T)=0$ by Corollary
\ref{cor:ttfrec}(b). Therefore $B_{2}\in{}^{\bot_{0}}\im(i_{*})=\Ker(i^{*})$
by Corollary \ref{cor:ttfrec}(a). 
\end{rem}

\begin{lem}\label{lem:inf} Let $f$ be a morphism in $\mathcal{A}$. Then $f$
is an inflation if, and only if, $i_{*}f$ is an inflation.
\end{lem}

\begin{proof} If $f$ is an inflation, then $i_{*}f$ is so since $i_{*}$ is extriangulated. 
\

Let $i_{*}f$ be an inflation. Then $(i^{!}\circ i_{*})f$ is an inflation since $i^{!}$ preserves inflations. Thus, by
using that $_{3}\varphi:1_{\mathcal{A}}\rightarrow i^{!}\circ i_{*}$
is an isomorphism, it follows that $f$ is an inflation by Lemma \ref{lem:infdef}. 
\end{proof}
\begin{lem}\label{lem:R} For $\eta\in\mathbb{E}_{\mathcal{B}}(i_{*}A_{1},i_{*}A_{2})$
there is a conflation $\delta:\:\suc[A_{2}][A][A_{1}][a][b]$ in $\mathcal{A}$
such that $\eta$ is realized by $\suc[i_{*}A_{2}][i_{*}A][i_{*}A_{1}][i_{*}a][i_{*}b]$.
\end{lem}

\begin{proof}
Let $\suc[i_{*}A_{2}][B][i_{*}A_{1}][f][g]$ be a realization of $\eta$.
Since $j^{*}$ is extriangulated, we have that 
$
\suc[(j^{*}\circ i_{*})A_{2}][j^{*}B][(j^{*}\circ i_{*})A_{1}][j^{*}f][j^{*}g]
$
is a conflation in $\mathcal{C}$. Observe that $B\in\Ker(j^{*})=\im(i_{*})$
since $(j^{*}\circ i_{*})=0$. Hence, there is $A\in\mathcal{A}$
such that $B=i_{*}A$. Thus $f=i_{*}(a)$ and $g=i_{*}(b)$ by
condition (AR2). Let us prove that $\suc[A_{2}][A][A_{1}][a][b]$
is a conflation in $\mathcal{A}$. Firstly, note that $a$ is an inflation
by Lemma \ref{lem:inf}. Hence, there is a conflation $\suc[A_{2}][A][A'][a][b']$
in $\mathcal{A}$ and thus we get a conflation
$\suc[i_{*}A_{2}][i_{*}A][i_{*}A'][i_{*}a][i_{*}b']$. Note that $i_{*}a=f.$ Then,
 there is an isomorphism $\beta:i_{*}A_{1}\rightarrow i_{*}A'$
such that $\beta\circ(i_{*}b)=i_{*}b'$ by \cite[Cor.3.6]{NP}. Hence
$i^{*}\beta:i^{*}i_{*}A_{1}\rightarrow i^{*}i_{*}A'$ is an isomorphism
in $\mathcal{A}$ such that $i^{*}\beta\circ(i^{*}i_{*}b)=i^{*}i_{*}b'$.
Lastly, recall that ${}{}_{2}\psi:i^{*}\circ i_{*}\rightarrow1_{\mathcal{A}}$
is an isomorphism. Therefore, we have that $\gamma:={}_{2}\psi_{A'}\circ(i^{*}\beta)\circ{}_{2}\psi_{A_{1}}^{-1}$
is an isomorphism such that $\gamma\circ b=b'$ (see the diagram below)
\[
\xymatrix{i^{*}i_{*}A\ar[r]^{i^{*}i_{*}b}\ar[d]^{_{2}\psi_{A}}\ar@(u,u)[rr]^{i^{*}i_{*}b'} & i^{*}i_{*}A_{1}\ar[r]^{i^{*}\beta}\ar[d]^{_{2}\psi_{A_{1}}} & i^{*}i_{*}A'\ar[d]^{_{2}\psi_{A'}}\\
A\ar[r]^{b}\ar@(d,d)[rr]_{b'} & A_{1}\ar[r]^{\gamma} & A'
}
\]
Then, it follows from \cite[Prop.3.7]{NP} that $\suc[A_{2}][A][A_{1}][a][b]$
is a conflation in $\mathcal{A}$.
\end{proof}

\begin{prop}\label{prop:Giso} If $\mathcal{B}$ is AET4, then $\Gamma_{i_{*}}:\mathbb{E}_{\mathcal{A}}\rightarrow\mathbb{E}_{\mathcal{B}}\circ(i_{*}^{op}\times i_{*})$
is an isomorphism. 
\end{prop}

\begin{proof} Let $\mathcal{B}$ be AET4 and  $A_{1},A_{2}\in\mathcal{A}.$ Let us show firstly that the map given by $\Gamma:=\Gamma_{i_{*}}:\mathbb{E}_{\mathcal{A}}(A_{1},A_{2})\rightarrow\mathbb{E}_{\mathcal{B}}(i_{*}A_{1},i_{*}A_{2})$
is injective. Indeed, let $\eta:\:\suc[A_{2}][A][A_{1}][a][b]$ be 
 a conflation such that $\Gamma(\eta)=0.$ Then, by \cite[Cor.3.5]{NP} we get that $i_{*}(a)$ is a split-mono. Furthermore, by Remark \ref{rem:rec} (b) there is a natural
isomorphism ${}_{3}\varphi:1_{\mathcal{A}}\rightarrow i^{!}\circ i_{*}$.
Hence $a$ is a split-mono since $a={}_{3}\varphi_{B}^{-1}\circ\left(i^{!}i_{*}a\right)\circ{}_{3}\varphi_{A}$ and thus  $\eta=0$ by \cite[Cor.3.5]{NP}.

It remains to prove that $\Gamma$ is surjective. For this, consider
the objects $B_{1}:=i_{*}A_{1}$ and $B_{2}:=i_{*}A_{2}$ in $\mathcal{B}$.
Since $\mathcal{B}$ is AET4, it follows from Theorem A that there
is a universal $\mathbb{E}$-extension $\eta\in\mathbb{E_{\mathcal{B}}}(B_{1}^{(X)},B_{2})$.
Recall that $B_{2}=i_{*}A_{2}$ and $B_{1}^{(X)}=i_{*}(A_{1}^{(X)})$
(see Remark \ref{rem:rec} (c)). Hence, by Lemma \ref{lem:R}, there
is a conflation $\delta:\:\suc[A_{2}][A][A_{1}^{(X)}][a][b]$ in $\mathcal{A}$
such that $\eta$ is realized by $\suc[i_{*}A_{2}][i_{*}A][i_{*}A_{1}^{(X)}][i_{*}a][i_{*}b]$. 

Let us prove that $\Gamma(\delta)$ is a universal $\mathbb{E}$-extension
in $\mathcal{B}$. By (ET3), there is a morphism $h:i_{*}A_{1}^{(X)}\rightarrow i_{*}A_{1}^{(X)}$
such that $\eta=\Gamma(\delta)\cdot h$. Now, since $\eta$ is a universal
$\mathbb{E}$-extension, for every $\epsilon\in\mathbb{E}_{\mathcal{B}}(i_{*}A_{1},i_{*}A_{2})$
there is a morphism $f_{\epsilon}:i_{*}A_{1}\rightarrow i_{*}A_{1}^{(X)}$
such that $\eta\cdot f_{\epsilon}=\epsilon$. Thus $\Gamma(\delta)$
is a universal $\mathbb{E}$-extension because $\epsilon=\eta\cdot f_{\epsilon}=\Gamma(\delta)\cdot h\cdot f_{\epsilon}$
for all $\epsilon\in\mathbb{E}_{\mathcal{B}}(i_{*}A_{1},i_{*}A_{2})$.

Now, note that we have the following commutative diagram 
\[
\xymatrix{_{\mathcal{A}}(A_{1},A_{2})\ar[r]^{_{\mathcal{A}}(A_{1},a)}\ar[d]^{i_{*}} & _{\mathcal{A}}(A_{1},A)\ar[r]^{_{\mathcal{A}}(A_{1},b)}\ar[d]^{i_{*}} & _{\mathcal{A}}(A_{1},A_{1}^{(X)})\ar[r]^{\delta\cdot-}\ar[d]^{i_{*}} & \mathbb{E}_{\mathcal{A}}(A_{1},A_{2})\ar[d]^{\Gamma}\\
_{\mathcal{B}}(i_{*}A_{1},i_{*}A_{2})\ar[r]^{\Hom_{\mathcal{B}}(i_{*}A_{1},i_{*}a)} & _{\mathcal{B}}(i_{*}A_{1},i_{*}A)\ar[r]^{\Hom_{\mathcal{B}}(i_{*}A_{1},i_{*}b)} & _{\mathcal{B}}(i_{*}A_{1},i_{*}A_{1}^{(X)})\ar[r]^{\Gamma(\delta)\cdot-} & \mathbb{E}_{\mathcal{B}}(i_{*}A_{1},i_{*}A_{2})
}
\]
where $\Gamma(\delta)\cdot-$ is surjective because $\Gamma(\delta)$
is a universal $\mathbb{E}$-extension. Indeed, the right square on the
above diagram commutes since 
$
\Gamma(\delta\cdot f)=\Gamma(\delta)\cdot i_{*}(f)
$
for all $f\in\Hom_{\mathcal{A}}(A_{1},A_{1}^{(X)})$ (see Definition
\ref{def:funextr}). Finally, by using that $i_{*}$ is fully faithful, we
can check by diagram chasing that $\Gamma$ is an epimorphism. 
\end{proof}
\begin{thm}\label{M1RETC}
Let $(\mathcal{A},\mathcal{B},\mathcal{C})$ be a recollement of extriangulated categories such that $\mathcal{A}$ is coproduct-compatible (resp. product compatible) and $\mathcal{B}$ is AET4 (resp. AET4{*}). Then $\mathcal{A}$ is AET4 (resp. AET4{*}). 
\end{thm}

\begin{proof}
Let us prove the statement for AET4 since the case AET4{*} follows
by duality. 
By Theorem A, it is enough to show that the natural transformation
\[
\tau_{\mathcal{A}}:\mathbb{E}_{\mathcal{A}}(\coprod_{i\in I}X_{i},Y)\rightarrow\prod_{i\in I}\mathbb{E}_{\mathcal{A}}(X_{i},Y)\text{, }\epsilon\mapsto\left(\epsilon\cdot\mu_{i}^{X}\right)_{i\in I}
\]
 is an isomorphism. For this, recall that for a family $\{A_{i}\}_{i\in I}$ in 
 $\mathcal{A},$ we have that 
$i_{*}\Big(\coprod_{i\in I}^{\mathcal{A}}X_{i}\Big)=\coprod_{i\in I}^{\mathcal{B}}i_{*X_{i}}$ (see \ref{rem:rec} (c)). Hence, for any $Y\in\mathcal{A}$,
we have the following commutative diagram
\[
\xymatrix{\mathbb{E}_{\mathcal{A}}(\coprod_{i\in I}X_{i},Y)\ar[r]^{\tau_{\mathcal{A}}}\ar[d]^{\Gamma_{i_{*}}} & \prod_{i\in I}\mathbb{E}_{\mathcal{A}}(X_{i},Y)\ar[d]^{\Gamma_{i_{*}}}\\
\mathbb{E}_{\mathcal{B}}(\coprod_{i\in I}^{\mathcal{B}}i_{*}X_{i},i_{*}Y)\ar[r]^{\tau_{\mathcal{B}}} & \prod_{i\in I}\mathbb{E}_{\mathcal{B}}(X_{i},Y).
}
\]
Note that, by Proposition
\ref{prop:Giso},  the columns in the above diagram are isomorphisms. Moreover,  the row in the bottom is an isomorphism by Theorem
A. Therefore,  we can conclude that $\tau_{\mathcal{A}}$ is an isomorphism as
desired. 
\end{proof}

\begin{thm}\label{M2RETC}
Let $(\mathcal{A},\mathcal{B},\mathcal{C})$ be a recollement of extriangulated categories such that $\mathcal{C}$ is coproduct-compatible (resp. product-compatible)
and $\mathcal{B}$
is AET4 (resp. AET4{*}) and WIC. Then $\mathcal{C}$ is AET4
(resp. AET4{*}). 
\end{thm}

\begin{proof}
By Theorem A, it is enough to show that $\mathcal{C}$ has universal
$\mathbb{E}_\C$-extensions. Consider $C_{1},C_{2}\in\mathcal{C}$ and
a conflation $\epsilon:\:\suc[C_{1}][C][C_{2}][a][b]$ in $\mathcal{C}$. 
By Lemma \ref{lem:torsrec}, there is a conflation
in $\mathcal{B}$ given by $\theta_{1}:\:\suc[i_{*}i^{!}j_{!}C_{1}][j_{!}C_{1}][B][{_{3}\psi_{j_{!}C_{1}}}][\omega].$
Moreover, by (ER4) there are conflations $\theta_{2:\:}\suc[B_{1}][j_{!}C_{1}][B_{2}][\alpha][x]$
and $\theta_{3}:\:\suc[B_{2}][j_{!}C][j_{!}C_{2}][y][j_{!}b]$ in
$\mathcal{B}$ such that $y\circ x=j_{!}a$. Here, $B_{1}=i_{*}A$
for some $A\in\mathcal{A}$ by Remark \ref{rem:rec-1}. Recall that
we have the isomorphism 
$
\Hom_{\mathcal{B}}(i_{*}A,j_{!}C_{1})\cong\Hom_{\mathcal{A}}(A,i^{!}j_{!}C_{1})
$
given by the map $f\mapsto(i^{!}f)\circ{}_{3}\varphi_{A}$ with inverse
given by $g\mapsto{}_{3}\psi_{j^{!}C_{1}}\circ(i_{*}g)$. In particular
$
\alpha={}_{3}\psi_{j^{!}C_{1}}\circ(i_{*}i^{!}\alpha)\circ i_{*}({}{}_{3}\varphi_{A});
$
 and thus, by (ET3)$^{op}$ there is a morphism $\beta:B'\rightarrow B$
such that $\theta_{1}\cdot\beta=\tilde{\alpha}\cdot\theta_{2}$, where
$\tilde{\alpha}=(i_{*}i^{!}\alpha)\circ i_{*}({}{}_{3}\varphi_{A})$
\[
\xymatrix{B_{1}\ar[r]^{\alpha}\ar[d]^{\tilde{\alpha}} & j_{!}C_{1}\ar[r]^{x}\ar@{=}[d] & B_{2}\ar[d]^{\beta}\\
i_{*}i^{!}j_{!}C_{1}\ar[r]^{_{3}\psi_{j_{!}C_{1}}} & j_{!}C_{1}\ar[r]^{\omega} & B.
}
\]
We assert that $j^{*}\beta$ is an isomorphism. Indeed, $j^{*}\omega$
is an isomorphism since $j^{*}i_{*}i^{!}j_{!}C_{1}=0$ and $j^{*}x$
is an isomorphism by Remark \ref{rem:rec-1}. Therefore $j^{*}\beta=j^{*}\omega\circ(j^{*}x)^{-1}$
is an isomorphism. 

Let $\theta:=\beta\cdot\theta_{3}\in\mathbb{E}_{\mathcal{B}}(j_{!}C_{2},B)$.
If $\eta\in\mathbb{E}_{\mathcal{B}}(j_{!}C_{2}^{(X)},B)$ is a universal
$\mathbb{E}$-extension, then there is a morphism $f\in\Hom_{\mathcal{B}}(j_{!}C_{2},j_{!}C_{2}^{(X)})$
such that $\eta\cdot f=\theta;$ and thus, 
\begin{alignat*}{1}
\Gamma_{j^{*}}(\eta)\cdot(j^{*}f) & =\Gamma_{j^{*}}(\theta)\\
 & =\Gamma_{j^{*}}(\beta\cdot\theta_{3})\\
 & =(j^{*}\beta)\cdot\Gamma_{j^{*}}(\theta_{3})\\
 & =\left(j^{*}(\beta)\circ j^{*}(x)\right)\cdot\Gamma_{j^{*}j_{!}}(\epsilon)\\
 & =j^{*}(\omega)\cdot\left({}{}_{4}\varphi_{C_{1}}\cdot\epsilon\cdot{}_{4}\varphi_{C_{2}}^{-1}\right),
\end{alignat*}
where the last two equalities follow from Remark \ref{rem:rec-1}
and Remark \ref{rem:natrec} (c). Hence
\[
\epsilon={}{}_{4}\varphi_{C_{1}}^{-1}\cdot j^{*}(\omega)^{-1}\cdot\Gamma_{j^{*}}(\eta)\cdot(j^{*}f)\cdot{}_{4}\varphi_{C_{2}}.
\]
 Therefore, ${}{}_{4}\varphi_{C_{1}}^{-1}\cdot j^{*}(\omega)^{-1}\cdot\Gamma_{j^{*}}(\eta)$
is a universal $\mathbb{E}$-extension of $C_{2}$ by $C_{1}$. 
\end{proof}

\appendix

\section{Extriangulated adjoint pairs}

Let $S:\mathcal{D}\rightarrow\mathcal{C}$ and $T:\mathcal{C}\rightarrow\mathcal{D}$
be functors. If there is a natural equivalence $\Hom_{\mathcal{C}}(S(?),-)\rightarrow\Hom_{\mathcal{D}}(?,T(-))$
of bifunctors, we say that $(S,T)$ is an \textbf{adjoint pair between
$\mathcal{C}$ and $\mathcal{D}$}. It is well-known that this is
equivalent to the existence of natural transformations $\varphi:1_{\mathcal{D}}\rightarrow T\circ S$
and $\psi:S\circ T\rightarrow1_{\mathcal{C}}$ such that $\psi_{SD}\circ S(\varphi_{D})=1_{SD}$
and $T(\psi_{C})\circ\varphi_{TC}=1_{TC}$ for all $D\in\mathcal{D}$
and $C\in\mathcal{C}$ (see \cite[Thm. 3.1.5]{B}).

In this section, we will see under which conditions an adjoint pair
$(S,T)$ between extriangulated categories $(\mathcal{C},\mathbb{E},\s)$
and $(\mathcal{D},\mathbb{F},\t)$ induces a natural isomorphism $\mathbb{E}\circ(S^{op}\times1_{\C})\xrightarrow{\sim}\mathbb{F}\circ(1_{\D}^{op}\times T)$.
It is worth mentioning that a similar isomorphism has been built in
\cite[Lemma 2.16]{WWZ} under the assumption that the categories have
enough projectives or enough injectives. In contrast, our construction
will not require these assumptions. Similar results have also been
studied in \cite[Prop.2.5]{A} and \cite[Lem.5.1]{HJ} for abelian
categories,  in \cite[Lem.2.1]{ECO} for exact categories, 
and in \cite[Sec.3.7]{BGLS} for extriangulated categories. Let us begin with the following definition. 

\begin{defn}
Let $(\mathcal{C},\mathbb{E},\s)$ and $(\mathcal{D},\mathbb{F},\t)$
be extriangulated categories, and consider functors $T:\mathcal{C}\rightarrow\mathcal{D}$,
$S_{1}:\mathcal{X}\rightarrow\mathcal{C}$ and $S_{2}:\mathcal{Y}\rightarrow\mathcal{C}$.
We say that $T$ is \textbf{$(S_{1},S_{2})$-extriangulated} if $T$
is additive and there is a natural transformation 
\[
\Gamma_{T}^{(S_{1},S_{2})}:\mathbb{E}\circ(S_{1}^{op}\times S_{2})\rightarrow\mathbb{F}(T^{op}S_{1}^{op}\times TS_{2})
\]
such that $\t\Big((\Gamma_{T}^{(S_{1},S_{2})})_{(S_{1}X,S_{2}Y)}(\eta)\Big)=T(\s(\eta)),$
for $\eta\in\Ebb(S_{1}X,S_{2}Y).$ 
\end{defn}

\begin{rem}\label{rem:funext}Let $(\mathcal{C},\mathbb{E},\s)$ and $(\mathcal{D},\mathbb{F},\t)$
be extriangulated categories, and consider the functors $T:\mathcal{C}\rightarrow\mathcal{D}$
and $S:\mathcal{D}\rightarrow\mathcal{C}$. 
\begin{enumerate}
\item If $T$ is $(S,1)$-extriangulated, then $T(g)\cdot\Gamma_{T}^{(S,1)}(\eta)\cdot TS(f)=\Gamma_{T}^{(S,1)}(g\cdot\eta\cdot Sf)$
for all $\eta\in\mathbb{E}(SA,B)$, $f\in\Hom_{\mathcal{D}}(A',A)$
and $g\in\Hom_{\mathcal{C}}(B,B')$. 
\item If $S$ is $(1,T)$-extriangulated, then $ST(g)\cdot\Gamma_{S}^{(T,1)}(\eta)\cdot S(f)=\Gamma_{S}^{(T,1)}(Tg\cdot\eta\cdot f)$
for all $\eta\in\mathbb{F}(A,TB)$, $f\in\Hom_{\mathcal{D}}(A',A)$
and $g\in\Hom_{\mathcal{C}}(B,B')$. 
\item If $T$ is $(S,1)$-extriangulated and $S$ is $(1,T)$-extriangulated,
then: 
\begin{enumerate}
\item the composition $S\circ T$ is $(S,1)$-extriangulated via 
\[
\Gamma_{S\circ T}^{(S,1)}:=(\Gamma_{S}^{(1,T)}\cdot(T^{op}S^{op}\times1))\circ\Gamma_{T}^{(S,1)}.
\]
\item the composition $T\circ S$ is $(1,T)$-extriangulated via 
\[
\Gamma_{T\circ S}^{(1,T)}:=(\Gamma_{S}^{(1,T)}\cdot(1\times ST))\circ\Gamma_{S}^{(1,T)}.
\]
\end{enumerate}
\end{enumerate}
\end{rem}

\begin{defn}
Let $(\mathcal{C},\mathbb{E},\s)$, $(\mathcal{D},\mathbb{F},\t)$
and $(\mathcal{D}',\mathbb{F}',\t')$ be extriangulated categories,
and consider functors $T:\mathcal{C}\rightarrow\mathcal{D}$, $T':\mathcal{C}\rightarrow\mathcal{D}$,
$S_{1}:\mathcal{X}\rightarrow\mathcal{C}$ and $S_{2}:\mathcal{Y}\rightarrow\mathcal{C}$.
A \textbf{natural transformation of $(S_{1},S_{2})$-extriangulated
functors} $\alpha:(T,\Gamma_{T}^{(S_{1},S_{2})})\rightarrow(T',\Gamma_{T'}^{(S_{1},S_{2})})$
is a natural transformation $\alpha:T\rightarrow T'$ such that $\alpha_{S_{2}Y}\cdot\Gamma_{T}^{(S_{1},S_{2})}(\eta)=\Gamma_{T'}^{(S_{1},S_{2})}(\eta)\cdot\alpha_{S_{1}X}$.
In such case, we will say that $\alpha$ is \textbf{$(S_{1},S_{2})$-extriangulated
}for short. 
\end{defn}

\begin{example}\label{ExA}
Let $(\mathcal{C},\mathbb{E},\s)$ and $(\mathcal{D},\mathbb{F},\t)$
be extriangulated categories, and let $(S:\D\to \C,T:\C\to \D)$ be an adjoint pair. 
\begin{enumerate}
\item \cite[Sec.2]{ECO} Let $(\mathcal{C},\mathbb{E},\s)$ and $(\mathcal{D},\mathbb{F},\t)$ be exact categories. Then $T$ is $(S,1)$-extriangulated if $T$ preserves
conflations of the form $\suc[A][B][SX][f][g].$

\item If $T$ is extriangulated, then $T$ is $(S,1)$-extriangulated. Indeed,
it can be seen that, for the natural transformation 
\[
\Gamma_{T}^{(S,1)}:=\Gamma_{T}\cdot(S^{op}\times1):\mathbb{E}(S(?),-)\rightarrow\mathbb{F}(TS(?),T(-))
\]
$\t(\Gamma_{T}^{(S,1)}(\eta))=[\suc[TA][TB][TSX][Tf][Tg]]$ if $\s(\eta)=[\suc[A][B][SX][f][g]].$

\item If $S$ is extriangulated, then $S$ is $(1,T)$-extriangulated. 
\item If $\alpha$ is a natural transformation of extriangulated functors,
then it is a natural transformation of $(S_{1},S_{2})$-extriangulated
functors for any $S_{1}$ and $S_{2}$. 
\end{enumerate}
\end{example}

Let $(S:\D\to \C,T:\C\to \D)$ be an adjoint pair between extriangulated categories
such that $T$ is $(S,1)$-extriangulated. Then, we have
the natural transformation 
$
\Gamma_{T}^{(S,1)}:\mathbb{E}(S(?),-)\rightarrow\mathbb{F}(TS(?),T(-))
$ and the unit $\varphi:1_{\mathcal{D}}\rightarrow T\circ S$
induces the natural transformation 
\[
(\mathbb{F}\circ(1\times T))\cdot(\varphi^{op}\times1):\mathbb{F}(TS(?),T(-))\rightarrow\mathbb{F}(?,T(-)).
\]
By taking the composition $(\mathbb{F}\circ(1\times T))\cdot(\varphi^{op}\times1))\circ\Gamma_{T}^{(S,1)}$,
we define 
\[
\tau:\mathbb{E}(S(?),-)\rightarrow\mathbb{F}(?,T(-)).
\]
If $S$ is $(1,T)$-extriangulated, we can build
the natural transformation 
\[
\sigma:\mathbb{F}(?,T(-))\rightarrow\mathbb{E}(S(?),-),
\]
defined as $\sigma:=((\mathbb{E}\circ(S^{op}\times1))\cdot(1\times\psi))\circ\Gamma_{S}^{(1,T)}$. Therefore, We have proved the following result. 

\begin{prop}\label{prop:existe} Let $(\mathcal{C},\mathbb{E},\s)$ and $(\mathcal{D},\mathbb{F},\t)$
be extriangulated categories, and let $(S:\D\to \C,T:\C\to D)$ be an adjoint pair.
\begin{enumerate}
\item If $T$ is $(S,1)$-extriangulated, then there is a natural transformation
\[
\tau:\mathbb{E}(S(?),-)\rightarrow\mathbb{F}(?,T(-))
\]
defined as $\tau(\eta)=\Gamma_{T}^{(S,1)}(\eta)\cdot\varphi_{D}$
for all $\eta\in\mathbb{E}(S(D),C)$. 
\item If $S$ is $(1,T)$-extriangulated, then there is a natural transformation
\[
\sigma:\mathbb{F}(?,T(-))\rightarrow\mathbb{E}(S(?),-)
\]
defined as $\sigma(\eta)=\psi_{C}\cdot\Gamma_{S}^{(1,T)}(\eta)$ for
all $\eta\in\mathbb{F}(D,T(C))$. 
\end{enumerate}
\end{prop}

Let us prove that $\tau$ and $\sigma$ always are monomorphisms. 

\begin{prop}\label{prop:mono} Let $(\mathcal{C},\mathbb{E},\s)$ and $(\mathcal{D},\mathbb{F},\t)$
be extriangulated categories, and let $(S:\D\to\C,T:\C\to \D)$ be an adjoint pair. 
\begin{enumerate}
\item If $T$ is $(S,1)$-extriangulated, then $\tau:\mathbb{E}(S(?),-)\rightarrow\mathbb{F}(?,T(-))$
is a monomorphism. 
\item If $S$ is $(1,T)$-extriangulated, then $\sigma:\mathbb{F}(?,T(-))\rightarrow\mathbb{E}(S(?),-)$
is a monomorphism. 
\end{enumerate}
\end{prop}

\begin{proof}
We only prove (a) since (b) follows with similar arguments. For this, it
is enough to show that $\tau:\mathbb{E}(S(D),C)\rightarrow\mathbb{F}(D,T(C))$
is a monomorphism for every $D\in\mathcal{D}$ and $C\in\mathcal{C}$.
Let $\eta\in\mathbb{E}(S(D),C)$ be such that $\tau(\eta)=0$. If $\s(\eta)=[\suc[C][E][SD][f][g]]$,
then $\t(\Gamma_{T}^{(S,1)}(\eta))=[\suc[TC][TE][TSD][Tf][Tg]]$.
Moreover, since $\tau(\eta)=0$, we have that $\t(\tau(\eta))=[TC\overset{\left[\begin{smallmatrix}1\\
0
\end{smallmatrix}\right]}{\rightarrow}TC\amalg D\overset{\left[\begin{smallmatrix}0 & 1\end{smallmatrix}\right]}{\rightarrow}D]$. Observe that $\tau(\eta)=\Gamma_{T}^{(S,1)}(\eta)\cdot\varphi_{D}$.
Hence, we have the following commutative diagram 
\[
\xymatrix{TC\ar[r]^{\left[\begin{smallmatrix}1\\
0
\end{smallmatrix}\right]}\ar@{=}[d] & TC\amalg D\ar[r]^{\left[\begin{smallmatrix}0 & 1\end{smallmatrix}\right]}\ar[d]^{\left[\begin{smallmatrix}x & y\end{smallmatrix}\right]} & D\ar[d]^{\varphi_{D}}\\
TC\ar[r]^{Tf} & TE\ar[r]^{Tg} & TSD.
}
\]
Note that $Tg\circ y=\varphi_{D}$ and thus $g\circ\psi_{E}\circ Sy=\psi_{SD}\circ STg\circ Sy=\psi_{SD}\circ S\varphi_{D}=1_{SD}$.
Therefore $\eta=0$ by \cite[Cor.3.5]{NP}. 
\end{proof}
We point out that the following result is an immediate consequence of the
above proposition and the definition of $\tau$ and $\sigma$. 

\begin{cor}\label{prop:mono-1} Let $(\mathcal{C},\mathbb{E},\s)$ and $(\mathcal{D},\mathbb{F},\t)$
be extriangulated categories, and let $(S:\D\to \C,T:\C\to D)$ be an adjoint pair. 
\begin{enumerate}
\item If $T$ is $(S,1)$-extriangulated, then $\Gamma_{T}^{(S,1)}:\mathbb{E}(S(?),-)\rightarrow\mathbb{F}(TS(?),T(-))$
is a monomorphism. 
\item If $S$ is $(1,T)$-extriangulated, then $\Gamma_{S}^{(1,T)}:\mathbb{F}(?,T(-))\rightarrow\mathbb{E}(S(?),ST(-))$
is a monomorphism. 
\end{enumerate}
\end{cor}

Note that, if $T$ is $(S,1)$-extriangulated and $S$ is $(1,T)$-extriangulated,
then $T\circ S$ is $(1,T)$-extriangulated and $S\circ T$ is $(S,1)$-extriangulated
(see Remark \ref{rem:funext}). We will use implicitly this fact in
the following statement. 

\begin{prop}\label{prop:0} Let $(\mathcal{C},\mathbb{E},\s)$ and $(\mathcal{D},\mathbb{F},\t)$
be extriangulated categories, and let $(S:\D\to \C,T:\C\to D)$ be an adjoint pair such that $T$ is $(S,1)$-extriangulated
and $S$ is $(1,T)$-extriangulated. 
\begin{enumerate}
\item If $\varphi:1\rightarrow T\circ S$ is $(1,T)$-extriangulated, then
$\tau:\mathbb{E}(S(?),-)\rightarrow\mathbb{F}(?,T(-))$ is an isomorphism
and $\tau^{-1}=\sigma$. 
\item If $\psi:S\circ T\rightarrow1$ is $(S,1)$-extriangulated, then $\sigma:\mathbb{F}(?,T(-))\rightarrow\mathbb{E}(S(?),-)$
is an isomorphism and $\sigma^{-1}=\tau$. 
\end{enumerate}
\end{prop}

\begin{proof}
We only prove (a) since (b) follows similarly. For this, note that 
\[
\tau\circ\sigma(\eta)=\Gamma_{T}^{(S,1)}(\psi_{A}\cdot\Gamma_{S}^{(1,T)}(\eta))\cdot\varphi_{Y}=T(\psi_{A})\cdot(\Gamma_{T}^{(S,1)}\Gamma_{S}^{(1,T)}(\eta))\cdot\varphi_{Y}
\]
for all $\eta\in\mathbb{F}(Y,TA)$. But $\Gamma_{T}^{(S,1)}\Gamma_{S}^{(1,T)}(\eta)=\Gamma_{T\circ S}^{(S,1)}(\eta).$
and thus 
\[
\tau\circ\sigma(\eta)=T(\psi_{A})\cdot\Gamma_{T\circ S}^{(S,1)}(\eta)\cdot\varphi_{Y}=T(\psi_{A})\cdot\varphi_{TA}\cdot\eta=\eta
\]
Hence, $\tau$ is a split-epi. Then, by Proposition \ref{prop:mono}  $\tau$ is an isomorphism.
\end{proof}

Assume that $(S,T)$ is an adjoint pair between the extriangulated categories
$(\mathcal{C},\mathbb{E},\s)$ and $(\mathcal{D},\mathbb{F},\t)$.
Under what conditions having $\varphi:1\rightarrow T\circ S$ $(1,T)$-extriangulated
implies that $\psi:S\circ T\rightarrow1$ is $(S,1)$-extriangulated
and vice versa? We will seek to answer this question in the following
lemma.

\begin{lem}\label{lem:phi} Let $(\mathcal{C},\mathbb{E},\s)$ and $(\mathcal{D},\mathbb{F},\t)$
be extriangulated categories, and let $(S:\D\to \C,T:\C\to\D)$ be an adjoint pair such that $T$ is $(S,1)$-extriangulated
and $S$ is $(1,T)$-extriangulated. 
\begin{enumerate}
\item If $\varphi:1\rightarrow T\circ S$ is $(1,T)$-extriangulated, then
$T\psi_{C}\cdot\Gamma_{TST}^{(S,1)}(\eta)\cdot\varphi_{TSD}=\Gamma_{T}^{(S,1)}(\eta)$
for all $\eta\in\mathbb{E}(SD,C)$. 
\item If $\psi:S\circ T\rightarrow1$ is $(S,1)$-extriangulated, then $\psi_{STC}\cdot\Gamma_{STS}^{(1,T)}(\eta)\cdot S\varphi_{D}=\Gamma_{S}^{(1,T)}(\eta)$
for all $\eta\in\mathbb{F}(D,TC)$. 
\item If $\varphi:1\rightarrow T\circ S$ is $(1,T)$-extriangulated, $T$
is extriangulated, and $TST\psi_{SD}\circ\varphi_{TSTSD}=1_{TSTSD}$
for all $D\in\mathcal{D}$, then $\psi:S\circ T\rightarrow1$ is $(S,1)$-extriangulated. 
\item If $\psi:S\circ T\rightarrow1$ is $(S,1)$-extriangulated, $S$ is
extriangulated, and $\psi_{STSTC}\circ STS\varphi_{TC}:=1_{STSTC}$
for every $C\in\mathcal{C}$, then $\varphi:1\rightarrow T\circ S$
is $(1,T)$-extriangulated. 
\end{enumerate}
\end{lem}

\begin{proof}
We only prove (a) and (c) since (b) and (d) follow similarly. For this,
consider $\eta\in\mathbb{E}(SD,C)$. 

(a) Since $\varphi:1\rightarrow T\circ S$ is $(1,T)$-extriangulated,
we have that 
\[
\Gamma_{TS}^{(1,T)}\Gamma_{T}^{(S,1)}(\eta)\cdot\varphi_{TSD}=\varphi_{:TC}\cdot\Gamma_{T}^{(S,1)}(\eta)
\]
and thus $T\psi_{C}\cdot\Gamma_{TS}^{(1,T)}\Gamma_{T}^{(S,1)}(\eta)\cdot\varphi_{TSD}=T\psi_{C}\cdot\varphi_{:TC}\cdot\Gamma_{T}^{(S,1)}(\eta)=\Gamma_{T}^{(S,1)}(\eta).$ 

(c) By (a) we have $T\psi_{C}\cdot\Gamma_{TST}(\eta\cdot\psi_{SD})\cdot\varphi_{TSTSD}=\Gamma_{T}(\eta\cdot\psi_{SD})$.
Then
\[
\Gamma_{T}(\psi_{C}\cdot\Gamma_{ST}\eta)=T\psi_{C}\cdot\Gamma_{TST}(\eta)\cdot TST\psi_{SD}\cdot\varphi_{TSTSD}=\Gamma_{T}(\eta\cdot\psi_{SD})
\]
since $\psi_{STSTC}\circ STS\varphi_{TC}:=1_{STSTC}$ by hypothesis.
Therefore $\psi_{C}\cdot\Gamma_{ST}\eta=\eta\cdot\psi_{SD}$ because
$\Gamma_{T}^{(S,1)}$ is monic by Corollary \ref{prop:mono-1}. Thus
 $\psi$ is $(S,1)$-extriangulated. 
\end{proof}
Note that the previous result holds true if we replace $(1,T)$-extriangulated
and $(S,1)$-extriangulated with extriangulated (see proposition below). 

\begin{lem}\label{lem:phi-1} Let $(\mathcal{C},\mathbb{E},\s)$
and $(\mathcal{D},\mathbb{F},\t)$ be extriangulated categories, and let
$(S:\D\to\C,T:\C\to\D)$ be an adjoint pair
such that $T$ is extriangulated and $S$ is extriangulated. 
\begin{enumerate}
\item If $\varphi:1\rightarrow T\circ S$ is extriangulated,
then $T\psi_{C}\cdot\Gamma_{TST}^{(S,1)}(\eta)\cdot\varphi_{TSD}=\Gamma_{T}^{(S,1)}(\eta)$
for all $\eta\in\mathbb{E}(SD,C)$. 
\item If $\psi:S\circ T\rightarrow1$ is extriangulated,
then $\psi_{STC}\cdot\Gamma_{STS}^{(1,T)}(\eta)\cdot S\varphi_{D}=\Gamma_{S}^{(1,T)}(\eta)$
for all $\eta\in\mathbb{F}(D,TC)$. 
\item If $\varphi:1\rightarrow T\circ S$ is extriangulated,
$T$ is extriangulated, and $TST\psi_{A}\circ\varphi_{TSTA}=1_{TSTA}$
for all $A\in\mathcal{C}$, then $\psi:S\circ T\rightarrow1$ is extriangulated. 
\item If $\psi:S\circ T\rightarrow1$ is extriangulated,
$S$ is extriangulated, and $\psi_{STSB}\circ STS\varphi_{B}:=1_{STSB}$
for every $B\in\mathcal{D}$, then $\varphi:1\rightarrow T\circ S$
is extriangulated. 
\end{enumerate}
\end{lem}

\begin{proof}
It follows from similar arguments as in the proof of Lemma \ref{lem:phi}.
\end{proof}

\begin{prop}\label{prop:1} Let $(\mathcal{C},\mathbb{E},\s)$ and $(\mathcal{D},\mathbb{F},\t)$ be extriangulated categories, and let $(S:\D\to\C,T:\C\to\D)$ be an adjoint pair. 
\begin{enumerate}
\item If $T$ is $(S,1)$-extriangulated and $\tau:\mathbb{E}(S(?),-)\xrightarrow{\sim}\mathbb{F}(?,T(-)),$ then there is a natural transformation 
$
\Lambda:\;\mathbb{F}(?,-)\rightarrow\mathbb{E}(S(?),S(-)).
$
Moreover, for $\eta\in\mathbb{F}(C,A)$, $\Lambda(\eta)$ satisfies
that $\Gamma_{T}^{(S,1)}\circ\Lambda(\eta)$ is an $\mathbb{F}$-extension
$\eta'\in\mathbb{F}(TSC,TSA)$ such that $\eta'\cdot\varphi_{C}=\varphi_{A}\cdot\eta$. 
\item If $S$ is $(1,T)$-extriangulated and $\sigma:\mathbb{F}(?,T(-))\xrightarrow{\sim}\mathbb{E}(S(?),-),$ then there is a natural transformation 
$
\Lambda:\;\mathbb{E}(?,-)\rightarrow\mathbb{F}(T(?),T(-)).
$
Moreover, for $\eta\in\mathbb{E}(C,A)$, $\Lambda(\eta)$ satisfies
that $\Gamma_{S}^{(1,T)}\circ\Lambda(\eta)$ is an $\mathbb{E}$-extension
$\eta'\in\mathbb{E}(STC,STA)$ such that $\eta\cdot\psi_{C}=\psi_{A}\cdot\eta'$. 
\end{enumerate}
\end{prop}

\begin{proof}
We only prove (a) since (b) follows similarly. Indeed, we can
consider $\tau^{-1}:\mathbb{F}(?,T(-))\rightarrow\mathbb{E}(S(?),-)$ and  $\mathbb{F}\cdot(1\times\varphi):\mathbb{F}(?,-)\rightarrow\mathbb{F}(?,TS(-)).$ Then, define $\Lambda$ as the composition 
\[
\Lambda\;:=\;(\tau^{-1}\cdot(1\times S))\circ(\mathbb{F}\cdot(1\times\varphi)):\;\mathbb{F}(?,-)\rightarrow\mathbb{E}(S(?),S(-))
\]
(see diagram below)\\
 \[
\begin{tikzpicture}[-,>=to,shorten >=1pt,auto,node distance=2.5cm,main node/.style=,x=2cm,y=-2cm]
\coordinate (U) at (0,0.5);
\coordinate (V) at (1,0);
\coordinate (U') at (0,-0.5);
\coordinate (V') at (1,1);
\node (1) at (0,0.5)  {$\mathcal{D}^{op} \times \mathcal{C} $};
\node (2) at (1,0)  {$\mathcal{D}^{op} \times \mathcal{D} $};
\node (2') at (1,1)  {$\mathcal{C}^{op} \times \mathcal{C} $};
\node (3) at (2,0.5)  {$\mathbf{Ab}$};
\node (5) at (-1,0)  {$\mathcal{D}^{op} \times \mathcal{D} $};
\node (A1) at (barycentric cs:U=1 ,U'=.5)  {$$};
\node (A2) at (barycentric cs:U=.5 ,U'=1)  {$$};
\node (B1) at (barycentric cs:V'=.5 ,V=1)  {$$};
\node (B2) at (barycentric cs:V'=1 ,V=.5)  {$$};
\draw[->, thin]  (1)  to [below left] node  {$\scriptstyle {S}^{op} \times 1 $} (2');
\draw[->, thin]  (1)  to [below right] node  {$\scriptstyle 1 \times T $} (2);
\draw[->, thin]  (2)  to  node  {$\scriptstyle \mathbb{F} $} (3);
\draw[->, thin]  (2')  to [below right] node  {$\scriptstyle \mathbb{E} $} (3);
\draw[->, thin]  (5)  to [below left] node  {$\scriptstyle 1\times S$} (1);
\draw[->, thin]  (5)  to [out=30,in=150] node  {$\scriptstyle 1 \times 1 $} (2);
\draw[->, thin]  (A2)  to node  {$\scriptstyle 1 \times \varphi $} (A1);
\draw[->, thin]  (B1)  to node  {$\scriptstyle \tau ^{-1}$} (B2);
\end{tikzpicture}
\]

Lastly, for $\eta\in\mathbb{F}(C,A)$, $\Lambda(\eta)=\tau^{-1}(\varphi_{A}\cdot\eta)$
and thus $\varphi_{A}\cdot\eta=\tau\circ\Lambda(\eta)=(\Gamma_{T}^{(S,1)}\circ\Lambda(\eta))\cdot\varphi_{C}$. 
\end{proof}

\begin{lem}\label{lem:adj} Let $(\mathcal{C},\mathbb{E},\s)$ and $(\mathcal{D},\mathbb{F},\t)$
be extriangulated categories, and let $(S:\D\to\C,T:\C\to\D)$ be an adjoint pair.
\begin{enumerate}
\item If $\mathcal{C}$ is WIC, $T$ is $(S,1)$-extriangulated and $\tau:\mathbb{E}(S(?),-)\xrightarrow{\sim}\mathbb{F}(?,T(-)),$
 then $S$ preserves inflations. 
\item If $\mathcal{D}$ is WIC, $S$ is $(1,T)$-extriangulated and $\sigma:\mathbb{F}(?,T(-))\xrightarrow{\sim}\mathbb{E}(S(?),-),$
then $T$ preserves deflations. 
\end{enumerate}
\end{lem}

\begin{proof}
We only prove (a) since (b) follows similarly.

Consider a $\t$-conflation $\eta:\:\suc[X][Y][Z][a][b]$ and let
$\suc[TSX][Y'][Z][f][g]$ be a realization of $\varphi_{X}\cdot\eta\in\mathbb{F}(Z,TSX)$.
Since $\tau$ is an isomorphism, there is an $\s$-conflation $\eta':\:\suc[SX][Y''][SZ][f'][g']$
such that $\varphi_{X}\cdot\eta=\tau(\eta')=\Gamma_{T}^{(S,1)}(\eta')\cdot\varphi_{Z}$.
Then, we obtain the following commutative diagram by applying $S$
to the realizations of the morphisms of $\t$-conflations obtained
by the above considerations
\[
\xymatrix{SX\ar[r]^{Sa}\ar[d]^{S\varphi_{X}} & SY\ar[r]^{Sb}\ar[d]^{S\varphi_{X}'} & SZ\ar@{=}[d]\\
STSX\ar[r]^{Sf}\ar@{=}[d] & SY'\ar[r]^{Sg}\ar[d]^{S\varphi'_{Z}} & SZ\ar[d]^{S\varphi_{Z}}\\
STSX\ar[r]^{STf'}\ar[d]^{\psi_{SX}} & STY''\ar[r]^{STg'}\ar[d]^{\psi_{Y''}} & STSZ\ar[d]^{\psi_{SZ}}\\
SX\ar[r]^{f'} & Y''\ar[r]^{g'} & SZ.
}
\]
Observe that $\psi_{Y''}\circ S\varphi'_{Z}\circ S\varphi'_{X}\circ Sa=f'\circ\psi_{SX}\circ S\varphi_{X}=f'$
and thus $Sa$ is an inflation since $f'$ is an inflation and $\mathcal{C}$
is left WIC.
\end{proof}

\subsection{Higher extension groups}

Let $(\mathcal{C},\mathbb{E},\s)$ be an extriangulated category and
$|\mathcal{C}|$ be a set of representatives of the iso-classes of objects
in $\mathcal{C}$. If $\mathcal{C}$ is essentially small, then one
can build the groups of higher extensions $\mathbb{E}^{n}(A,B)$ for
all $A,B\in\mathcal{C}$ and $n\geq2$ (see \cite[Sec.3]{G}, \cite[Sec.2]{BGLS},
or \cite[Sec.5.1]{AMP}). This construction is a generalization of
the gluing of short exact sequences that is done in exact categories.
In particular, by doing this we get a family of functors $\{\mathbb{E}^{n}(-,-):\mathcal{C}^{op}\times\mathcal{C}\rightarrow\Ab\}_{n\geq1}$
 satisfying the following properties.
 
\begin{enumerate}
\item [(A)] \cite[Sec.3.1]{G} $\;\mathbb{E}^{1}=\mathbb{E}.$ 
\item [(B)] For $\epsilon\in\mathbb{E}^{n}(A,B)$, $f:B\rightarrow B'$
and $g:A'\rightarrow A,$ we use the notation 
\[
f\cdot\epsilon:=\mathbb{E}^{n}(A,f)(\epsilon)\text{ and }\epsilon\cdot g:=\mathbb{E}^{n}(g^{op},B)(\epsilon).
\]
\item [(C)] There is an epimorphism 
$
\phi_{i,j}:\coprod_{C\in|\mathcal{C}|}\mathbb{E}^{i}(C,B)\times\mathbb{E}^{j}(A,C)\rightarrow\mathbb{E}^{i+j}(A,B)
$
with kernel equal to the abelian group generated by 
\[
\left\{ (\epsilon,f\cdot\epsilon')-(\epsilon\cdot f,\epsilon')\,|\:\epsilon\in\mathbb{E}^{i}(C,B),\epsilon'\in\mathbb{E}^{j}(A,C'),f\in\Hom_{\mathcal{C}}(C',C)\right\} .
\]
In particular, $\phi_{i,j}(\epsilon,f\cdot\epsilon')=\phi_{i,j}(\epsilon\cdot f,\epsilon')$
for all $\epsilon\in\mathbb{E}^{i}(C,B),$ $\epsilon'\in\mathbb{E}^{j}(A,C'),$
$f\in\Hom_{\mathcal{C}}(C',C).$

\item [(D)] \cite[Lem.3.8]{G} For $\epsilon_{1}\in\mathbb{E}^{i}(C_{1},B),\epsilon'_{1}\in\mathbb{E}^{j}(A,C_{1}),\epsilon_{2}\in\mathbb{E}^{i}(C_{2},B),\epsilon'_{2}\in\mathbb{E}^{j}(A,C_{2})$,
we have that 
$
\phi_{i,j}\left((\epsilon_{1},\epsilon'_{1})+(\epsilon_{2},\epsilon'_{2})\right)=\phi_{i.j}(\epsilon,\epsilon'),
$
where $\epsilon$ and $\epsilon'$ are defined from
 $C_i\xrightarrow{\mu_i}C_1\coprod C_2\xrightarrow{\pi_i}C_i$
as follows: $\epsilon:=\epsilon_{1}\cdot\pi_{1}+\epsilon_{2}\cdot\pi_{2}$
and $\epsilon':=\mu_{1}\cdot\epsilon_{1}'+\mu_{2}\cdot\epsilon_{2}'.$
 
\item [(E)] For $\eta\in\mathbb{E}^{i+j}(A,B)$ there
is $(\epsilon,\epsilon')\in\coprod_{C\in|\mathcal{C}|}\mathbb{E}^{i}(C,B)\times\mathbb{E}^{j}(A,C)$
such that $\eta=\phi_{i,j}(\epsilon,\epsilon')$. In this case, we
will use the notation $\eta=\epsilon\cdot\epsilon'$. In \cite{BGLS},
this is denoted as $\eta=\epsilon\smile\epsilon'$ and is called \emph{cup
product} (see \cite[Sec.2.4(2.4.3)]{BGLS}).

\item [(F)] To sum up, $\mathbb{E}^{i+j}(A,B)$ is an abelian group such
that its elements can be expressed as $\eta=\epsilon\cdot\epsilon'$
with $\epsilon\in\mathbb{E}^{i}(C,B)$ and $\epsilon'\in\mathbb{E}^{j}(A,C)$
for some $C\in\mathcal{C}$. Furthermore $\epsilon\cdot(f\cdot\epsilon')=(\epsilon\cdot f)\cdot\epsilon'$
for every morphism $f$. 

\item [(G)] Every extriangulated functor $F:(\mathcal{C},\mathbb{E},\s)\rightarrow(\mathcal{D},\mathbb{F},\t)$
induces a family of natural transformations 
$
\left\{ \Gamma_{F}^{k}:\mathbb{E}^{k}(-,?)\rightarrow\mathbb{F}^{k}(F(-),F(?))\right\} _{k\geq1}
$
 such that $\Gamma_{F}^{i}(\epsilon)\cdot\Gamma_{F}^{j}(\epsilon')=\Gamma_{F}^{i+j}(\eta)$
if $\eta=\epsilon\cdot\epsilon'$ (see \cite[Prop.3.6]{BGLS}). Note
that, for another extriangulated functor $G:(\mathcal{D},\mathbb{F},\t)\rightarrow(\mathcal{E},\mathbb{G},\t')$,
$\Gamma_{G}^k\circ\Gamma_{F}^k=\Gamma_{GF} ^k$  $\forall k\geq1$. 

\item [(H)] It follows from (E) that, if $\alpha:F\rightarrow F'$
is an extriangulated natural transformation, then $\alpha_{B}\cdot\Gamma_{F}^{n}(\eta)=\Gamma_{F'}^{n}(\eta)\cdot\alpha_{A}$
for all $\eta\in\mathbb{E}(A,B)$ and all $n>0$. 
\end{enumerate}
The following is a modified version of \cite[Cor.3.10]{BGLS}. We
include a proof for the sake of completeness.

\begin{prop}\label{prop:higher} Let $(\mathcal{C},\mathbb{E},\s)$ and $(\mathcal{D},\mathbb{F},\t)$
be extriangulated categories, and let $(S:\D\to \C,T:\C\to\D)$ be an adjoint pair such that $T$ and $S$ are extriangulated.
Consider the natural transformations $\tau^{n}:\mathbb{E}^{n}(S(?),-)\rightarrow\mathbb{F}^{n}(?,T(-))$
and $\sigma^{n}:\mathbb{F}^{n}(?,T(-))\rightarrow\mathbb{E}^{n}(S(?),-)$
induced by the maps
\begin{alignat*}{1}
\mathbb{E}^{n}(S(D),C) & \stackrel{\tau_{D,C}^{n}}{\rightarrow}\mathbb{F}^{n}(D,T(C))\text{, }\eta\mapsto\Gamma_{T}^{n}(\eta)\cdot\varphi_{D},\text{ and }\\
\mathbb{F}^{n}(D,T(C)) & \stackrel{\sigma_{D,C}^{n}}{\rightarrow}\mathbb{E}^{n}(S(D),C)\text{, }\eta\mapsto\psi_{C}\cdot\Gamma_{S}^{n}(\eta).
\end{alignat*}
Then, the following statements hold true. 
\begin{enumerate}
\item If $\varphi:1\rightarrow T\circ S$ is extriangulated, then $\sigma^{n}$
is a split-mono and $\tau^{n}$ is a split-epi.
\item If $\psi:S\circ T\rightarrow1$ is extriangulated, then $\tau^{n}$
is a split-mono and $\sigma^{n}$ is a split-epi.
\item If $\varphi:1\rightarrow T\circ S$ and $\psi:S\circ T\rightarrow1$
are extriangulated, then $\tau^{n}$ is an isomorphism whose inverse is $\sigma^{n}.$
\end{enumerate}
\end{prop}

\begin{proof}
We only prove (b) since (a) follows by duality and (c) follows from
(a) and (b). 

Let $\eta\in\mathbb{E}^{n}(S(D),C)$. By item (E), there are $\epsilon\in\mathbb{E}^{n-1}(S(D),E)$
and $\epsilon'\in\mathbb{E}(E,C)$ such that $\eta=\epsilon'\cdot\epsilon$
and thus 
\begin{alignat*}{1}
\sigma_{D,C}^{n}\circ\tau_{D,C}^{n}(\eta) & =\sigma^{n}\left(\Gamma_{T}^{n}(\eta)\cdot\varphi_{D}\right)\\
 & =\psi_{C}\cdot\Gamma_{S}^{n}\Gamma_{T}^{n}(\eta)\cdot S\varphi_{D}\\
 & =\psi_{C}\cdot\Gamma_{ST}^{n}(\epsilon'\cdot\epsilon)\cdot S\varphi_{D}\\
 & =\psi_{C}\cdot\Gamma_{ST}(\epsilon')\cdot\Gamma_{ST}^{n-1}(\epsilon)\cdot S\varphi_{D}\\
 & =\epsilon'\cdot\psi_{E}\cdot\Gamma_{ST}^{n-1}(\epsilon)\cdot S\varphi_{D}\\
 & =\epsilon'\cdot\epsilon\cdot\psi_{SD}\cdot S\varphi_{D}\\
 & =\epsilon'\cdot\epsilon=\eta.
\end{alignat*}
Therefore $\tau_{D,C}^{n}$ is a split-mono and $\sigma_{D,C}^{n}$
is a split-epi. 
\end{proof}

{\footnotesize{}{}\vskip3mm\quad\; Alejandro Argud{\'i}n-Monroy}\\
 {\footnotesize{}{} Instituto de Matem{\'a}ticas, Facultad de Ciencias,}\\
 {\footnotesize{}{} Universidad Nacional Aut{\'o}noma de M{\'e}xico,}\\
 {\footnotesize{}{} Circuito Exterior, Ciudad Universitaria,}\\
 {\footnotesize{}{} CDMX 04510, M{\'E}XICO.}\\
 {\footnotesize{}{} }\texttt{\footnotesize{}{}argudin@ciencias.unam.mx}{\footnotesize\par}

\noindent {\footnotesize{}\noindent {}\vskip3mm Octavio Mendoza
Hern{\'a}ndez}\\
 {\footnotesize{}{} Instituto de Matem{\'a}ticas}\\
 {\footnotesize{}{} Universidad Nacional Aut{\'o}noma de M{\'e}xico,}\\
 {\footnotesize{}{} Circuito Exterior, Ciudad Universitaria,}\\
 {\footnotesize{}{} CDMX 04510, M{\'E}XICO.}\\
 {\footnotesize{}{} }\texttt{\footnotesize{}{}omendoza@matem.unam.mx}{\footnotesize\par}

\noindent {\footnotesize{}\noindent {}\vskip3mm Carlos E. Parra}\\
 {\footnotesize{}{} Instituto de Ciencias F{\'i}sicas y Matem{\'a}ticas}\\
 {\footnotesize{}{} Edificio Emilio Pugin, Campus Isla Teja}\\
 {\footnotesize{}{} Universidad Austral de Chile}\\
 {\footnotesize{}{} 5090000 Valdivia, CHILE}\\
 {\footnotesize{}{} }\texttt{\footnotesize{}{}carlos.parra@uach.cl}{\footnotesize\par}

\end{document}